\documentclass[12pt]{article}
\usepackage{float}
\usepackage[top=20mm,left=20mm,right=20mm,bottom=20mm]{geometry}
\usepackage{amsmath,amssymb}
\usepackage{makeidx}
\usepackage{amsfonts}
\usepackage{mathtools}
\usepackage[ansinew]{inputenc}
\usepackage[usenames,dvipsnames]{pstricks}
\usepackage{subcaption}
\usepackage{epsfig}
\usepackage{pst-grad} 
\usepackage{pst-plot} 
\usepackage{cite}

\usepackage{amsthm}
\newtheorem{theorem}{Theorem}

\newtheorem{lemma}{Lemma}
\newtheorem{proposition}[theorem]{Proposition}

\newcommand\del{\partial}
\newcommand\ope{\mathrm{e}}
\newcommand\opd{\mathrm{d}}
\newcommand\opi{\mathrm{i}}
\newcommand\oppi{\mathrm{\pi}}
\newcommand\dummy[1]{}

\def\vec#1{\mbox{\boldmath{\protect{\begin{math}#1\end{math}}}}}

\makeindex

\begin{document}
\begin{center}
  \setlength\baselineskip{100pt}
  \LARGE
  Traveling pulses with oscillatory tails,
  figure-eight-like stack of isolas, and dynamics in heterogeneous media
\end{center}

\begin{center}
  Yasumasa Nishiura$^1$ and Takeshi Watanabe$^2$
  
  \bigskip

  {\small
    $^1$
    Research Center of Mathematics for Social Creativity,
    Research Institute for Electronic Science,
    Hokkaido University, Sapporo, 060-0812, Japan
    
    E-mail: {\tt yasumasa@pp.iij4u.or.jp}
    
    \medskip
    
    $^2$
    Organization for Regional Collaborative Research and Development,
    Suwa University of Science,
    Chino, Nagano, 391-0292, Japan
  }
  
\end{center}

\bibliographystyle{unsrt}

\begin{abstract}
  The interplay between 1D traveling pulses with oscillatory tails (TPO) and heterogeneities of bump type is studied for a generalized three-component FitzHugh-Nagumo equation. First, we present that stationary pulses with oscillatory tails (SPO) form a ``snaky'' structure in homogeneous spaces, after which TPO branches form a ``figure-eight-like stack of isolas'' located adjacent to the snaky structure of SPO. Herein, we adopted input resources such as voltage-difference as a bifurcation parameter. A drift bifurcation from SPO to TPO was observed by introducing another parameter at which these two solution sheets merged. In contrast to the monotone tail case, a nonlocal interaction appeared in the heterogeneous problem between the TPO and heterogeneity, which created infinitely many saddle solutions and finitely many stable stationary solutions distributed across the entire line. The response of TPO displayed a variety of dynamics including pinning and depinning processes along with penetration (PEN) and rebound (REB). The stable/unstable manifolds of these saddles interacted with the TPO in a complex manner, which created a subtle dependence on the initial condition, and a difficulty to predict the behavior after collision even in 1D space. Nevertheless, for the 1D case, a systematic global exploration of solution branches induced by heterogeneities (heterogeneity-induced-ordered patterns; HIOP) revealed that HIOP contained all the asymptotic states after collision to predict the solution results without solving the PDEs. The reduction method of finite-dimensional ODE (ordinary differential equation) system allowed us to clarify the detailed mechanism of the transitions from PEN to pinning and pinning to REB based on a dynamical system perspective. Consequently, the basin boundary between two distinct outputs against the heterogeneities yielded an infinitely many successive reconnection of heteroclinic orbits among those saddles, because the strength of heterogeneity increased and caused the aforementioned subtle dependence of initial condition. 
  
\end{abstract}

\paragraph{keywords}
three-component FitzHugh--Nagumo equation,
traveling pulses with oscillatory tails, snaky pattern, center manifold reduction,
multi-time scale analysis, basin boundary

\section{Introduction}

Spatially localized structures are common in various pattern-forming systems, appearing in nerve impulses \cite{HodgkinHuxley1952}, chemical reactions \cite{PhysRevLett.66.3083,PhysRevLett.69.945,MIKHAILOV200679,doi:10.1063/1.2752494,doi:10.1021/jp308813m}, fluid mechanics \cite{PhysRevA.44.6466,batiste_knobloch_alonso_mercader_2006,PhysRevLett.94.184503,Watanabe2012}, granular materials \cite{oscilon_Swinney}, buckling \cite{Hunt2000}, morphogenesis \cite{COTTERELL2015257,Baker2006}, vegetation \cite{HilleRisLambersetal2001,PhysRevE.91.022924}, and numerous self-organized systems \cite{Yochelis_2008,PhysRevE.71.065301,PhysRevE.90.032923}. For more details, refer reviews and books \cite{Meinhardt_1995,Liehr2013,Purwins2005,NishiuraAMS2002,doi:10.1146/annurev-conmatphys-031214-014514}. One of the intriguing dynamics for localized patterns is that they display self-replication and self-destruction such as that of a living cell \cite{Pearson}, \cite{NU_PhysicaD_1999}, including spatiotemporal chaos \cite{NU_PhysicaD_2001} as a combination of replication and destruction. 

In this study, we are specifically interested in the moving localized patterns such as traveling pulses and spots. As a model system, a class of three-component systems can be deemed appropriate, because it supports the ``coexistence'' of stable localized traveling patterns that are independent of spatial dimensions unlike two-component systems to yield a natural setting for consideration of the interacting dynamics among those patterns and heterogeneities. In particular, two representative model systems are popular in this context. First, the one-activator and two-inhibitor system observed in gas-discharge phenomenon \cite{PhysRevLett.78.3781,Bodeetal2002}, and second, an extension of activator-depletion system \cite{Meinhardt_1995} originally proposed for shell patterns. Herein, we employed the model system (\ref{eq2.01}) called the generalized FitzHugh--Nagumo equations that pertains to the first category. Certain extensive studies have reported the existence and stability of 1D pulse with monotone tails \cite{DVK,vHDK,vHDK,NSJDE2021}.
In case the localized patterns are in motion, collisions among them are unavoidable and result in varying dynamics, including annihilation, coalescence, and splitting. Therefore, to understand such dynamics with large deformation, the notion of ``scattor'' is beneficial as illustrated by \cite{NishiuraTeramotoUeda2003,PhysRevE.67.056210,doi:10.1063/1.2087127,IIMA2009449,Watanabe2012}. 

Heterogeneities or defects in the media are extremely common \cite{PhysRevLett.91.178101,doi:10.1063/1.1450565,doi:10.1137/S0036144599364296} and influence the dynamics of moving patterns through pinning (or blocking), rebound, splitting, and annihilation depending on the strength of heterogeneity
\cite{doi:10.1137/0149030,doi:10.1137/S0036139998349298,Nishiuraetal2007,PhysRevE.75.036220,PhysRevE.79.046205, Information_Exchange,NishiuraTeramotoYuan2012,van_Heijster2019,Peter_van_Heijster_and_Teramoto_2020,Xie.defect,doi:10.1142/9789812708687_0011,PhysRevE.83.021916,PhysRevE.76.016205,PhysRevE.95.062209}. Remarkably, even
chaotic motion is observed in periodic media in case the heterogeneity period is comparable
to the size of the traveling pulse\cite{Nishiuraetal2007}. Interestingly, inhomogeneity itself can act as a generator of localized patterns \cite{PRAT2005177,PhysRevE.76.066201,doi:10.1137/13091261X}.


The interaction with heterogeneities starts from the tail portion of the pattern, and there are two types of tails: monotone and oscillatory ones.
For the monotone case, existence, stability, and interaction have been studied extensively so a majority of the studies for heterogeneous problems are concentrated on the monotone tail
case\cite{Ei2002TheMO,Nishiuraetal2007,PhysRevE.75.036220,PhysRevE.79.046205,NishiuraTeramotoYuan2012,NishiuraTeramotoUeda2003,PhysRevE.67.056210,doi:10.1063/1.2087127,PhysRevE.80.046208,PhysRevE.69.056224}. 
On the contrary, the interactive manner among these patterns is much richer and complicated for the oscillatory case owing to the attractive and repulsive forces appearing alternately at the tail, which enables the formation of crystal-like structures and
binary star rotations \cite{Liehr2013,Bodeetal2002}.
However, only a few extensive studies have reported the existence and stability for such a class (e.g., \cite{doi:10.1137/140999177} for FitzHugh--Nagumo equations), which still remains as a fertile ground for further understanding.

Therefore, a considerable variation occurs between the monotone and oscillatory tails in the heterogeneous problems.
The varying behaviors between the two types of traveling spots in case of interacting with the heterogeneity of the slit type are presented in Fig.\ref{slit.eps}. The trajectory is not straight and appears almost unpredictable for oscillatory case (Fig.\ref{slit.eps}(a)), which creates a stark contrast with monotone one (Fig.\ref{slit.eps}(b)). The dynamics in Fig.\ref{slit.eps} resembles that of droplets bouncing on a liquid surface, as performed by Couder et al. \cite{PhysRevE.88.011001,PhysRevLett.97.154101}. In addition, it manifests a macroscopic quantum-like behavior similar to that observed in the Young's slit experiment. 
This study primarily aims to at least partially answer the following natural question ``what is the origin of such a subtle behavior?''. Relatively, one can imagine that such a sensitive behavior arises from the distribution of several unstable equilibria of saddle type proximate of the slit, and the orbit bounces back multiple times as it approaches the slit. We believe that this is at least qualitatively true, but to pursue the details of the dynamics, we initially need to understand the 1D case. In particular, a remarkable aspect is that several key features arising from the oscillatory tails are inherited to the 1D case for both the homogeneous and heterogeneous problems, and therefore, we focused on the 1D problem in this study. The sensitive dependence of the initial conditions of 1D traveling pulses for the three distinct heights of the bump is illustrated in Fig.\ref{transition.eps}. Moreover, the response varied upon collision with the bump from the first to the second encounter, regardless of the almost identical profiles of the pulses immediately before collisions. Another example is presented in Fig.\ref{spiral.eps}, wherein the initial data are similar to an unstable stationary pulse in the vicinity of the bump. The separator or basin boundary between the penetration and rebound in (position, velocity)--space appears to spiral into the location of the stationary pulse. Therefore, the outcome cannot be conveniently predicted in case the initial data are similar. This specific case is clarified in Section 5 by characterizing the basin boundary for the reduced finite-dimensional system, which focuses on more complex problems in higher dimensional space from the perspective of a dynamical system. Hereinafter, the abbreviated notations PEN, REB, OSC, and STA are used to represent the response of TPO to heterogeneity, namely, penetration, rebound, oscillatory pinning, and stationary pinning, respectively.

\begin{figure*}
  \begin{center}
    \includegraphics[width=.9\hsize]{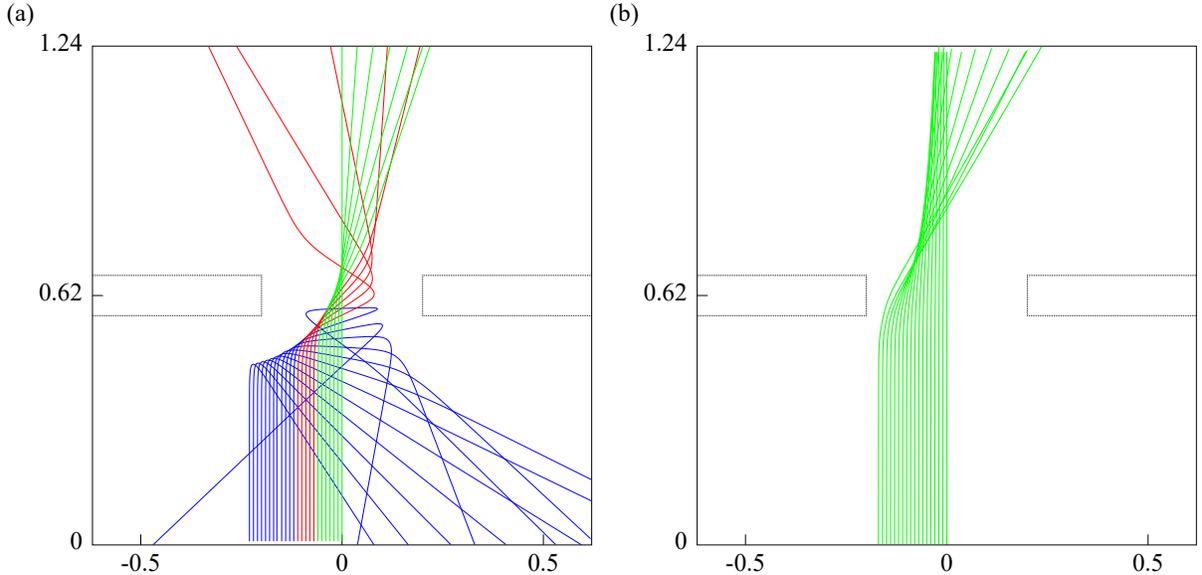}
    \caption{Comparison between monotone and oscillatory tails of traveling spots for model system (\ref{eq2.01}). (a) Oscillatory case: sensitive toward initial condition, specifically, the orbit abruptly alters direction as it approaches slit, regardless of the continuous modifications in initial conditions. Height $\varepsilon$ of slit is equal to 0.1. (b) Monotone case. Smooth orbit behavior with no sudden variations in direction. Height is denoted by $\varepsilon = 0.14$. Parameter values are as follows. 
      (a): $(Du, Dv, Dw)=(2.5\times10^{-4}, 0, 9.64\times10^{-4}),
      \kappa_{1}=-0.1,\kappa_{2}=2.5,\kappa_{3}=0.3,\kappa_{4}=1.0, (\tau, \theta)=(3.35, 0)$. 
      (b): $(Du, Dv, Dw)=(0.9\times10^{-4}, 1.0\times10^{-3}, 1.0\times10^{-2}), 
      \kappa_{1}=-7.3,\kappa_{2}=2.0,\kappa_{3}=1.0,\kappa_{4}=8.5, (\tau, \theta)=(40.0, 0)$.
    }
    \label{slit.eps}
  \end{center}
\end{figure*}

\begin{figure*}
  \begin{center}
    \includegraphics[width=.8\hsize]{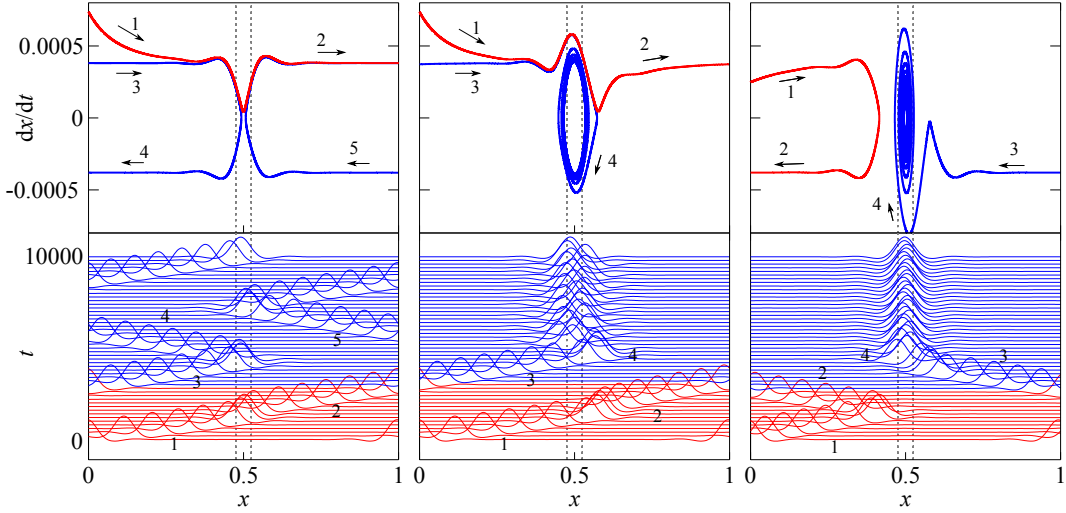}
    \caption{Subtle behaviors of TPO in vicinity to transition points. Three PDE simulations on circle are presented (bump height :
      $\varepsilon=-0.0004875$, $0.00079375$, $0.0024375$,
      from left to right). The above-mentioned values of bump height are almost equal to the values of transition PEN--REB, PEN--OSC, and REB--STA, respectively. The upper panel displays the location and velocity corresponding to the traveling pulse depicted in the lower panel. In addition, the red and blue portions of the orbit display the first and second collisions, respectively. Although the profile and velocity of the red pulse were similar to those of the blue pulse when they were almost colliding with the bump, the outcomes of the first and second collisions varied. As discussed in Section 5, the pulse orbit coincided with the basin boundary at the critical bump height in a reduced finite-dimensional system, which was indicated by the left stable manifold of the saddle point located at the center of bump for the PEN--REB case (left). At this instant, the basin of attraction in the phase space was categorized into two parts by the pulse orbit itself. Therefore, in case the bump height is proximate to the critical bump height and the initial profile is almost identical to that of the TPO, a slight variation in the initial setting causes a drastic variation in the outcome after collision. Moreover, similar basin structures were observed in other cases as well, as detailed in Section 5.
    }\label{transition.eps}
  \end{center}
\end{figure*}

\begin{figure*}
  \begin{center}
    \includegraphics[width=.7\hsize]{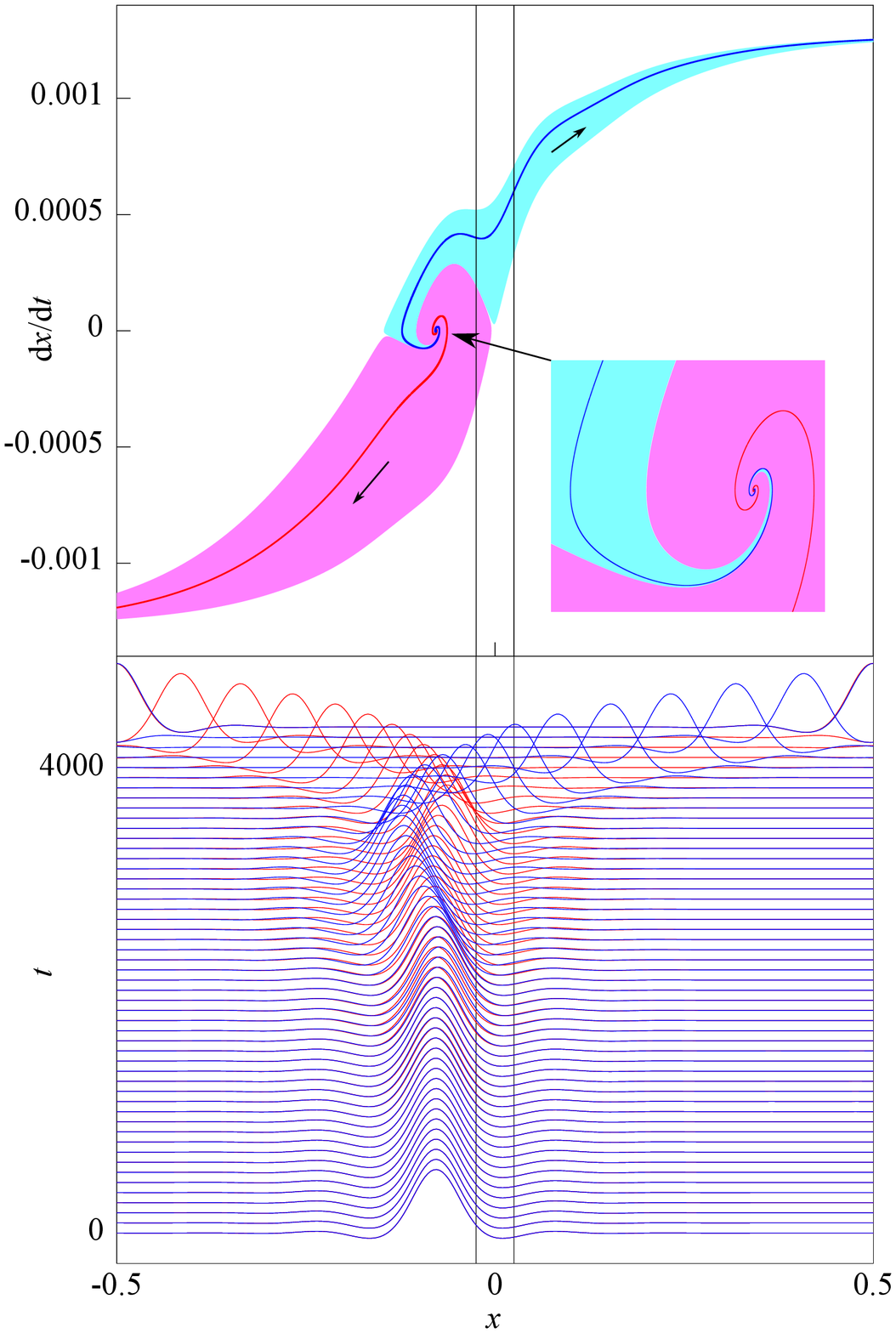}
    \caption{Sensitive dependence on initial condition. Basin boundary between PEN and REB spirals into an unstable stationary pulse located near the bump. The bump height was selected as $\varepsilon = -0.00005$ belonging to the inside of PEN regime in phase diagram presented in Fig.\ref{phase_PDE.eps}. An unstable stationary pulse was stuck near the bump, which is one of the HIOPs. We selected two initial profiles that were almost identical with a difference of less than $10^{-8}$. Blue (red) portion goes across (rebound) the bump, implying the existence of basin boundary spiraling around stuck pulse. In Section 5, we discuss that such a basin boundary exists in reduced ODE systems.}
    \label{spiral.eps}
  \end{center}
\end{figure*}

As a necessary prerequisite for the 1D heterogeneous problem, the existence and stability properties of TPO should be considered in homogeneous problem (i.e., no heterogeneities). The numerically global parametric dependency of stationary and traveling pulses are detailed in Section 3, wherein a snaky structure of the stationary pulses was observed along the parameter $\kappa_{1}$ that is proportional to the external input such as voltage difference. In addition, a typical snaky feature was observed with an increased number of peaks via saddle-node bifurcations along the snaky branch, and the ladder patterns appeared via symmetry-breaking bifurcations from the primary snaky structure, which appeared in several model equations such as the Swift--Hohenberg equation, Ginzburg--Landau equation, and reaction--diffusion equations
\cite{WOODS1999147,PhysRevFluids.2.064401,Kozyreff,Ponedeletal2016,MA20101867,Burke2007,doi:10.1063/1.2746816,doi:10.1137/06067794X,doi:10.1063/1.4792711,BBKM_Eckhaus_snaking,Beck2009,KY_foliated_snaking_2020,ACV_localized_unifying_framework_2021}

Such a snaky structure signifies several aspects, for instance, the origin of the localized stationary pulse. More specifically, the ``localization'' occurred through ``beat'' bifurcation (described in Section 3). Here the beat bifurcation indicated that two wave trains with nearby
wave-numbers interacted, and the resulting wave exhibited a beat in its
amplitude before its envelope eventually became a localized one. 
As for the traveling pulses, the set of traveling pulses emerged as a figure-eight-like stack of isolas located proximate to the snaky branch of the stationary pulses. As depicted in Fig.\ref{bif3.eps}, they appeared as an imperfection of snaky structure of stationary pulses. Interestingly, a stationary pulse solution was obtained in case the isola was traced along another parameter such as $\tau$, i.e., a drift bifurcation occurred as depicted in Fig.\ref{bif3.eps}. Moreover, the traveling pulse sheet was attached to the wall of the stationary pulse in two-parameter space $(\kappa_{1}, \tau)$. Therefore, such a combined set of sheets can be termed as the {\it gutter-like structure} comprising snaky stationary branches and isolas of traveling pulses. Thus, the shape of the TPO resulted from the localization process via beat bifurcation and traveling motion from the drift bifurcation. Note that the stable portion of the branch of traveling pulses was restricted to a subinterval of $\kappa_{1}$ owing to the snaky structure.

In this study, all numerical computations were performed on a circle, i.e., a finite interval with periodic boundary conditions. Nevertheless, most of the qualitative aspects of the original problem were preserved, because the system size was adequately large in comparison to that of the spatially localized patterns. This is partially justified by reducing the PDE dynamics to a finite-dimensional ODE system on the infinite line, and thereafter, comparing its dynamics with the PDE simulations and HIOP structure on a circle (as detailed in Section 5). Therefore, we used the same terminology of TPO for the traveling pulses on both the circle and infinite line, unless there was no confusion.

Upon introducing a heterogeneity into the media, the interaction between the traveling pulses and heterogeneity produced various interesting dynamics such as rebound, pinning, and even chaotic behavior (\cite{PhysRevE.75.036220,Nishiuraetal2007,PhysRevE.79.046205}). 
This study aimed to clarify the dynamic interplay between the TPO and a heterogeneity of bump type in the media, as illustrated in Fig.\ref{hetero.eps}, and explain the influence of the heterogeneity strength $\varepsilon$ on the outputs (Fig.\ref{phase_PDE.eps}). Note that the bump constitutes as the basic unit in the heterogeneity class, and various heterogeneities can be constructed via combinations of them.

In this study, we developed a two-fold strategy: the heterogeneity-induced ordered patterns (HIOP) approach and a reduction method to a finite-dimensional system. Originally, the HIOP is an object across the entire line that is defined by the set of all solution branches of the ordered patterns caused by heterogeneity, such as stationary pulses pinned to the heterogeneity. Herein, we adopted $\varepsilon$ as a bifurcation parameter. However, we considered the HIOP on a circle and attempted to numerically determine all the relevant solution branches in the heterogeneous space. In particular, we focused on the stationary and time-periodic solutions. As such, this solution branch can be determined following two methods. First, search a connected component of bifurcating branches emanating from the trivial background state with $\varepsilon$ as the bifurcating parameter. Second, a continuation method starting from a solution to the homogeneous problem as a seed. Consequently, these two approaches are complementary to each other, and the two resulting solution sets exhibited an intersection, implying that any solution in the component is connected to the trivial constant state. As the translation invariance only holds for the homogeneous case, a selection rule must be followed to select the limiting profile of the heterogeneous problem as $\varepsilon$ tends to zero. In particular, this was caused by the interaction between the oscillatory tails and heterogeneity, and numerous solutions were selected from a continuum of translated results. Note that various solution branches of the heterogeneous problem may exhibit the same limiting profile in the modulo translation as the height approaches zero. Moreover, we identifed the solutions of the HIOP modulo translation in case $\varepsilon$ was equal to zero. Thus, countably many stack of limiting solutions was observed, such as the {\it spoke} depicted in Fig.\ref{all.eps} (b).

As for the periodic solutions on a circle, the TPO in homogeneous space constitutes as a basic time-periodic solution that can be extended to a periodic solution to heterogeneous problem for small bump heights. Similarly, rebound solutions can be conveniently captured as a member of periodic solutions on $\mathbb S^1$. Interestingly, the period diverged for specific values of $\varepsilon$ in which homoclinic or heteroclinic bifurcation occurred, and the associated saddle point corresponding to one of the shifted stationary solutions stated above could be detected. Thus, the homoclinic or heteroclinic bifurcation is a {\it junction} point between the periodic solutions and shifted stationary solutions. Based on these observations, all the HIOP solutions presumably originated in the trivial constant solution. The length of the circle was not inadequately short in comparison to the sizes of the pulse and bump, which aided in retaining the property of the HIOP on the circle to qualitatively reflect the global structure of the solution branches on the infinite line. This attribute can be justified at least partially following the reduction method explained hereinafter.


The second approach, i.e., reduction to a finite dimensional system in vicinity to the drift bifurcation point, is a powerful tool for understanding the pulse dynamics in heterogeneous media \cite{Nishiuraetal2007,PhysRevE.79.046205,NishiuraTeramotoYuan2012,Nishiura_Teramoto_2012,Liehr2013}. This reduction is possible for the infinite line, and the initial data associated with the TPO were established at $\pm\infty$.
These two methods are complementary to each other and allow us to understand an entire dynamical structure. A strong benefit of knowing the global structure of HIOP is that we can determine a complete list of candidates for the asymptotic state of the TPO after collision even for large bump heights, as the HIOP approach is unrestricted in terms of the parameter being proximate to the drift bifurcation point. Moreover, the reduction method enables us to comprehend the detailed mechanism of transition, e.g., the transition from PEN to REB occurs in a dynamical system perspective. In addition, the behavior of the basin boundary between the two outputs in the neighborhood of the critical height can be clarified as well. As such, the oscillatory nature of the interaction between the tails and bump becomes visible in the behavior of the basin boundary as infinitely many reconnections appear among the critical points with the variation in height. 
Furthermore, the reduction method can approximately evaluate the dynamics on the entire line in comparison to that on a circle. More importantly, it aids us to identify the objects on a circle (PDE simulations and HIOP) with the corresponding objects on the infinite line, even though the validity of the reduction method is assured only within the neighborhood of the drift bifurcation point. 

Although the behavior of the TPO was in tandem with that of TPO situated far from the bump, it did not exist in the heterogeneous media on a circle. Therefore, in case of using it as an initial condition, it signified an approximated function of the TPO in a homogeneous problem located far from the heterogeneity. Evidently, such a TPO-like initial data cannot be uniquely determined and depends on the approximation method as well as the length of the circle. More precisely, such an ambiguity creates challenges in predicting the behavior and sensitivity to the initial condition, as depicted in Fig.\ref{transition.eps} and Fig.\ref{spiral.eps}. Thus, we ought to resolve all these ambiguities by identifying the basin boundaries of the outputs after collision in the reduced ODE system, as discussed in Section 5. In particular, the reduced system was defined on the infinite line and the associated true TPO was defined as a solution living at $\pm\infty$. The basin boundary emerged as a heteroclinic orbit connecting the two critical points and underwent reconnections among the critical points distributed on the infinite line, as discussed in Section 5. However, we proceeded with this informal definition of the TPO on a circle to avoid any confusion. Overall, the asymptotic state of the PDE dynamics starting from the TPO on a circle with heterogeneity can be characterized as follows.


\begin{proposition}
  The solution set HIOP on a circle contains all the asymptotic states of the TPOs after colliding with the heterogeneity of the bump type.  
\end{proposition}


Thus, the HIOP contains the $\Omega$-limit set of TPO on a circle. As discussed earlier, the HIOP contains a set of solutions in homogeneous space within the limits of $\varepsilon \rightarrow 0$. Hereinafter, we identified two solutions of the HIOP in terms of correspondence via translation. Relatively, the solution structure of the HIOP at $\varepsilon = 0$ was highly degenerated but quantized as such that only the countably many solutions could emanate from the continuum solutions toward the homogeneous problem (as detailed in Section 4). 

Although the HIOP prepares all the options for the asymptotic state for a given height $\varepsilon$, it does not indicate the presence of a unique attractor for a given $\varepsilon$. In contrast, the final destination depends on the initial condition, as multiple candidates are present.




In Sections 4 and 5, we discuss the relationships between the basin boundary and a solution starting from the associated TPO living at $-\infty$ to precisely determine the fate of the orbit. However, the dynamics of the reduced ODE system on the infinite line is consistent with the structure of the HIOP and PDE dynamics on a circle.



As the height increases and approaches the edge of the snaky and isola structures, a new type of patterns are frequently created via the interaction between the tail and bump. For instance, as observed from Fig.\ref{m.eps}-(j) with $\varepsilon =0.011$, a large stationary pulse is created in the neighborhood of the bump region after the collision, and the resulting new two-peak solution is stable and persists within a small neighborhood of the associated height parameter. Nonetheless, such a process is beyond the capability of the reduced ODE system. In contrast, the HIOP encompasses those solutions and the multipeak solutions depicted in Fig.\ref{m.eps} pertains to the HIOP (as portrayed in Fig.\ref{S2.eps}-(n)). Therefore, we can predict the post-collision state by carefully observing the HIOP structure for a particular height. This issue is briefly discussed in Section 4. In this study, we focused on the regime situated away from the edge of the snaky and isola structures, i.e., inside the admissible interval in which the reduction method functions consistently with the PDE dynamics.

One of the most interesting dynamics of the TPO is a cyclic response against the bump, which arises from the oscillatory nature of the tail. In particular, the first two cycles of the REB--OSC--STA are ilustrated in Fig.\ref{phase_PDE.eps}-(B), wherein the negative height is displayed as its modulus increases owing to the limitation of the stable TPO existing as the snaky structure. Additionally, the constraint in the parameter space $\kappa_1$ implied that only the finitely many cyclic responses were observed. Thus, we further reduced the PDE dynamics to finite-dimensional ODEs proximate to a drift bifurcation point to more prominently observe the cyclic response \cite{Liehr2013,Nishiuraetal2007}. In particular, the reduced ODEs presented its cyclic response REB--OSC--STA infinitely many times, as depicted in the phase diagram Fig.\ref{ODE2.eps}. 

An additional advantage of this reduction is that it enables us to precisely understand the occurrence of the transition. First, recall that the TPO solution is associated with the traveling pulse solution located at $-\infty$, and the entire pulse orbit can be regarded as a heteroclinic or homoclinic orbit connecting the TPO with its asymptotic state, and the transition occurred by switching the destination from one state to another with variations in height.

To comprehend the mechanism of the transition along with the instant of its occurrence, the basin boundary is required to be identified in the initial space based on its variation with the $\varepsilon$. Consequently, for each fixed $\varepsilon$, the {\it basin boundary} separating the two outputs can be characterized based on the stable manifold of the saddle point relevant to the transition, as depicted in Fig.\ref{e+0.000557127_PEN.eps}. Moreover, the destination of the time-reversal direction of this stable manifold is one of the critical points, and the location of this critical point progresses leftward as $\varepsilon$ is negatively increased. Therefore, {\it infinitely many times of reconnection} occur prior attaining the transition point at which the pulse orbit coincides with the stable manifold. For further growth of $|\varepsilon|$, the pulse orbit reduces under the basin boundary and converges to another asymptotic state. 

The basin boundary successively undergoing infinitely many reconnections among the relevant critical points immediately prior to the transition (Fig.\ref{PEN-OSC.eps} in Section 5) is vital for understanding the sensitive behaviors of TPOs in various situations. In particular, a couple of examples are presented in Fig.\ref{transition.eps}, which revealed that the bump height was in the neighborhood of the transitions such as PEN to REB. Although the profiles of the red and blue orbits became similar as they approached the bump, their outputs varied. A more subtle case is presented in Fig.\ref{spiral.eps}, wherein the basin boundary encompassed the critical point (in time-reversal direction) associated with the unstable stationary pulse. Therefore, the final output was extremely sensitive toward the selection of the initial data surrounding it in any direction of ray (as detailed in Section 6). The above discussion was summarized as follows.






\begin{proposition}
  The basin boundary between the two distinct outputs for the reduced ODE system (\ref{eq3.50}) can be characterized by the stable manifold of the specific saddle point relevant to the transition. The destination in the time-reversal direction of the basin boundary is one of the critical points of Eq.(\ref{eq3.50}) for each fixed $\varepsilon$ up to the concerning transition point. As $\varepsilon$ approached the transition point, the location of the destination progressed leftward and the basin boundary precisely coincided with the pulse orbit at the transition point to experience infinitely many reconnections among the critical points, which causes a sensitive dependence of the initial condition in the neighborhood of a critical point, especially for unstable spiral case (Section 6).
\end{proposition}

\begin{figure*}
  \centering
  \includegraphics[width=10cm]{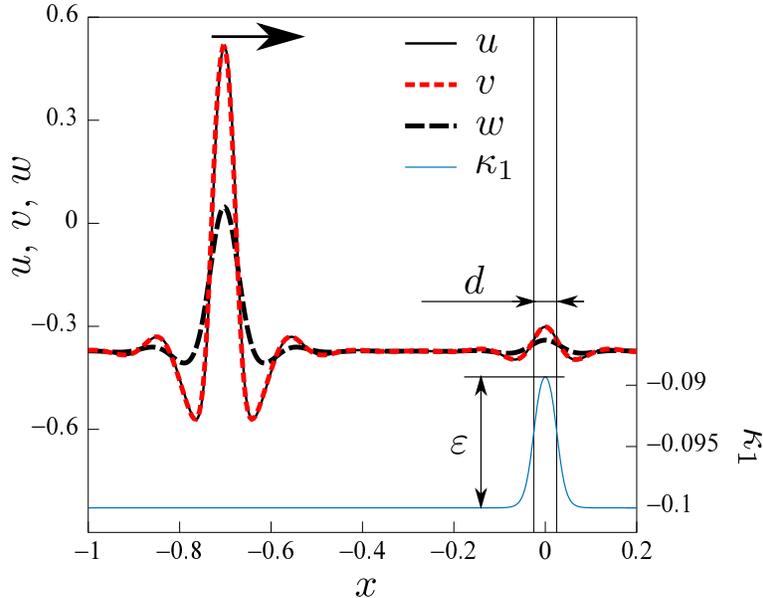}
  \caption{Traveling pulse in heterogeneous media of bump type.
    Heterogeneity $\kappa_{1}(x,\varepsilon,d)$ (blue-solid curve) of bump type with width $d$ and height $\varepsilon\neq0$ is introduced, which induces nonuniform pattern around the bump as a new background state. 
    Traveling pulse solution with oscillatory tails is initially placed
    far leftward from bump heterogeneity.
    Pulse propagates to right-hand direction, collides with the bump, and emits various outputs depending on height $\varepsilon$.
  }
  \label{hetero.eps}
\end{figure*}



The remainder of this paper is organized as follows.
In Section 2, we introduce the model equations and present the appropriate parameter regime for this study. In Section 3, we study the homogeneous problem. In context, a snakes-and-ladders structure of stationary pulses and figure-eight-like stack of isolas structure of traveling pulses are presented with a mechanism of localization of patterns, including the explanation of the occurrence of drift bifurcation based on a global bifurcation perspective. In Section 4, we study the heterogeneous problem of bump type by numerically exploring an entire set of HIOPs. For a given bump height, the asymptotic state can be expressed as one of the HIOPs without solving evolutions of the model system. In Section 5, we reduced the PDE dynamics to a finite dimensional system using multiple time-scale analysis, which revealed a periodic cycle of pulse response against the bump heterogeneity. Based on these findings, we clarified the occurrence of transitions according to variations in height and explained the behaviors of stable and unstable manifolds of saddle points distributed on the line, especially in the vicinity of the transition points. Moreover, the interrelation between the pulse orbit and stable/unstable manifolds are vital for understanding the mechanism of pulse response against the bump heterogeneity. We discuss the sensitive dependence on the initial conditions in Section 6, specifically, for the spiral case. Ultimately, we concluded the article in Section 7 with the findings, scope, open problems, and limitations of this study.  \\

\section{Model equation}

The model system constitutes the following three-component reaction diffusion equations:
\begin{equation}
  \label{eq2.01}
  \left\{ \begin{array}{lcl}
    u_{t}=D_{u} \Delta u+\kappa_{2} u-u^{3}-\kappa_{3} v-\kappa_{4} w+\kappa_{1}\\
    \tau v_{t}=D_{v} \Delta v+u-v\\
    \theta w_{t}=D_{w} \Delta w+u-w
  \end{array}\right.
\end{equation}
where $\Delta$ denotes the Laplacian; $u = u(x,t)$, $v = v(x,t)$, and $w = w(x,t)$ depend on time $t$ and space $x\in\mathbb{R}^{1}$; $\kappa_{1}$, $\kappa_{2}$, $\kappa_{3}$, and $\kappa_{4}$ represent kinetic parameters; $\tau$ and $\theta$ denote time constants; $D_{u}$, $D_{v}$, and $D_{w}$ denote diffusion coefficients. The developed system is a one-activator and two-inhibitor model, and it was derived and studied by Purwin et al.\cite{PhysRevLett.78.3781,Purwins2005} (\cite{Liehr2013} including 
the references), wherein they derived a qualitative model for
the gas discharge phenomenon. In particular, the parameter $\kappa_{1}$ is proportional to the voltage difference between the two plates, so it can be externally controllable and becomes a bifurcation parameter in this study. Eq.(\ref{eq2.01}) can be regarded as a generalized FitzHugh--Nagumo system of three-component type, as it was reduced to the usual FitzHugh--Nagumo equations upon discarding the second inhibitor $w$. 

The above-mentioned system (\ref{eq2.01}) displayed surprisingly rich and
complex behaviors from a dynamical system perspective.
For instance, the system (\ref{eq2.01}) supports stable stationary/traveling
pulses in 1D and spot (bullet) solutions with monotone/oscillatory tails in 2D (3D) depending on the parameters.
In addition, self-replication, self-destruction, or spatiotemporal
chaos\cite{Nishiuraetal2007,PhysRevE.75.036220} were observed in the appropriate parameter regime. The dynamics of the interaction with heterogeneities is the fundamental research topic of this study, and it was investigated for localized patterns, but primarily for monotone tail case \cite{Nishiuraetal2007,PhysRevE.75.036220,Ei2002TheMO,PhysRevE.79.046205,NishiuraTeramotoYuan2012,NishiuraTeramotoUeda2003,PhysRevE.67.056210,doi:10.1063/1.2087127,PhysRevE.80.046208,PhysRevE.79.046205,PhysRevE.69.056224}. Nonetheless, this research topic is still a fertile ground for patterns with oscillatory tails. Thus, the third component $w$ is indispensable for the coexistence of
multiple number of stable localized traveling patterns in higher dimensional spaces,
provided that the diffusivity $D_{w}$ of $w$ is larger than those of $u$ and $v$,
which is an ideal setting for studying collision dynamics in higher
dimensional space.
Therefore, Eq.(\ref{eq2.01}) was designated as a paradigm model for studying spatial localized patterns \cite{Liehr2013,NishiuraTeramotoUeda2003}.

\subsection{Parameter setting and bump heterogeneity}

The first task is to determine an appropriate parameter regime in which stable TPOs can be observed. Generally, a wave instability (Hopf bifurcation with finite wave number) occurs in the class of three-component reaction--diffusion system \cite{Sakamoto_wave_instability,Nakao_wave_instability}. Thus, these patterns can be determined in case certain localization mechanism exerts to the system. Prior to exploring the traveling pulses, we discuss in Section 3 that Turing instability\cite{Turing1952} coupled with beat bifurcation produces stationary localized patterns with snaky structures under the following set of parameters that were employed in the subsequent discussion.

\begin{equation}
  \label{parametervalues}
  \left\{ \begin{array}{lcl}
    \kappa_{2}=1.17,\kappa_{3}=0.3,\kappa_{4}=1.0,\\
    \theta=0.0,\\
    (D_{u},D_{v},D_{w})=(1.1\times10^{-4},0,9.8\times10^{-4}).
  \end{array}\right.
\end{equation}
\noindent
Note that $\theta=D_{v}=0$. Remarkably, the system (\ref{eq2.01}) supported stable traveling pulses under the same setting (\ref{parametervalues}) in the form of figure-eight-like stack of isolas instead of snaky structures, as explained in Section 3.1, which appears as an imperfection of the snaky structure.

Under the setting (\ref{parametervalues}), the system (\ref{eq2.01}) yields
\begin{equation}
  \label{eq_bif_1}
  \left\{
  \begin{array}{rcl}
    u_t&=&D_u\Delta u+\kappa_2u-u^3-\kappa_3v-\kappa_4w+\kappa_1\\
    \tau v_t&=&u-v\\
    0&=&D_w\Delta w+u-w
  \end{array}
  \right.
\end{equation}

\noindent
The homogeneous solution $(\underline{u}(\kappa_{1}),
\underline{v}(\kappa_{1}), \underline{w}(\kappa_{1})$) of Eq.(\ref{eq_bif_1})
must satisfy
$\underline{u}=\underline{v}=\underline{w}$, and it is uniquely determined
from the current parameter setting. More precisely, it is a solution of
\begin{equation}
  \label{eq2.03}
  -u^{3}+(\kappa_{2}-\kappa_{3}-\kappa_{4})u+\kappa_{1}=0,
\end{equation}
under the sign conditions $\kappa_{2}-\kappa_{3}-\kappa_{4}<0$ (refer to Eq.(\ref{parametervalues})).
This is called the background state, as the tail part of the traveling
pulse converges to this constant state on both sides. Note that $\underline{u}(\kappa_{1})$ is a monotonously increasing function of $\kappa_{1}$, and $\underline{u}(\kappa_{1}) < 0$ for $\kappa_{1} < 0$. If the origin is translated to the background state and Eq.(\ref{eq_bif_1}) is rewritten, then the following operation such as $(1-D_w\Delta)^{-1}$ is possible on $L^2(\mathbb R)$, considering the localized patterns. However, this study considered the system (\ref{eq_bif_1}) on a bounded interval with periodic boundary conditions, namely, on $\mathbb S^1$. Moreover, the generality is not lost, unless the length of the circle is adequately large in comparison to  the size of the concerned localized patterns.
Similarly, the third equation can be solved as
\begin{equation}
  \label{eq_bif_4}
  w=\left(1-D_w\Delta \right)^{-1}u,\;\;\;
  u,w\in L^2(\mathbb S^1),
\end{equation}
and therefore, Eq.(\ref{eq_bif_1}) becomes a system of two components $(u,v)$ with nonlocal terms:
\begin{equation}
  \label{eq_bif_2}
  \left\{
  \begin{array}{rcl}
    u_t&=&\left(
    D_u\Delta +\kappa_2-\kappa_4\left(1-D_w\Delta \right)^{-1}\right)u\\
    &&-\kappa_3v-u^3+\kappa_1\\
    \tau v_t&=&u-v
  \end{array}
  \right.
\end{equation}


We are interested in the heterogeneity of bump type introduced in the parameter space $\kappa_{1}$, which is discussed as follows.

\begin{equation}
  \label{eq4.01}
  \kappa_{1}(x,\varepsilon,d)=\underline{\kappa}_{1}
  +\widehat{\kappa}_{1}(x,\varepsilon,d)
\end{equation}
where $\underline{\kappa}_{1}$ denotes the baseline situated far from the bump and will be selected as $\underline{\kappa}_{1}=-0.1$ in Section 4. In context, the baseline can be selected independently, but it should belong to the briefly defined admissible interval of the TPO.
The profile $\widehat{\kappa}_{1}$ can be expressed as
\begin{equation}
  \label{eq4.02}
  \begin{array}{l}
    \widehat{\kappa}_{1}(x,\varepsilon,d)\\
    \displaystyle
    =\varepsilon\left(
    \frac{1}{1+\ope^{-\gamma(x+d/2)}}+\frac{1}{1+\ope^{\gamma(x-d/2)}}-1\right),
  \end{array}
\end{equation}
where
$\varepsilon$ controls the height and $d$ denotes the width of the bump.
In addition, $\gamma$ controls the steepness of the slope and was fixed as 
$\gamma=100$ throughout the study. The heterogeneity $\widehat{\kappa}_{1}$ was altered in appropriate way for computations on circle, if necessary.
Although the height $\varepsilon$ and width $d$ were varied, we added 
a restriction for the height $\varepsilon$ to ensure 
the existence of stable traveling pulse solutions for the associated
homogeneous problem both inside and outside of the bump. Otherwise, much more complicated dynamics emerged near the bump, which was beyond the scope of this study. Therefore, we opted for the parameter regime with respect to $\kappa_{1}$ in which the stable traveling pulse existed, i.e., {\it the admissible interval} for TPO.
In particular, $\varepsilon\in[-0.035866, 0.012]$
(or $\kappa_{1}\in[-0.135866,-0.088]$) lies between the two saddle-node points plotted in Fig.\ref{snakes_and_ladders.eps}. This topic is elaborated further in the subsequent section.


This study primarily considered a bump heterogeneity with short width, as the
analysis for the wider case can be conducted in a similar approach to that of \cite{Nishiuraetal2007,PhysRevE.79.046205}.
Ultimately, the parameter $\tau$ controlled the velocity of the traveling
pulse and was vital for determining the drift bifurcation.
In particular, the drift bifurcation point was expressed as 
$\tau_{c}=1/\kappa_{3}=3.33\cdots$ (see Appendix A.2), which enabled us to reduce the PDE dynamics to a finite-dimensional ODEs.
In the subsequent sections, we used $\tau=3.35$ for the PDE and ODE simulations.


Upon introducing a heterogeneity, the constant background state ceases to be a solution and should be replaced by an inhomogeneous solution. In Section 4, such a nonconstant
stationary pattern existed in the admissible interval, as depicted in Fig.\ref{S0.eps}, which is emanated from the homogeneous background state at $\varepsilon=0$. We refer this set of all solutions of (3) as HIOP, which was caused by the heterogeneity.



A schematic representation of the traveling pulse with the heterogeneity of bump type $\kappa_{1}(x,\varepsilon,d)$ (solid-black curve for $u$) is illustrated in Fig.\ref{hetero.eps}, wherein a small localized solution (blue color) represents a new background state.
The traveling pulse solution with the oscillatory tails was initially placed far from the bump heterogeneity, and it started to propagate rightward. Thereafter, it collided with the new
background state and emitted various outputs as discussed in Section 4.

\subsection{Numerical scheme}

The outline of the numerical scheme is stated herein. As the periodic boundary conditions were imposed on Eq.(\ref{eq_bif_2}), the $u$, $v$, and $\kappa_1$ were discretized as
\[
\left(\begin{array}{c}
  u(x,t)\\v(x,t)
\end{array}\right)=\sum_{l=-\left[\frac{L}{2}\right]}^{\left[\frac{L-1}{2}\right]}
\ope^{\opi klx}\left(\begin{array}{c}
  u_l(t)\\v_l(t)
\end{array}\right),\;\;\;
\kappa_1(x;\underline{\kappa}_{1},\varepsilon)
=\sum_{l=-\left[\frac{L}{2}\right]}^{\left[\frac{L-1}{2}\right]}
\ope^{\opi klx}\kappa_{1,l},
\]
and thus, a dynamical system:
\begin{equation}\label{dzdt}
  \frac{\opd\vec z}{\opd t}=\vec f(\vec z),
\end{equation}
could be obtained, where
\[
\vec z=\{z_{lm}\}\;\;(l=1,2,\cdots,L;\;m=1,2),\;\;\;
z_{l1}=u_l,\;z_{l2}=v_l,
\]
and the flow $\vec\phi=\{\phi_{lm}\}$ generated by the
the vector field $\vec f$ was defined as
\[
\vec z(t)=\vec\phi(\vec z^{(0)},t),
\]
where $\vec\phi(\vec z^{(0)},0)=\vec z^{(0)}$.
Various classes of the solutions including steady, steady traveling,
time periodic, and periodic traveling solutions were expressed as the conditions
to the flow $\vec\phi(\vec z^{(0)},t)$ listed in Table\ref{conditions}.
\begin{table}
  \renewcommand{\arraystretch}{1.5}
  \begin{center}
    \caption{
      Conditions for various classes of solutions for flow
      $\protect\vec\phi(\protect\vec z^{(0)},t)$,
      where $U$ denotes a phase/group velocity; $\beta$, a period;
      $t$, an arbitrary positive time;
      $\|\{a_{lmn}\}\|=\sqrt{\sum_{l,m,n}a_{lmn}^2}$;
      $\eta$, a small positive value
      that determines the convergence criterion
      with $\eta=10^{-10}$.
    }\label{conditions}
    \begin{tabular}{lll}
      \hline
      Class of solution&Condition&Convergence criterion\\
      \hline\hline
      Steady
      &
      $\phi_{lmn}\left(\vec z^{(0)},t\right)=z_{lmn}^{(0)}$
      for $^\forall t$
      &
      $\displaystyle
      \int_0^t\left\|\left\{\frac{\del
        \phi_{lmn}\left(\vec z^{(0)},t'\right)
      }{\del t'}\right\}\right\|\opd t' < \eta
      $\\
      \hline
      Steady traveling
      &
      $\phi_{lmn}\left(\vec z^{(0)},t\right)\ope^{\opi klUt}=z_{lmn}^{(0)}$
      for $^\forall t$ and $^\exists U$
      &
      $\displaystyle
      \int_0^t\left\|\left\{\frac{\del
        \phi_{lmn}\left(\vec z^{(0)},t'\right)\ope^{\opi klUt'}
      }{\del t'}\right\}\right\|\opd t'<\eta
      $\\
      \hline
      Time-periodic
      &
      $\phi_{lmn}\left(\vec z^{(0)},\beta\right)=z_{lmn}^{(0)}$
      for $^\exists\beta$
      &
      $\left\|\left\{z_{lmn}^{(0)}
      -\phi_{lmn}\left(\vec z^{(0)},\beta\right)\right\}\right\|
      <\eta$\\
      \hline
      Periodic traveling
      &
      $\phi_{lmn}\left(\vec z^{(0)},\beta\right)\ope^{\opi klU\beta}=z_{lmn}^{(0)}$
      for $^\exists U$ and $^\exists \beta$
      &
      $\left\|\left\{z_{lmn}^{(0)}-\phi_{lmn}\left(\vec z^{(0)},\beta\right)
      \ope^{\opi klU\beta}\right\}\right\|<\eta$\\
      \hline
    \end{tabular}
  \end{center}
\end{table}

Subsequently, we sought the initial value $\vec z^{(0)}$ satisfying the condition
using the shooting method with the Newton--Krylov
method \cite{Watanabe2010,Watanabe2012}.
Any special technique such as the arc-length method is not used for branch
continuation.
Thereafter, the solution at the tip of the branch was used as the
initial estimation of the Newton--Krylov method to extend the branch, and the Arnoldi method was utilized to evaluate the linear stability of
each solution.
In addition, certain additional techniques were utilized when the branch yielded a saddle-node bifurcation point.
Upon continuing steady branch, the eigenmode obtained from the linear stability analysis was added with small weight
(1\% to 10\% with respect to the $L_2$ norm of $\vec z$) to the tip
solution to derive the initial estimation of the new branch.
In contrast, the period $T_0$ remained constant for the time-periodic branches, whereas the bifurcation parameter $\varepsilon$ or $\kappa_1$ was
considered as a variable.
In context, refer to Watanabe et al. \cite{Watanabe2012,Salingeretal2014} for the details of the calculation method.

\section{Homogeneous problem}

\subsection{Linearized analysis}

As a prerequisite condition, we identified the admissible interval of the TPO with respect to $\kappa_{1}$ containing the stable traveling pulse.
Such an admissible interval relies on other parameters as well. Specifically, the drift bifurcation occurs as $\tau$ crosses the following critical value $\tau_c$ (as detailed in Appendix A.2)
\begin{equation}
  \label{eq:drift}
  \tau_c=1/\kappa_3=3.33\cdots.
\end{equation}
The solution branch of stable traveling pulses ought to be explored at least in two-parameter space $(\kappa_{1}, \tau)$. Notably, the above-mentioned drift bifurcation point does not depend on the system size under periodic boundary conditions. Without losing generality, we fixed $\tau = 3.35 > \tau_c$ and investigated the global solution branches of Eq.(\ref{eq_bif_2}) with respect to $\kappa_{1}$ emanating from the constant background solution. In this study, we aim to demonstrate that the stationary solutions form a snaky structure and the traveling solutions form a stack of isolas with a figure-eight shape, as depicted in Fig.\ref{snakes_and_ladders.eps} and Fig.\ref{isola.eps}, respectively. This will enable the identification of the admissible interval for observing both the generation and disappearing mechanism of these patterns with varying $\kappa_{1}$. Although  finitely many traveling pulses exist with varying number of peaks, we are primarily interested in the single-peak pulse.



First, the linear stability of the background solution $(\underline u(\kappa_1), \underline v(\kappa_1))$ was investigated, which satisfied Eq.(\ref{eq2.03}) and $\underline u(\kappa_1)=\underline v(\kappa_1)=\underline w(\kappa_1)$. The $\underline u(\kappa_1)$ function is a monotonously increasing function of $\kappa_1$.
On linearizing Eq.(\ref{eq_bif_2}) with respect to $(\underline u(\kappa_1),
\underline v(\kappa_1))$ and substituting
$(u(\kappa_1,x,t),v(\kappa_1,x,t))
=(\underline u(\kappa_1),\underline v(\kappa_1))
+(\hat u,\hat v)\ope^{\lambda t+\opi\omega x}$
into it, we derived
\begin{equation}
  \label{eq_bif_3}
  \lambda\left(
  \begin{array}{c}
    \hat u\\\hat v
  \end{array}
  \right)
  =
  \left(
  \begin{array}{cc}
    a(\kappa_1)&-\kappa_3\\
    1/\tau&
    -1/\tau
  \end{array}
  \right)
  \left(
  \begin{array}{c}
    \hat u\\
    \hat v,
  \end{array}
  \right)
\end{equation}
where
\[
a(\kappa_1)=-D_u\omega^2+\kappa_2-3\{\underline u(\kappa_1)\}^2-
\frac{\kappa_4}{1+D_w\omega^2}
\]
$a(\kappa_1)$ denotes a monotonously increasing function of $\kappa_1$ for $\kappa_1 < 0$, because $\{\underline u(\kappa_1)\}^2$ is monotonously decreasing for the same region. In fact, as we will see soon, the interesting structure appears in $\kappa_1 < 0$. The eigenvalue problem (\ref{eq_bif_3}) can be solved to yield the solution.
\[
\lambda=\frac{
  a(\kappa_1)-\displaystyle\frac1\tau
  \pm\sqrt{\left(a(\kappa_1)-\frac1\tau\right)^2
    -\frac{4(\kappa_3-a(\kappa_1))}{\tau}}
}{2}.
\]
As $\kappa_1$ increased, $a(\kappa_1)$ attains $1/\tau$, and thereafter, $\kappa_3$, because $1/\tau < \kappa_3 = 0.3$, which implies that for any fixed $\tau>\tau_c$, the Hopf bifurcation initially occurs and the pitchfork bifurcation follows the same wavenumber
$\omega_c$ as $\kappa_{1}$ is increased, as depicted in Fig.\ref{Stability_Trivial.eps}. Herein, the critical wavenumber $\omega_c$ is expressed as 
\[
\omega_c=\sqrt{\frac{\sqrt{\kappa_4D_uD_w}-D_u}{D_uD_w}}.
\]

\begin{figure}
  \begin{center}
    \includegraphics[width=.6\hsize]{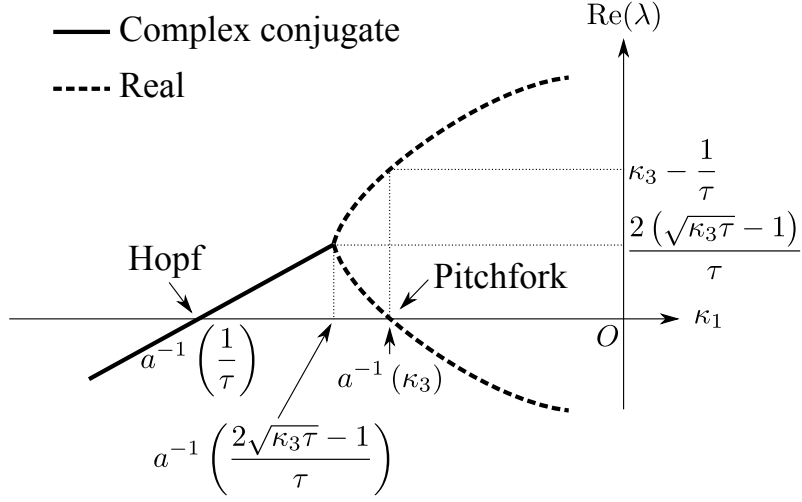}
    \caption{ Linear stability of trivial solution. As $\kappa_1$ increases, Hopf bifurcation initially occurs at $\kappa_1 = a^{-1}(1/\tau)$, and the traveling wave train emerges with the wavenumber $\omega_c$. Subsequently, pitchfork (Turing) bifurcation occurs at $\kappa_1 = a^{-1}(\kappa_3)$, and unstable stationary wave train bifurcates from trivial solution.}
    \label{Stability_Trivial.eps}
  \end{center}
\end{figure}
More precisely, 
$\displaystyle\lambda=\pm\opi\frac{\sqrt{\kappa_3\tau-1}}{\tau}$ is attained for $\displaystyle a(\kappa_1)=\frac{1}{\tau}$ with occurrence of the Hopf bifurcation. In addition, the traveling wave solution with wavenumber $\omega_c$ with phase velocity
$\displaystyle\pm\frac{\sqrt{\kappa_3\tau-1}}{\tau\omega_c}$
bifurcates from the trivial solution.
As $\kappa_1$ is further increased and 
$\displaystyle a(\kappa_1)=\frac{2\sqrt{\kappa_3\tau}-1}{\tau}$,
$\displaystyle\lambda=\frac{2\left(\sqrt{\kappa_3\tau}-1\right)}{\tau}$,
and thus, the conjugate pair degenerates to yield two real eigenvalues.
As $\kappa_1$ is increased further, one of the real eigenvalues becomes larger, whereas the
other becomes diminishes.
For $a(\kappa_1)=\kappa_3$, smaller real eigenvalue crosses the real axis,
and thus, pitchfork bifurcation occurs and stationary wave train with
wavenumber $\omega_c$ appears, which is known as
Turing instability \cite{Turing1952}. Notably, for $\tau = 1/\kappa_3$, the above-mentioned Hopf and pitchfork bifurcation points coincide with each other and becomes equal to the critical value
(\ref{eq:drift}).

\subsection{Snakes-and-ladders and figure-eight-like stack of isolas}

In the present bifurcation analysis, we considered the problem on a finite domain with periodic boundary conditions. Thus, the destabilized wavenumber was selected in a discretized manner. Moreover, we considered $x\in[0,0.56)\dummy{]}$, and the domain contained four waves as $8\oppi/\omega_c\approx0.56$. The system allowed more waves as its size enlarged, and the consequent hierarchical snaky and stack-of-isola structures become more stratified, as illustrated in Fig.\ref{1.68.eps}. However, its essence remained the same. The evolution of the global bifurcation diagram can be observed by comparing the two corresponding diagrams.

\begin{figure*}
  \begin{center}
    \includegraphics[width=\hsize]{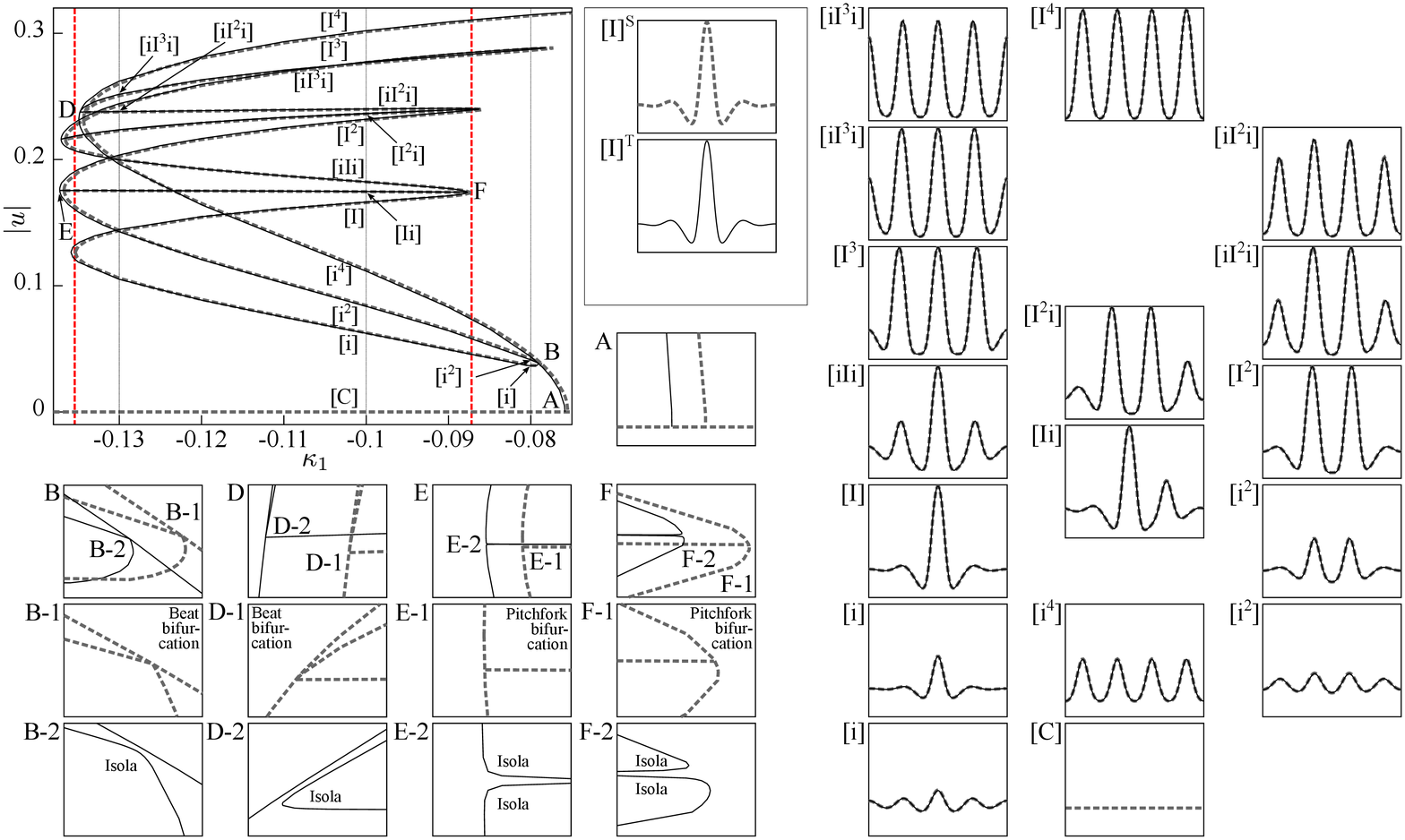}
    \caption{Global bifurcation diagram and visualizations of corresponding solutions with respect to $\kappa_{1}$ for a fixed $\tau=3.35$.
      System size: $0.56$.
      Vertical axis denotes the norm of $u$.
      Solid and dotted lines denote traveling and stationary patterns, respectively.
      Each magnified diagram (A, B, and D--F) details the local bifurcation.
      Associated solution profiles are depicted in the right-hand side.
      Barcode type notations such as $\mathrm{[iIi]}$ are used to indicate solution shapes.
      See text and Fig.\ref{schematic_snakes_and_ladders.eps} for meaning of barcode.
      These notations without superscript indicate both of them, i.e.,
      $\left\{[\cdot]\right\}=\left\{[\cdot]^\mathrm{S}\right\}\cap\left\{[\cdot]^\mathrm{T}\right\}$.
      An example of this notation for $\mathrm{[I]}$ is depicted in the figure:
      The solution branch and solution profile drawn in thick-gray dotted line is $\mathrm{[I]}^\mathrm{S}$, that of the thin-black solid line is $\mathrm{[I]^\mathrm{T}}$, and both of them are indicated by $\mathrm{[I]}$.
      The diagram comprises three portions: Stationary and traveling wave trains emanating from the trivial branch (see A and Fig.\ref{Stability_Trivial.eps}), snaky structure of stationary pulse patterns as a secondary bifurcation (dotted line in B-1), and figure-eight-like stack of isolas of traveling pulse patterns (solid line in B-2).
      The magnified structure of figure-eight-like stack of isolas is depicted in Fig.\ref{isola.eps}, wherein $\mathrm{[C]}$ denotes the trivial constant state. The vertical-dotted line at $\kappa_1 = -0.1$ denotes the baseline for the heterogeneous problem, i.e., the parameter $\kappa_1$ was maintained at a value of $-0.1$ outside the bump. In the pattern of $u$ depicted in the insets, certain patterns appeared as a sharp edge at both the ends. However, these patterns satisfied the periodic boundary condition and connected smoothly at both the ends.
    }\label{snakes_and_ladders.eps}
  \end{center}
\end{figure*}

\begin{figure*}
  \begin{center}
    \includegraphics[width=.2\hsize]{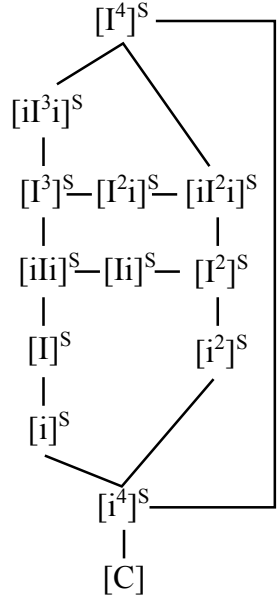}
    \caption{
      Schematic of snakes-and-ladders bifurcation structure.
      Even and odd type of snaky branches connect $\mathrm{[i^4]}^\mathrm{S}$ to $\mathrm{[I^4]}^\mathrm{S}$, and the ladder branches bridge these two parts. Barcode depicts the number of large and small peaks in the pattern; for instance, $\mathrm{[iIi]}$ depicts one large peak in the middle and two small peaks on both sides; superscript ``S'' denotes stationary branch.
    }\label{schematic_snakes_and_ladders.eps}
  \end{center}
\end{figure*}

\begin{figure*}
  \begin{center}
    \includegraphics[width=0.8\hsize]{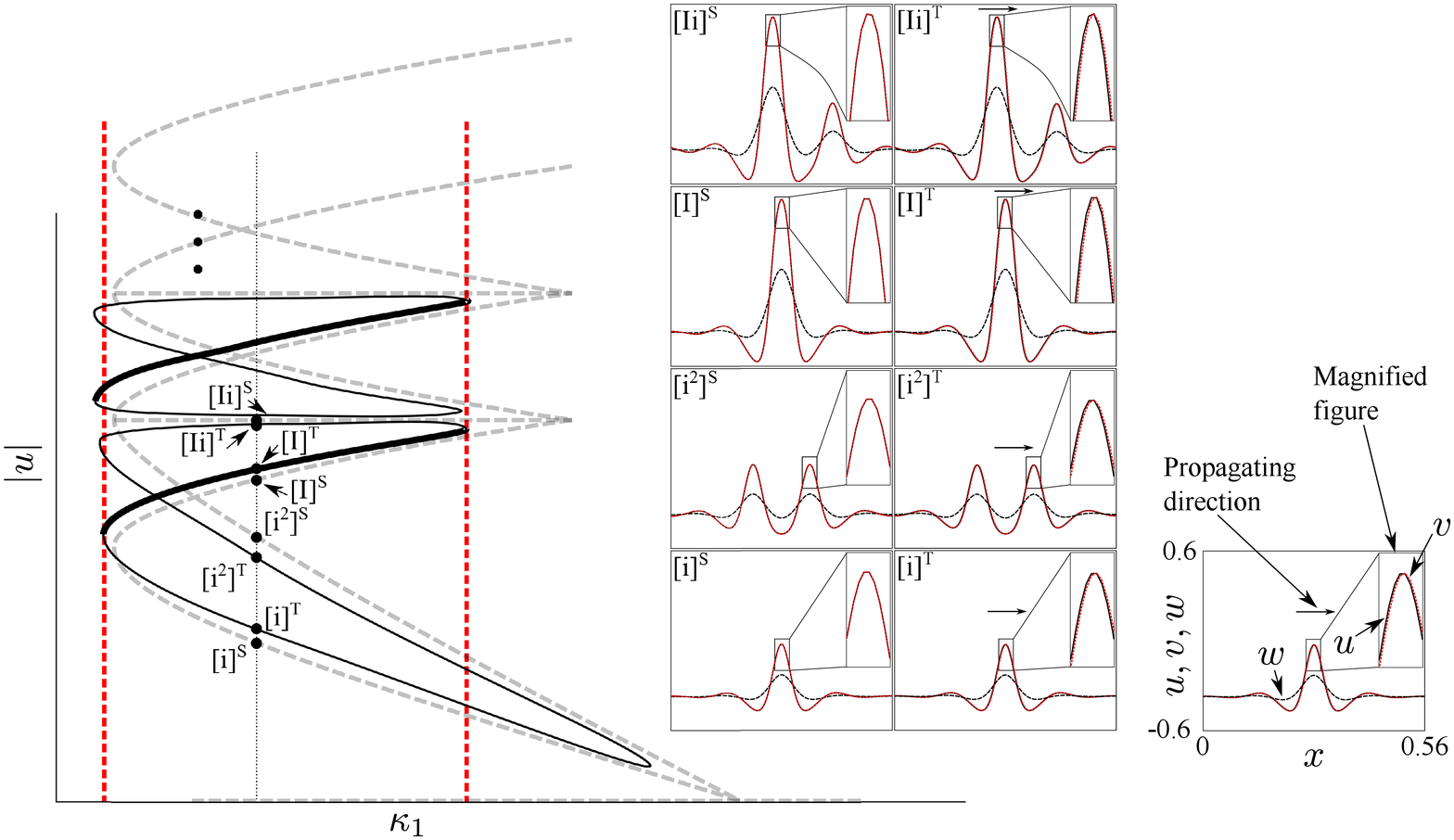}
    \caption{
      Relationship between snakes-and-ladders structure and stack of isolas in Fig.\ref{snakes_and_ladders.eps}. Each figure-eight-like isola represents a traveling pulse branch. The number of peaks increased from lower to upper stack, as illustrated in the right-hand side of the figure; superscripts of $\mathrm{[Ii]}^\mathrm{T}$ and $\mathrm{[Ii]}^\mathrm{S}$ denote ``Traveling'' and ``Stationary'', respectively, and the thick line depicts the stable portion of the branch. The admissible parameter region of $\kappa_1$ is defined as the interval spanned by two saddle-node points of the lowest isola.}\label{isola.eps}
  \end{center}
\end{figure*}

The parameters were set as (\ref{parametervalues}) and an interval of $\kappa_{1}$ was sought for which the stable traveling pulses existed. The four categories of solutions are as follows: stationary-wave-train, traveling-wave-train, stationary pulse, and traveling pulse.
Herein, we used ``pulse'' instead of ``localized pattern'' for the 1D case. Their interrelations were clarified by exploring the global behaviors of all the relevant solution branches emanating from the trivial constant state.
The resulting global bifurcation diagram is illustrated in Fig.\ref{snakes_and_ladders.eps}.
The thick-gray-dashed lines depict the stationary solutions, and the thin-black-solid lines depict the traveling solutions, which were used for denoting the respective pattern profiles.
Additionally, the stationary branches formed a typical snaky structure that has been earlier observed in several systems \cite{WOODS1999147,PhysRevFluids.2.064401,Kozyreff,Ponedeletal2016,MA20101867,Burke2007,doi:10.1063/1.2746816,doi:10.1137/06067794X,doi:10.1063/1.4792711,BBKM_Eckhaus_snaking,Beck2009,KY_foliated_snaking_2020,ACV_localized_unifying_framework_2021}.
The three-component system (\ref{eq_bif_1}) is accompanied by traveling pulse branches, and the profiles of both the solutions resemble each other in the selected parameter setting. Moreover, a barcode notation is used in Fig.\ref{snakes_and_ladders.eps} to denote the number of large and small peaks with their locations such as $\mathrm{[iIi]}$; one large peak in the middle and two small peaks on both sides. The superscript S or T is added to denote the stationary or traveling branch, respectively, and a barcode without superscripts denotes both types of solutions, as illustrated in Fig.\ref{snakes_and_ladders.eps}.

Thus, the traveling pulse branch can be regarded as an imperfection of the stationary pulse branch, because the profile of the traveling pulse does not retain the reflection symmetry of the stationary pulse owing to the drift bifurcation. Thus, the traveling pulse branch splits into several isolas rather than a single connected branch. In particular, the set of traveling pulse branches form a figure-eight-like stack of isolas, as illustrated in Fig.\ref{isola.eps}, which is one of the key features of the system, as described herein. A similar isola structure was observed in Fig. 3 of \cite{0133-0189_2009_Special_109} for symmetric two-pulse states; however, those were not traveling solutions.

Now let us observe the details of these structures. As depicted in Fig.\ref{Stability_Trivial.eps}, the trivial solution loses stability at
$\kappa_1=a^{-1}(\frac{1}{\tau})\approx-0.0758$
and the traveling-wave-train
solution subcritically bifurcates via Hopf bifurcation (wave instability).
Slightly after that, at
$\kappa_1=a^{-1}(\kappa_3)\approx-0.0755$,
the stationary-wave-train
solution subcritically bifurcates via pitchfork bifurcation.

\begin{figure}
  \begin{center}
    \includegraphics[width=.9\hsize]{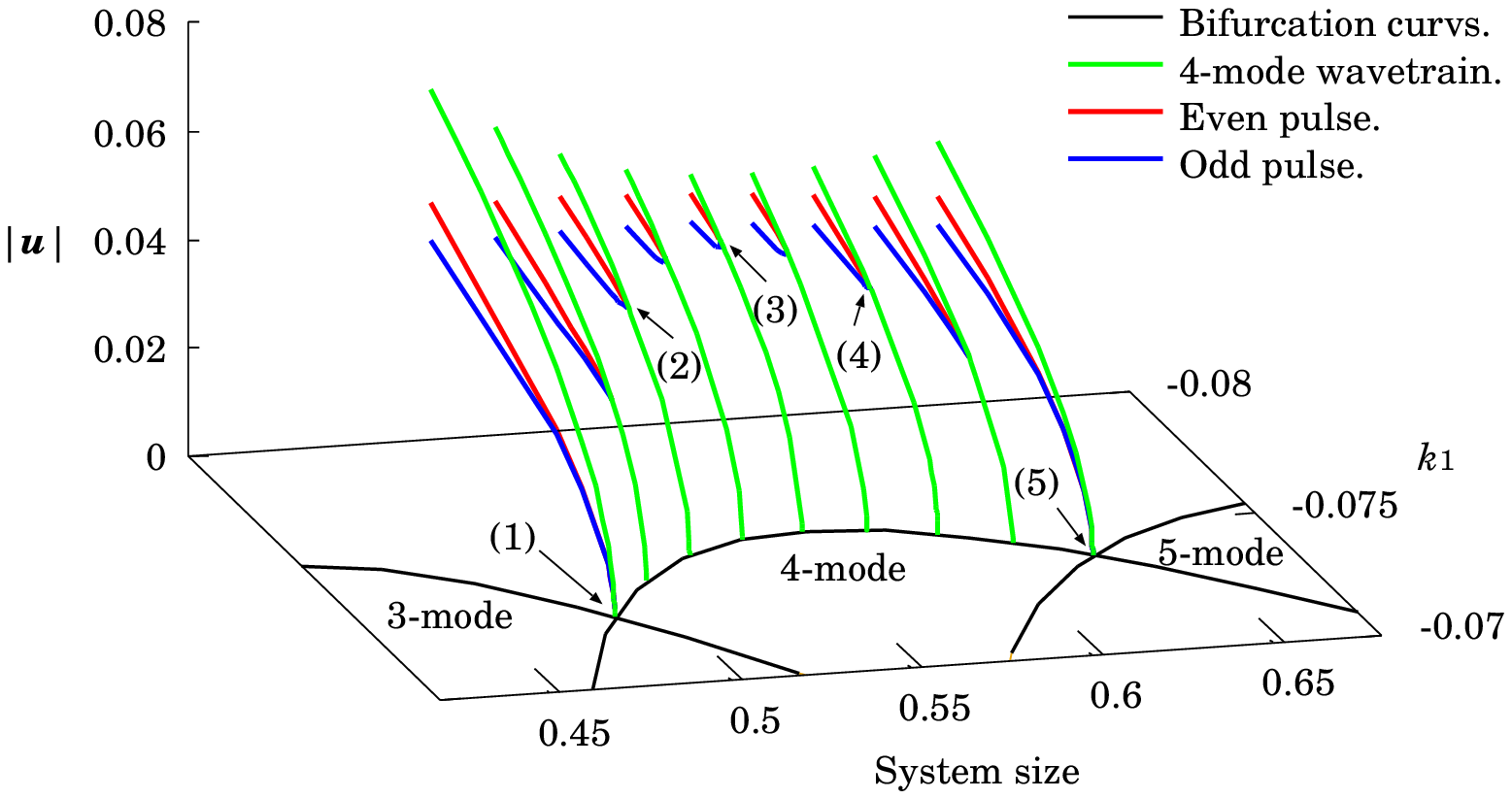}
    \caption{Relationship between beat bifurcation and domain size. Notably, two secondary branches emanated from 4-mode become supercritical when the system size is close to codimension 2 point such as (1) and (5) in the diagram, which is consistent with the local analysis conducted in \cite{Fujii-Mimura-Nishiura1982}. In case the system size is away from codimension 2 point such as (3), they are no more supercritical and the odd branch becomes subcritical, as depicted in inset B of Fig.\ref{snakes_and_ladders.eps}.}
    \label{bif_3d.eps}
  \end{center}
\end{figure}

First, the stationary branch was examined. As we are working on a finite interval, the trivial state loses stability for a particular wavenumber depending on the system size.
However, for certain size values, the trivial state simultaneously loses stability for two adjacent wavenumbers $n\omega$ and $(n+1)\omega$,
where $\omega$ denotes the basic wavenumber (see Fig.\ref{bif_3d.eps} for $n=4$).
Four branches bifurcated from the codimension-two point, namely, two wavetrains of wavenumbers $n\omega$ and $(n+1)\omega$, and two modulated mixed waves comprising wavenumbers $n\omega$ and $(n+1)\omega$ (Fig. 2.14 of \cite{Fujii-Mimura-Nishiura1982} and caption of Fig.\ref{bif_3d.eps}). 
Nevertheless, they unfold under a generic system size, and two modulated waves appear as secondary bifurcations, discussed as follows (for details, refer \cite{Fujii-Mimura-Nishiura1982,Burke_thesis}). 
To distinguish these two secondary branches, we assigned the label ``odd (even)'' to the branch the profile containing a local maximum (minimum) at the center including odd (even) number of peaks, referred to as $\mathrm{[i]}$ and $\mathrm{[i^2]}$ in Fig.\ref{snakes_and_ladders.eps}. The odd branch becomes subcritical for system size = 0.56, as depicted in the magnified diagram B-1 of Fig.\ref{snakes_and_ladders.eps}.

The two secondary branches display the onset of localization, as the two adjacent wavenumbers cause the ``beat'' phenomena, i.e., the amplitude of the envelope of modulated waves increases along each branch, and the odd (even) branch approaches a one (two)-peak stationary pulse such as $\mathrm{[i]}^\mathrm{S}$ ($\mathrm{[i^2]}^\mathrm{S}$) in Fig.\ref{snakes_and_ladders.eps}.
Although the superscript $\mathrm{S}$ does not appear in Fig.\ref{snakes_and_ladders.eps}, the $[\cdot]$ includes $[\cdot]^\mathrm{S}$, and thus, $\mathrm{S}$ was included in this paragraph because ``beat'' bifurcation only appeared for stationary branch as depicted in Fig.\ref{snakes_and_ladders.eps}-B.

As the parameter $\kappa_1$ decreased, they entered into a pinning regime and experienced saddle-node bifurcations. The number of peaks was increased by $2$ in case the branch turned around the saddle-node point on the right-hand side, such as $\mathrm{[I]}^\mathrm{S}\rightarrow\mathrm{[iIi]}^\mathrm{S}$ ($\mathrm{[I^2]}^\mathrm{S}\rightarrow\mathrm{[iI^2i]}^\mathrm{S}$). Once the domain was filled with the localized states, snaking must cease, and both branches terminated on the spatially periodic state $\mathrm{[I^4]}^\mathrm{S}$ (D in Fig.\ref{snakes_and_ladders.eps}). In addition, these branches are connected pairwise by secondary branches of \textit{asymmetric} states ($[\mathrm{Ii}]^\mathrm{S}$ and $[\mathrm{I^2i}]^\mathrm{S}$) forming a characteristic structure of \textit{the snakes-and-ladders structure}. Although the secondary bifurcations further approach the saddle-node bifurcations higher up in the snakes-and-ladders structure as the system size is increased, the branches invariably connect the unstable components of the odd and even branches (refer to E-1 and F-1). The above-mentioned computation implies that the three-component reaction diffusion system (\ref{eq_bif_1}) exemplifies the snakes-and-ladders structure of the stationary pulse solutions.

The traveling-wave-train bifurcated from the trivial
solution via Hopf bifurcation (thin-black-solid line in A of Fig.\ref{snakes_and_ladders.eps}), and it growed in parallel to the stationary-wave-train (thick-gray-dashed line) as $\kappa_1$ is decreased.
On the contrary, pulse branches appeared in vicinity of the stationary branches, such as in B of Fig.\ref{snakes_and_ladders.eps}. However, it was in the form of isolas (see B-2) instead of secondary bifurcations from traveling-wave-train. Additionally, their amplitude increased as $\kappa_1$ was decreased, and they existed close to the snaky structure of the stationary branches.
The global shape of the traveling pulses resembles {\it figure-eight-like stack of isolas}, as depicted in Fig.\ref{isola.eps}. In addition, the profiles of the three components are exhibited in Fig.\ref{isola.eps}: the thick-gray-dashed line depicts the stationary pulse, whereas the thin-black-solid line depicts the traveling pulse. Although these two types of branches exhibited similarities, the traveling pulses lacked reflectional symmetry. Therefore, the associated branch constituted an imperfection of the stationary pulses. Each of the figure-eight-like isolas comprised an odd and even traveling pulse branch connected via the saddle-node bifurcation, because both the traveling pulse branches could not retain the reflectional symmetry owing to propagation.  
Thus, an additional bifurcation parameter $\tau$ was required to identify the origin of the isola. Recall that the drift bifurcation occurred at $\tau=\tau_c$ for the stationary pulse branch, so all the figure-eight-like isolas contacted the stationary branch, as schematically represented in Fig.\ref{bif3.eps}. As we considered two parameters $\kappa_1$ and $\tau$, the solutions presented in the figure formed ``sheet'' instead of ``branch''. Moreover, the traveling sheet ceased to exist below $\tau = \tau_c$ and appeared as a ``gutter'' attached to the stationary sheet in the two-parameter space depicted in Fig.\ref{bif3.eps}.

\begin{figure}
  \begin{center}
    \includegraphics[width=.9\hsize]{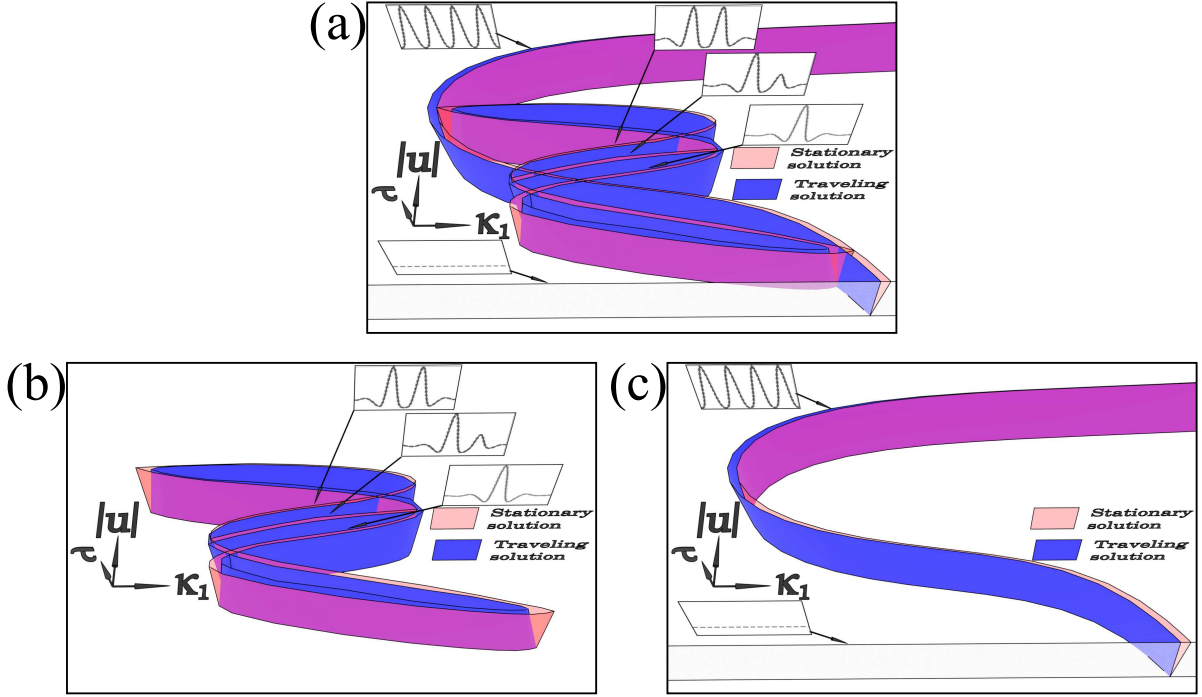}
    \caption{Schematic of stationary and traveling branches in $(\kappa_1, \tau)$-parameter space. The entire diagram (a) can be classified into two parts. The sheet of traveling wave train solutions (red) merges into sheet of stationary wave solutions (blue) as $\tau$ approaches drift bifurcation point $\tau_c$ and forms a gutter-like structure (b). Each figure-eight-like isola sheet contacts the associated snakes-and-ladders sheet as $\tau \downarrow \tau_c$ (c).  }  
    \label{bif3.eps}
  \end{center}
\end{figure}



For the stability of traveling branches, $\mathrm{[I}^n]^\mathrm{T}$ branches, where $n=1,2,3,4$ in Fig.\ref{snakes_and_ladders.eps} are stable at all instances, and their stability varied at each saddle-node bifurcation point.
Notably, both $\mathrm{[\cdot]}^\mathrm{T}$ and $\mathrm{[\cdot]}^\mathrm{S}$ are represented by the same barcode $\mathrm{[\cdot]}$ without superscript in Fig.\ref{snakes_and_ladders.eps}.
Moreover, the remaining branches and ladder branches were unstable at all instances.
For stationary branches, the stability features were almost the same as those
of traveling branches, except for an eigenvalue corresponding to the drift
bifurcation.
Thus, all stationary branches exhibited only one more positive eigenvalue
of drift instability in addition to the corresponding traveling ones.
The boundary of the isola structure determined the admissible interval containing the stable TPO. If one leaves the boundary, the destruction or replication of these pulses occur in a similar manner observed in \cite{NU_PhysicaD_1999,NU_PhysicaD_2001}, and Chapter 5 of \cite{Burke_thesis}.

Prior to concluding this subsection, we returned to the question to identify the admissible interval for TPO, i.e., we aimed to determined the range in which the stable one-peak traveling pulse existed.
Stable traveling one-peak pulse represents the thick portion of the $\mathrm{[I]}^\mathrm{T}$ branch in Fig.\ref{isola.eps},
and thus, this range was from lower saddle-node bifurcation point to upper
saddle-node bifurcation point of the branch.
The range was $-0.135866\le \kappa_{1}\le-0.088$, or equivalently $-0.035866\le \varepsilon\le 0.012$ with the baseline $\kappa_{1}=-0.1$. The stability of $\mathrm{[I]}^\mathrm{T}$ branch was investigated in certain larger regions, and the periodic domain size did not affect the range of stable region. In particular, the bifurcation structure in the three-times larger region in Fig.\ref{1.68.eps} was calculated,
and the relative difference of lower and upper limit were less than $0.3\%$, implying the robustness of the stable regime of the traveling pulses. 

\begin{figure}
  \begin{center}
    \includegraphics[width=.95\hsize]{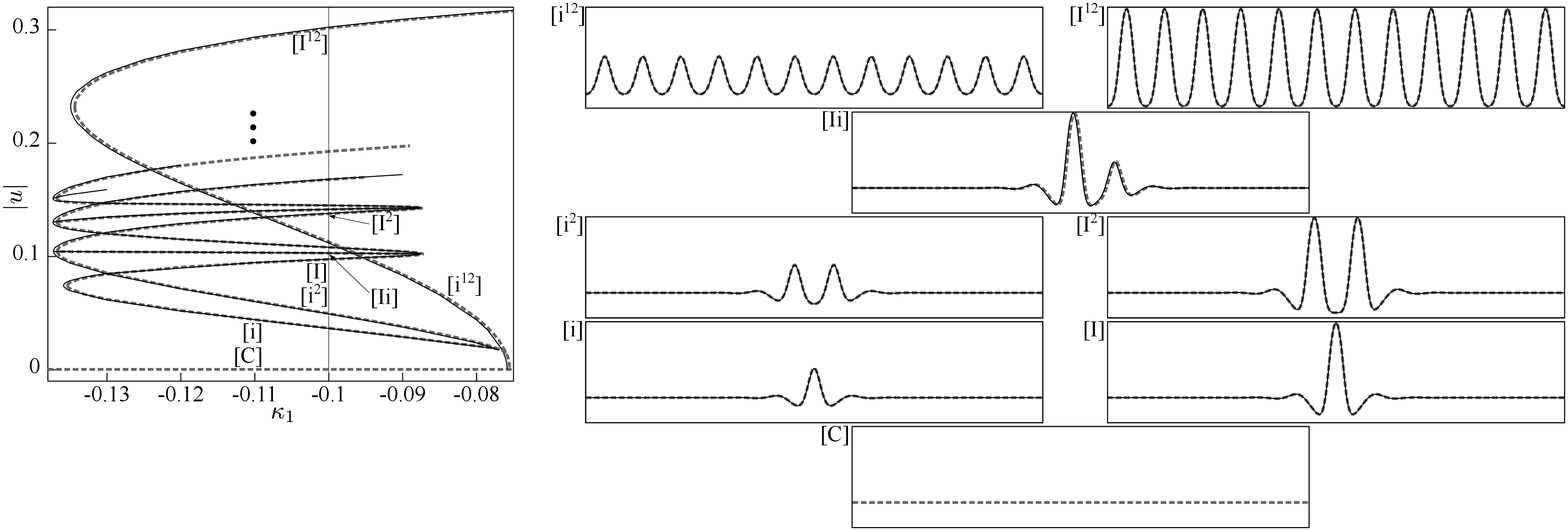}
    \caption{Bifurcation structure in three-times larger region. In case the system size was altered from $0.56$ to $1.68$, the associated bifurcation diagram qualitatively remained the same as Fig.\ref{snakes_and_ladders.eps}, although the hierarchical structure of snakes-and-ladders and stack of isolas becomes deeper proportional to the system size.}
    \label{1.68.eps}
  \end{center}
\end{figure}

Furthermore, we introduced a heterogeneity of the bump type in the parameter $\kappa_{1}$
in the following section with the baseline $\kappa_1=-0.1$ and its height
$\varepsilon$. We varied $\varepsilon$ and observed the behavior of TPO as it collided with bump. Although we could select a varying baseline for $\kappa_{1}$, our conclusion remained the same provided that it strictly belonged inside the admissible interval.

\section{Heterogeneous problem}

The heterogeneity of bump type (\ref{eq4.01}) with $d=0.05$ was introduced to the parameter $\kappa_1$ of Eq.(\ref{eq_bif_2}). First, we characterized the interaction of the TPO with the heterogeneity, including the outputs after collision. Subsequently, we will prove that HIOP contains all the outputs of TPO for the wide range of heights. We set up this problem in the following manner. As discussed in Section 2.1, the baseline outside of the bump was considered as 
$\kappa_1=-0.1$ from which the height $\varepsilon$ was measured. 
In this case, the height $\varepsilon$ of the bump was not arbitrary, but it can be varied in the admissible interval $-0.035866\le \varepsilon\le 0.012$, as discussed in the previous section. Interestingly, the extension of the HIOP branch is possible over the admissible interval of $\varepsilon$, which enabled us to predict the behavior of TPO beyond the admissible interval, although more complex dynamics emerges in such a regime, as discussed at the end of Section 4.3. Recall that the background state ceased to be in a constant state, and instead, inhomogeneous steady states localized around the bump region. The existence of new stable background state was numerically confirmed in the admissible interval, as depicted in Fig.\ref{S0.eps}. In particular, multiple candidates existed for a fixed small $\varepsilon$, so we employed the smallest stable HIOP emanating from the constant state at $\varepsilon=0$ as a new background state, which was uniquely determined in the admissible interval to avoid confusion (solid-red curve emanating from the origin in Fig.\ref{S0.eps}). Thus, we considered the heterogeneous problem of bump type, as discussed below.\\

\subsection{Phase diagram of PDE dynamics}

First, we examined the response of TPO upon collision with the bump and observed its variations with $\varepsilon$. The phase diagram of Fig.\ref{phase_PDE.eps} depicted that 
four qualitatively varying outputs relied on the width $d$ and height $\varepsilon$: PEN, OSC, STA, and REB. The red-dotted horizontal line in Fig.\ref{phase_PDE.eps} indicates the upper and lower boundaries of admissible interval to ensure that the pulse does not survive outside this regime.
Inside the admissible interval, we observed a single cycle of REB--OSC--STA in the negative $\varepsilon$ region, and similarly, in the positive region for a smaller $d$. The transition between the STA and OSC was caused by a local Hopf bifurcation around the stationary pulse, so the fundamental issue was caused by pinning--depinning processes such as PEN--OSC, REB--OSC, and STA--REB. Although only a single cycle is observed in Fig.\ref{phase_PDE.eps} owing to the admissibility, infinitely many cycles appeared in the reduced ODE system, as discussed in Section 5. 

Upon carefully observing the spatiotemporal behaviors of TPO in each regime (upper simulations in Fig \ref{m.eps}), the location of the pinning or depinning point was shifted as $|\varepsilon|$ increased. This location corresponds to the bump center in the first OSC regime immediately after the PEN ($\varepsilon = 0.0008$, refer to (f)), but it shifted leftward in the following OSC regime ($\varepsilon = 0.0101$, see (i)). The same thing happened for the STA and REB regimes as $\varepsilon$ varied, i.e., a higher bump height shifted the occurrence of the event from the bump. This arises from the oscillatory nature of the tail and the repetition of pinning (OSC and STA). Moreover, the depinning (REB) processes occurred till the $\varepsilon$ remained inside the admissible interval. In context,  the mathematical mechanism causing the transition from one state to another should be studied, including the mechanism required for predicting the outcome for a given height $\varepsilon$. In the following two subsections, we explore the global behavior of HIOP with respect to $\varepsilon$ and determine its relevance based on these questions.

\begin{figure*}
  \centering
  \includegraphics[width=16cm]{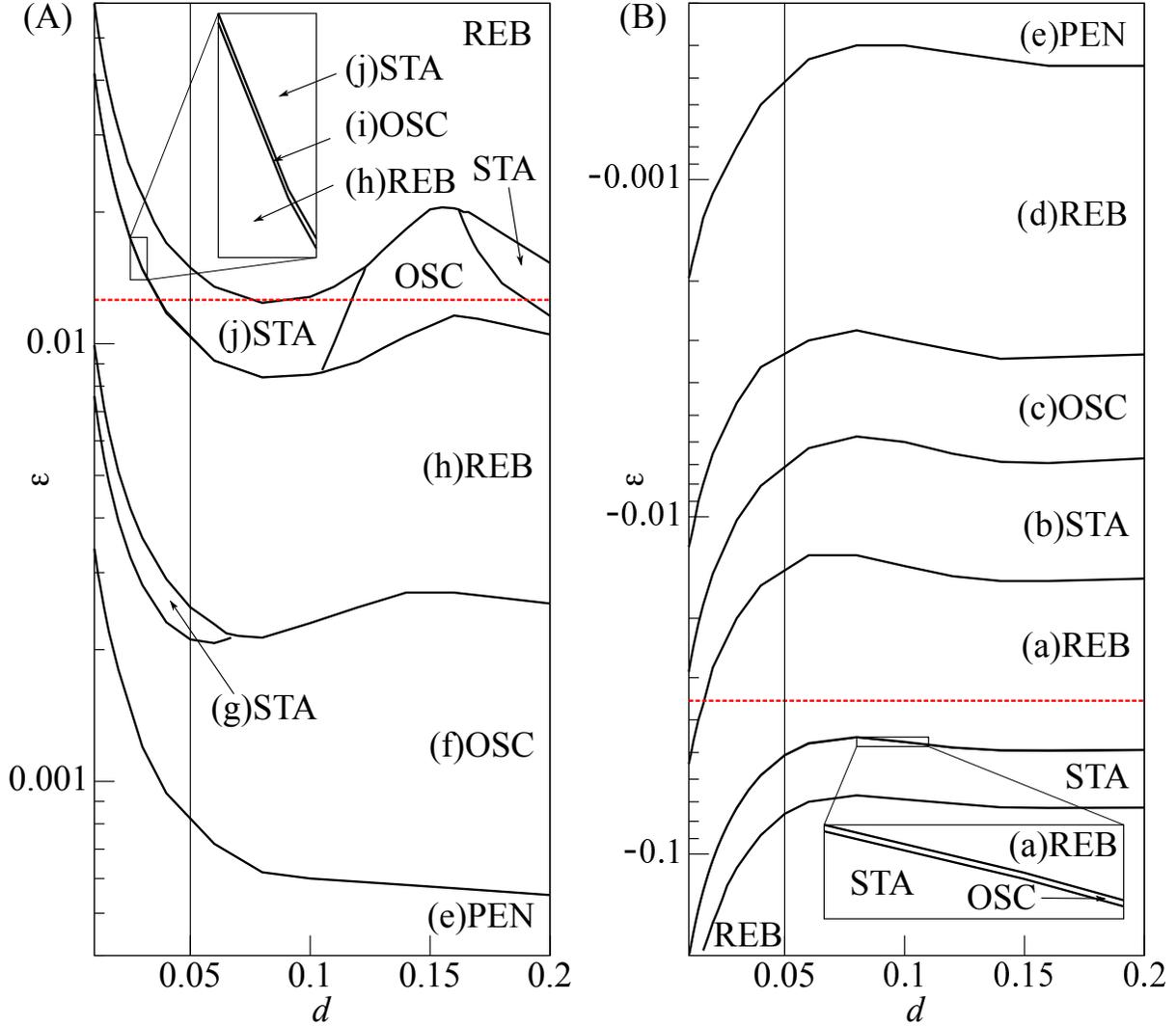}
  \caption{Phase diagram of PDE dynamics in ($d,\varepsilon$)-space.
    Vertical $\varepsilon$ axis is plotted in logarithmic scale.
    (A) $0.0004\leq\varepsilon\leq0.06$, $0.01\leq d\leq0.2$.
    (B) $-0.2\leq\varepsilon\leq-0.0003$, $0.01\leq d\leq0.2$.
    Extremely narrow OSC regions exist between REB and STA regions
    (insets display magnified pictures).
    Stable traveling pulse with oscillatory tails exist within the range
    $-0.035866\le \varepsilon\le 0.012$, as depicted with red-dotted lines.
    Beyond this range, stable traveling pulse did not exist in homogeneous space, and more complex behaviors emerged depending on the bump length. However, for small width $d$ such as $d=0.05$, the pulse persisted and responded against the bump. Refer to Fig.\ref{m.eps} for the spatiotemporal plot in each regime.
    Notably, the uppermost region of (a) is classified as REB, and this region exhibits more complex behavior, because an additional pulse emerged around the bump center, as discussed at the end of this section.
  }
  \label{phase_PDE.eps}
\end{figure*}

\subsection{Global structure of HIOP: peak-creation and peak-destruction}

The HIOP denotes the set of all solutions caused by the bump heterogeneity.
We explored all the relevant solution branches by numerical tracking
with respect to the height $\varepsilon$, starting from appropriate seeds
in homogeneous space.
In addition, we solved Eq.(\ref{dzdt}) by deriving from Eq.(\ref{eq_bif_2}) that depends on the parameters $\underline{\kappa}_{1}$ and $\varepsilon$. Herein, $\underline{\kappa}_{1}$ was maintained at $-0.1$ and $\varepsilon$ was treated as a bifurcation parameter.
The proposed strategy is three-fold: first, we used the stationary solutions of Eq.(\ref{eq_bif_2}) at $\underline{\kappa}_{1}=-0.1$ (vertical line in Fig.\ref{snakes_and_ladders.eps}) as seeds of continuation, which contained the cross-section of snakes-and-ladders structure including the trivial constant state. A schematic of the extension of the snakes-and-ladders solution to the HIOP is illustrated in Fig.\ref{k1_epsilon.eps}. Second, an additional category of seeds pertained to the class of shifted solutions depending on the peak type, as portrayed in the leftmost column of Fig.\ref{all.eps}-(b), i.e., countably many shifted solutions stagnated at the bump emerged for nonzero $\varepsilon$, which was similar to the spoke pattern depicted in Fig.\ref{all.eps}-(b) (bottom). A complete schematic of the HIOP solutions emanating from those seeds is illustrated in Fig.\ref{all.eps}-(a), wherein various colors indicate the shifted solutions and their continuation. In particular, each center of the spoke structure located at $\varepsilon=0$ is magnified and presented in Fig.\ref{all.eps}-(b), depending on the type of peak structure such as $\mathrm{[I]}^\mathrm{S}_n$. Note that the cross-section of HIOP at $\varepsilon=0$ contained several types of solutions that were distinct from those existing in the snakes-and-ladders structure, and we will discuss it later in this section. Third, we tracked the time-periodic solutions such as PEN and REB on a circle. Interestingly, each time-periodic branch was in contact with one of the stationary HIOPs as its period reached $+\infty$ in form of an homoclinic orbit. Consequently, it became an appropriate part of the connected component of HIOP. Ultimately, the number of the shifted solutions such as $\mathrm{[I]}^\mathrm{S}_n$, $\mathrm{[Ii]}^\mathrm{S}_n$, $\mathrm{[I^2]}^\mathrm{S}_n$, $\mathrm{[I^2i]}^\mathrm{S}_n$ increased proportionally with the system size. In Fig.\ref{snakes_and_ladders.eps}, the system size was $0.56$ and the maximum number of peaks was 4. Regardless, we believe this restriction did not influence in understanding the essence of entire diagram for larger or infinite system sizes. This is partially justified in Fig.\ref{1.68.eps}, wherein the selected system size was three-times larger, $1.68$. In other PDE simulation, the system size was selected as $1$.

\begin{figure}
  \begin{center}
    \includegraphics[width=.7\hsize]{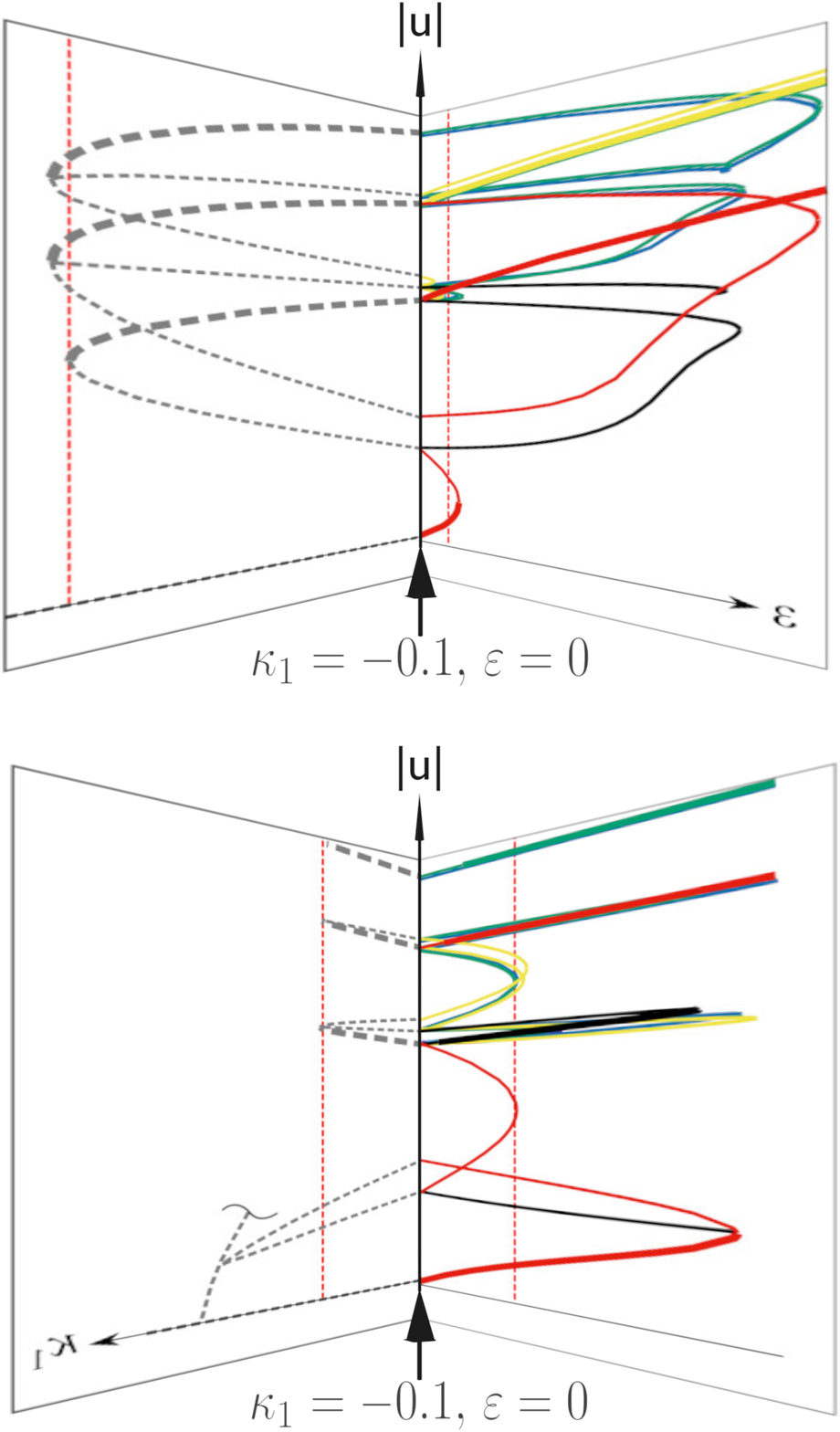}
    \caption{Schematic of relationship between snakes-and-ladders solution branches and HIOP solution branches. These two solution structures converge at $\varepsilon=0$ and $\kappa_1 = -0.1$.
      Each branch of HIOP can be extended to nonzero $\varepsilon$ by using snakes-and-ladders solution as a seed. Interestingly, similar hierarchical structure of HIOP branches was observed, as discussed in the text.
    }\label{k1_epsilon.eps}
  \end{center}
\end{figure}

Although several types of stationary solutions existed with varying number of peaks, we focused on the one-peak-pulse solution, because the research interest pertains to the dynamical behavior of one-peak traveling pulse solutions in the heterogeneous media.
Despite this constraint, countably many stationary solutions existed owing to the loss of translation invariance for an infinitely large system size. In particular, we are interested in a family of solution branches that converge to one of the translated one-peak-pulse solutions
$\mathrm{[I]}^\mathrm{S}_n$ ($n = 0, \pm 1, \pm 2, \dots$) as $\varepsilon \rightarrow 0$ in Fig.\ref{all.eps}-(b). The shift-distance between $\mathrm{[I]}^\mathrm{S}_0$ and
$\mathrm{[I]}^\mathrm{S}_n$ can be determined using a solvability condition for the bump heterogeneity upon applying a perturbative argument to it; however, this was not detailed further.


In this study, we focused on the five shifted solutions $\mathrm{[I]}^\mathrm{S}_0$ -- $\mathrm{[I]}^\mathrm{S}_{-4}$ and explored the global behavior with respect to $\varepsilon$. As a by-product, multi-peak-pulse solutions appeared through the deformation along the branch. These solutions yielded the candidates for the asymptotic behavior of one-peak traveling pulse, as the height $\varepsilon$ was in the vicinity of the boundary of the admissible interval, as discussed at the end of this section.
In context, a perspective of global behaviors of shifted solutions $\mathrm{[I]}^\mathrm{S}_n$ is briefly presented prior to its detailed discussion. As explained below, $\mathrm{[I]}^\mathrm{S}_n$ is a shifted solution of $\mathrm{[I]}^\mathrm{S}_0$ and can be constructed near $\varepsilon = 0$ in a perturbative manner. However, the global behavior of $\mathrm{[I]}^\mathrm{S}_n$ is unclarified, especially, its relation to the class of multi-peak solutions. Recall that the snakes-and-ladders structure in the homogeneous space explains the manner in which the number of peaks increased or decreased via saddle-node points to enable the inherent transmission of such a manner to the behavior of solution branches in the heterogeneous space. In particular, the snakes-and-ladders structure displayed footprints on the axis $\varepsilon = 0$, and the branches emanating from those seeds are supposed to interact with the continuation of $\mathrm{[I]}^\mathrm{S}_n$ branch. As such, the creation and destruction of peaks can be observed in the heterogeneous problem. Although this is true for this case, the situation is slightly more complicated than the snakes-and-ladders structure, because countably-many shifted solutions emerged for each seed solution at $\varepsilon=0$, discussed as follows.

\begin{figure}
  \begin{center}
    \includegraphics[width=\hsize]{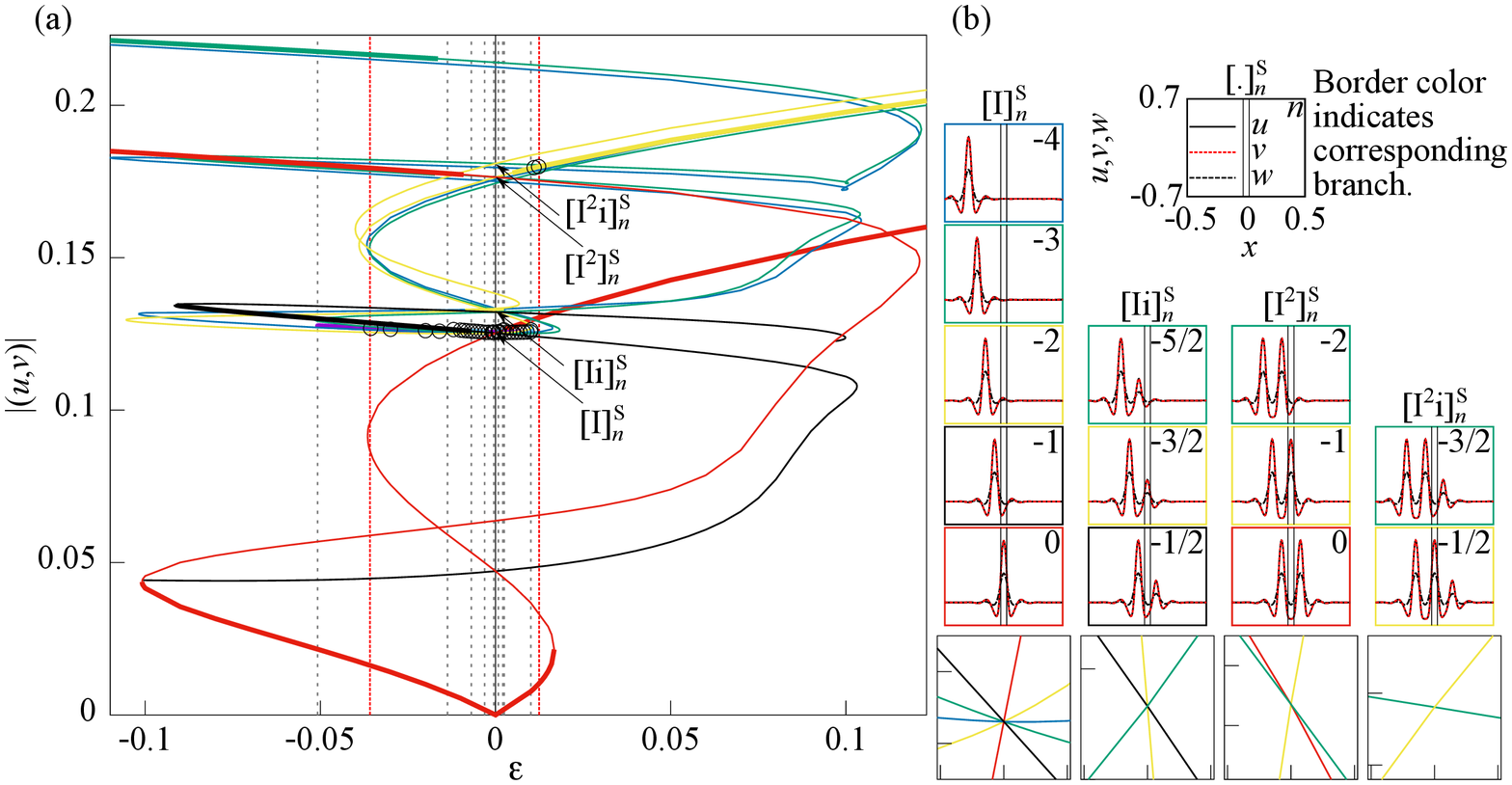}
    \caption{
      (a) Complete structure of HIOP and solution profiles. Red-dotted vertical lines indicate boundaries of admissible interval and thin-black dotted lines indicate boundaries of various asymptotic states (Fig.\ref{m.eps} presents a magnified view).
      Open circles indicate asymptotic states of collision for each $\varepsilon$.
      The branch color indicates the corresponding shifted one-peak solution: red ($\mathrm{[I]}^\mathrm{S}_{0}$), black ($\mathrm{[I ]}^\mathrm{S}_{-1}$), yellow ($\mathrm{[I]}^\mathrm{S}_{-2}$), green ($\mathrm{[I]}^\mathrm{S}_{-3}$), blue ($\mathrm{[I]}^\mathrm{S}_{-4}$). 
      (b) Shifted solutions are classified depending on the peak structure with barcode labeling. Five shifted single-peak solutions existed from $\mathrm{[I]}^\mathrm{S}_{0}$ to $\mathrm{[I]}^\mathrm{S}_{-4}$, and they formed a spoke pattern near $\varepsilon = 0$, as displayed at bottom. Remaining multipeak solutions displayed similar spoke patterns as well. All these solutions $\mathrm{[I]}^\mathrm{S}_{-n}$, $\mathrm{[Ii]}^\mathrm{S}_{n}$, $\mathrm{[I^2]}^\mathrm{S}_{n}$, $\mathrm{[I^2i]}^\mathrm{S}_{n}$ coincided with the associated solutions of snakes-and-ladders structure at $\varepsilon = 0$. Interrelation between varying  peak solutions is discussed in text. Note that $\mathrm{[I]}^\mathrm{S}_n$ in PDE corresponds to the critical point $P_n$ in the reduced ODE system, as detailed in Section 5.
    }\label{all.eps}
  \end{center}
\end{figure}

\begin{figure}
  \begin{center}
    \includegraphics[scale=.75]{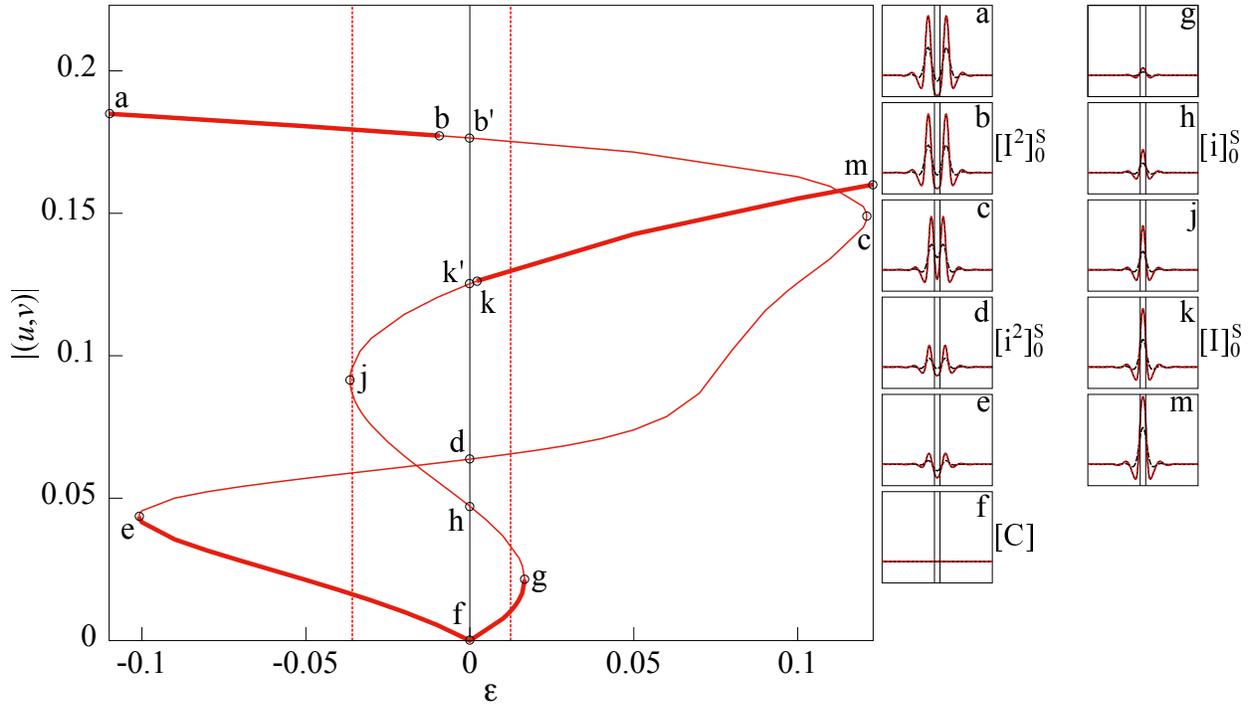}
    \caption{ Bifurcation of HIOP solutions from trivial state (red). Thick portion indicates stable branches. Two types of symmetric solutions bifurcated from trivial solution as $\varepsilon$ was varied. One is a single-peak solution ($\varepsilon > 0$) and the other two-peak pulse ($\varepsilon < 0$), which constitute as stable background states for nonzero $\varepsilon$ up to (g) and (e). Their stabilities were recovered after experiencing two saddle-node bifurcations, and they produced a stable, large one-peak (k) and two-peak pulse (b) solutions, respectively. 
    }
    \label{S0.eps}
  \end{center}
\end{figure}

First, the primary bifurcation from the trivial constant solution was studied for an adequately small absolute value of $\varepsilon$ (Fig.\ref{S0.eps}).
As the bump increases ($\varepsilon>0$), a one-peak HIOP emerges
from the trivial solution (Fig.\ref{S0.eps}-(g)). In contrast, a two-peak HIOP solution emerges for small $\varepsilon <0$ owing to the hollow space in between (Fig.\ref{S0.eps}-(e)). 
These peakes monotonically increased along the branch and experienced the first saddle-node bifurcation ($\varepsilon\approx-0.1$(e) and $\varepsilon\approx 0.02$(g)), and subsequently, the second one ((j) and (c)), and ultimately,
became stable one- and two-peak solutions, respectively ((k) and (b)).
Notably, these one- and two-peak solution branches crossed the line
$\varepsilon=0$ on five instances, and these five points correspond to the 
snakes-and-ladders branches within the homogeneous space (vertical line in Fig.\ref{S0.eps} at $\varepsilon=0$).
Moreover, these five intersections were $\mathrm{[i]}^\mathrm{S}_0$ (h), $\mathrm{[i^{2}]}^\mathrm{S}_0$ (d), $\mathrm{[I]}^\mathrm{S}_0$ (k$'$) and $\mathrm{[I^{2}]}^\mathrm{S}_0$ (b$'$), in addition to the trivial constant state $\mathrm{[C]}$ (f) in ascending order of magnitude $|(u,v)|$. Subsequently, $'$ indicated an intersecting point of a solution branch with $\varepsilon=0$, and we used the same alphabet as the neighboring recovery--stability point without $'$.

The expressions $\mathrm{[I]}^\mathrm{S}_{-1}$ - $\mathrm{[I]}^\mathrm{S}_{-4}$ are shifted solutions of
$\mathrm{[I]}^\mathrm{S}_0$ and coincide with each other at $\varepsilon=0$ modulo translation. In particular, their emergence and variation trend with $\varepsilon$ is explained further. We focused on only the negatively shifted solutions without any loss of generality. The basic mechanism for the emergence of $\mathrm{[I]}^\mathrm{S}_{-n}$ involves peak-destruction and peak-creation processes via the interaction between the tail and heterogeneity. The peak-destruction signifies a norm-descending process from a multipeak solution to a single-peak solution in which each peak of the multipeak solution successively disappeared via turning around the saddle-node bifurcations. The peak-creation indicated a norm-ascending process from a single peak solution to a multi-peak solution, wherein a new peak appeared around the bump region as $\varepsilon$ increased. 
Furthermore, each process is detailed as follows.

\begin{figure}
  \begin{center}
    \includegraphics[scale=.75]{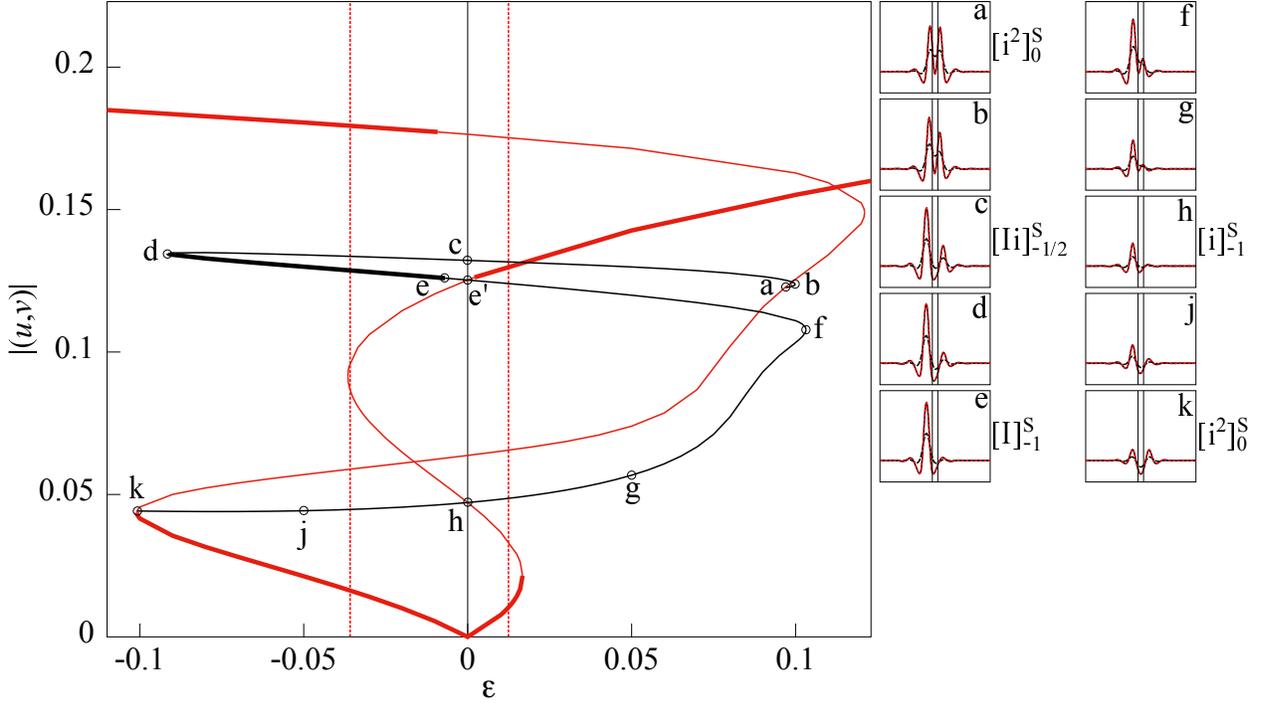}
    \caption{Bifurcation of HIOP solutions connecting $\mathrm{[I]}^\mathrm{S}_{-1}$ branch (black).
      A left-shifted solution $\mathrm{[I]}^\mathrm{S}_{-1}$ (e) is created via symmetry-breaking process from two-peak solution $\mathrm{[i^2]}^\mathrm{S}_{0}$ (k). Similarly, a symmetric two-peak solution $\mathrm{[i^2]}^\mathrm{S}_{0}$ (a) is deformed into $\mathrm{[I]}^\mathrm{S}_{-1}$ through non-symmetric solution $\mathrm{[Ii]}^\mathrm{S}_{-1/2} (c)$ that acts as a footprint of snakes-and-ladders structure.
    }\label{S1.eps}
  \end{center}
\end{figure}

\subsubsection{$\mathrm{[I]}^\mathrm{S}_{-1}$-branch}

As observed in Fig.\ref{S1.eps}, $\mathrm{[I]}^\mathrm{S}_{-1}$ ((e) and (e')) is present on the deformed S-shaped branch connecting the two symmetry-breaking bifurcations located at (a) and (k). Specifically, these were emanated from a symmetric two-peak branch centered at the middle of the bump. In addition, the pitchfork bifurcation (k) was located proximate to the saddle-node point and the two small peaks were identical at (k). Subsequently, one of the two peaks diminished, whereas the other one increased. The smaller peak disappeared and coincided with the translation of the small one-peak-pulse of $\mathrm{[i]}^\mathrm{S}_0$ at $\varepsilon = 0$, i.e., $\mathrm{[i]}^\mathrm{S}_{-1}$ . An additional symmetry-breaking bifurcation occurred at (a) on the small two-peak branch $\mathrm{[i^2]}^\mathrm{S}_{0}$, and the right-hand peak becomes smaller along (c) and (d), and eventually, it disappears at (e'). Nonetheless, the left peak increased and the $\mathrm{[I]}^\mathrm{S}_{-1}$ was formed at that location. The branch $\mathrm{[I]}^\mathrm{S}_{-1}$ was stable between (d) and (e), which corresponded to a stationary pinning state, as discussed later.
The rule for naming the noninteger suffix attached to a mixed type of pulse solution is explained herein. For instance, $\mathrm{[Ii]}^\mathrm{S}_{-1/2}$ in Fig.\ref{S1.eps}, which corresponds to a ladder solution in Fig. \ref{snakes_and_ladders.eps} connecting $\mathrm{[i^2]}^\mathrm{S}_{0}$ and
$\mathrm{[I]}^\mathrm{S}_{-1}$. Therefore, the suffix of $\mathrm{[Ii]}^\mathrm{S}$ yielded as $-1/2=(0+(-1))/2$.
In general, suppose $\mathrm{[I^{2m-1}]}^\mathrm{S}_{n}$ becomes $\mathrm{[I^{2m}]}^\mathrm{S}_{n-1}$ along the branch via $\mathrm{[I^{2m-1}i]}^\mathrm{S}$, then the suffix of the pulse solution $\mathrm{[I^{2m-1}i]}^\mathrm{S}$ can be defined as $\mathrm{[I^{2m-1}i]}^\mathrm{S}_{n-1/2}$ (Fig.\ref{S3.eps} for $m=1, n=-2$).
The latter mechanism of emergence of $\mathrm{[I]}^\mathrm{S}_{n}$ solution is essential for $n=-2, -3, -4$, which is elaborated in the following passage.

\begin{figure}
  \begin{center}
    \includegraphics[scale=.75]{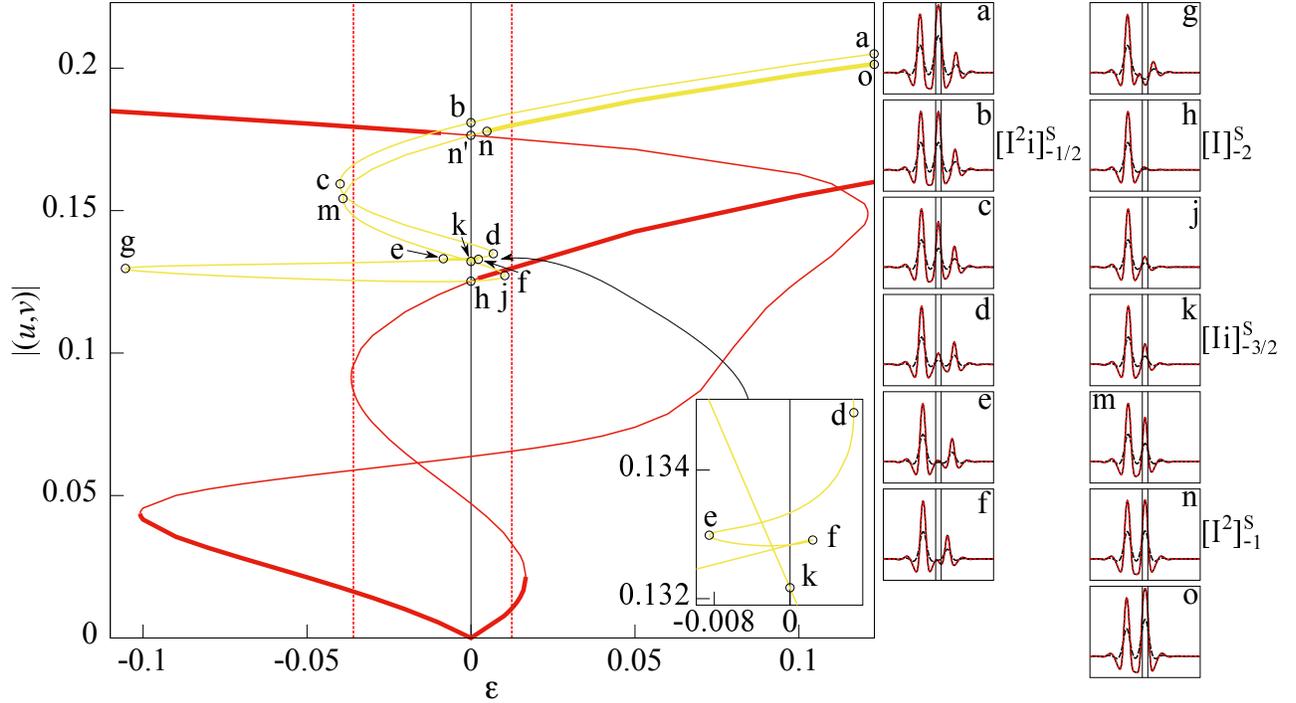}
    \caption{Bifurcation of HIOP solutions connecting $\mathrm{[I]}^\mathrm{S}_{-2}$ branch (yellow). Left-shifted solution $\mathrm{[I]}^\mathrm{S}_{-2}$ (h) was connected to (g), (f), and eventually, continued to a two-large and one-small peak solution $\mathrm{[I^2i]}^\mathrm{S}_{-1/2}$ (b). Similarly, $\mathrm{[I]}^\mathrm{S}_{-2}$ (h) was connected to (j), $\mathrm{[Ii]}^\mathrm{S}_{-3/2}$ (k), and ultimately, $\mathrm{[I^2]}^\mathrm{S}_{-1}$. Route (a)-(b)-$\dots$-(h) denoted a peak-annihilation process and (h)-(j)-$\dots$-(n)-(o) as a peak-creation process.
    }\label{S2.eps}
  \end{center}
\end{figure}

\subsubsection{$\mathrm{[I]}^\mathrm{S}_{-2}$-branch}

Furthermore, let us consider the connected component of the two-shifted solution $\mathrm{[I]}^\mathrm{S}_{-2}$ (h) (yellow) presented in Fig.\ref{S2.eps}.
This branch intersected with the snakes-and-ladders structure on
four hubs of translation, specifically, one-peak (h), ladders (k), (b), and two-peak (n$'$) solutions.
Although the branch crossed $\varepsilon=0$ at other points,
between (c) and (d), (d) and (e), (e) and (f), and (f) and (g),
these intersections were not located on the snakes-and-ladders branch. The distinction between these two classes emerged from the arrangement among peaks. In particular, the ladder solutions comprised several large peaks and one small peak arranged in a uniform and nearest distance, as displayed in
Figs. \ref{S1.eps}-(c), \ref{S2.eps}-(b),  \ref{S2.eps}-(k), \ref{S3.eps}-(e),
and \ref{S3.eps}-(j), and thus, they are expressed as $\mathrm{[I^{n}i]}^\mathrm{S}_{*}$.
In contrast, solution (o) in Fig.\ref{S3.eps} does not belong to this category, because the distance between the small peak and neighboring large peak is not the same as that of the ladder solutions $\mathrm{[I^{n}i]}^\mathrm{S}_0$. 
Let us follow the solutions along the branch from (a):three-peak-type to (o):two-peak-type. The branch (a)--(c) emerges from the hub of $\mathrm{[I^{2}i]}^\mathrm{S}_{-1/2}$, and it turns around the saddle-node bifurcation points (c) and (d). Moreover, the peak at
the center diminishes and eventually disappears at (e).
The solutions (e), (f), and (g) comprise a large peak and a small peak;
however, the interval between them does not correspond to that of the ladder solution (k).
Turning around the saddle-node bifurcation point (g), the small peak on the right-hand side diminished
and disappeared, before emerging as a two-shifted one-peak solution $\mathrm{[I]}^\mathrm{S}_{-2}$ (h).
Thereafter, the branch started to ascend in norm and experienced a peak-creation. 
This new peak emerged around the bump, grows significantly, and became a ladder solution at (k).
This peak continued growing (m) and the solution became one-shifted two-peak solution $\mathrm{[I^{2}]}^\mathrm{S}_{-2}$ at (n).
The suffix of the ladder solution (k) should be -3/2, as depicted from (h) to (n), because
$\left(\mathrm{[I]}^\mathrm{S}_{-2}+\mathrm{[I^2]}^\mathrm{S}_{-1}\right)/2=\mathrm{[Ii]}^\mathrm{S}_{-3/2}$.\\

\subsubsection{$\mathrm{[I]}^\mathrm{S}_{-3}$-branch}

A continuous deformation of $\mathrm{[I]}^\mathrm{S}_{-3}$ branch (green) from (a) to (q) is presented in Fig.\ref{S3.eps}.
The $\mathrm{[I]}^\mathrm{S}_{-3}$ branch resembled the $\mathrm{[I]}^\mathrm{S}_{-2}$ branch, and it intersected with snakes-and-ladders structure on the five hubs of translation: One-peak (m),two-peak (g$'$),three-peak (a$'$), and ladders (e), (j) solutions.
Although (o) and (q) crossed $\varepsilon=0$, they were not on
the snakes-and-ladders structure owing to the distinction of width of two peaks.
The variations in the number of peaks of the pulse can be easily understood by focusing on
saddle-node bifurcation points along the branch from (a) to (q).
The rightmost peak of the three-peak pulse (a) became smaller, coincided with the shifted ladder solution (e), and completely disappeared at the two-peak pulse $\mathrm{[I^2]}^\mathrm{S}_{-2}$ (g) via the saddle-node points (b) and (f). Upon continuing the branch downward, the right-hand peak of (g) became smaller and coincided with the shifted ladder solution (j). Thereafter, it completely disappeared at the shifted one-peak solution $\mathrm{[I]}^\mathrm{S}_{-3}$ (m) via two saddle-nodes points (h) and (k). Passing the point (m), the branch started to grow in amplitude owing to the emergence of new peak around the bump (n). This peak became larger, attained the same height as the peak on the left-hand side, and became a two-peak pulse at (q), which was not the same as $\mathrm{[I^2]}^\mathrm{S}_{-2}$ (g), because the middle vacancy was not fulfilled after the immediately preceding disappearing process, so the width between the two peaks of (q) was wider than that of (g). Therefore, (q) did not belong to the cross-section of the snakes-and-ladders at $\varepsilon=0$. 

\begin{figure}
  \begin{center}
    \includegraphics[scale=.75]{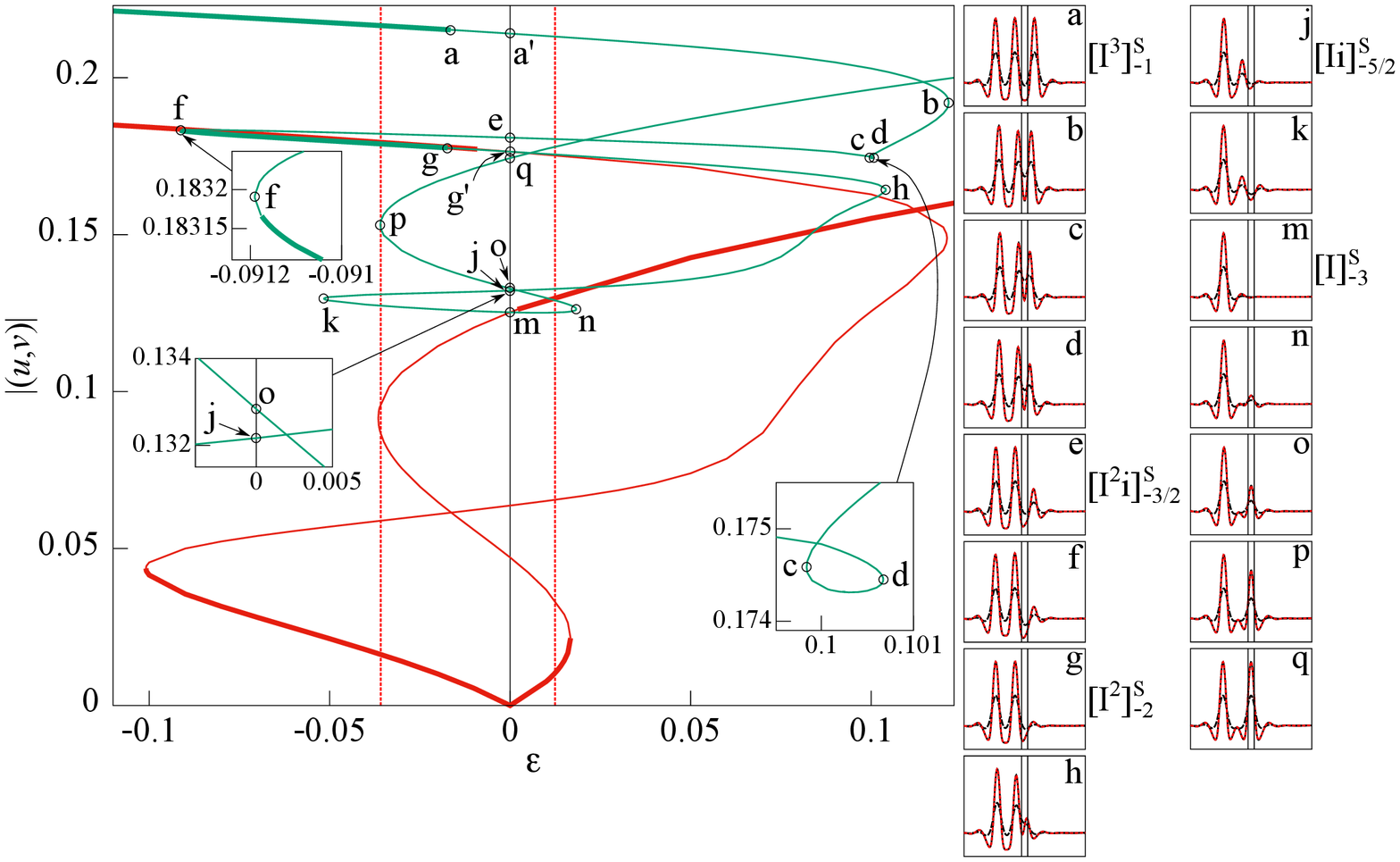}
    \caption{Bifurcation of HIOP solutions connecting to $\mathrm{[I]}^\mathrm{S}_{-3}$ branch (green). Starting from three-peak solution $\mathrm{[I^3]}^\mathrm{S}_{-1}$ (a), branch can be continued to $\mathrm{[I^2i]}^\mathrm{S}_{-3/2}$ (e), $\mathrm{[I^2]}^\mathrm{S}_{-2}$ (g), $\mathrm{[Ii]}^\mathrm{S}_{-5/2}$ (j), and ultimately, $\mathrm{[I]}^\mathrm{S}_{-3}$ (m), which represents a successive annihilation process of peaks. The solution $\mathrm{[I]}^\mathrm{S}_{-3}$ can be extended to the right-hand side and attained a two-peak solution (q), which is a peak-creation process around the bump region. Note that the distance between the two peaks of (q) was larger than that of $\mathrm{[I^2]}^\mathrm{S}_{-2}$ (g), and therefore, it did not account for a footprint of snakes-and-ladders structure, even though it was located on $\varepsilon=0$. Three magnified diagrams display the details of local structure.}
    \label{S3.eps}
  \end{center}
\end{figure}

\subsubsection{$\mathrm{[I]}^\mathrm{S}_{-4}$-branch}

Finally, the deformation of the $\mathrm{[I]}^\mathrm{S}_{-4}$ branch (blue)  via peak-destruction and peak-creation process from (a) to (q) is depicted in Fig.\ref{S4.eps}.
This branch intersected with the snakes-and-ladders on only one hub of the one-peak
solution $\mathrm{[I]}^\mathrm{S}_{-4}$.
Although the shape of the pulse on each branch appeared complicated, the route to $\mathrm{[I]}^\mathrm{S}_{-4}$ was qualitatively the same as that of the $\mathrm{[I]}^\mathrm{S}_{-3}$, i.e., a successive peak-destruction process. Starting from a three-peak solution (a) (note that the leftmost peak was located at the same position of $\mathrm{[I]}^\mathrm{S}_{-4}$), the middle-peak disappeared at (g) via the four saddle-node points (b), (c), (d), and (f). Subsequently, the rightmost peak diminished through the two saddle-node points (h) and (k), and eventually, it disappeared at (m), i.e., $\mathrm{[I]}^\mathrm{S}_{-4}$. Thereafter, the branch moved upward and the norm was increased, because a new peak emerged at the bump region, increased, and became a 2-peak solution (q) via two saddle-node points (n) and (p). In conclusion, the $\mathrm{[I]}^\mathrm{S}_{-4}$ was obtained via the destruction process of a multipeak solution with the leftmost peak located at the same position of $\mathrm{[I]}^\mathrm{S}_{-4}$. All these shifted solutions were unstable in the neighborhood of $\varepsilon=0$, but a few of them recovered their stabilities for a finite nonzero $\varepsilon$. In addition, they may serve as stationary pinning states.

\begin{figure}
  \begin{center}
    \includegraphics[scale=.75]{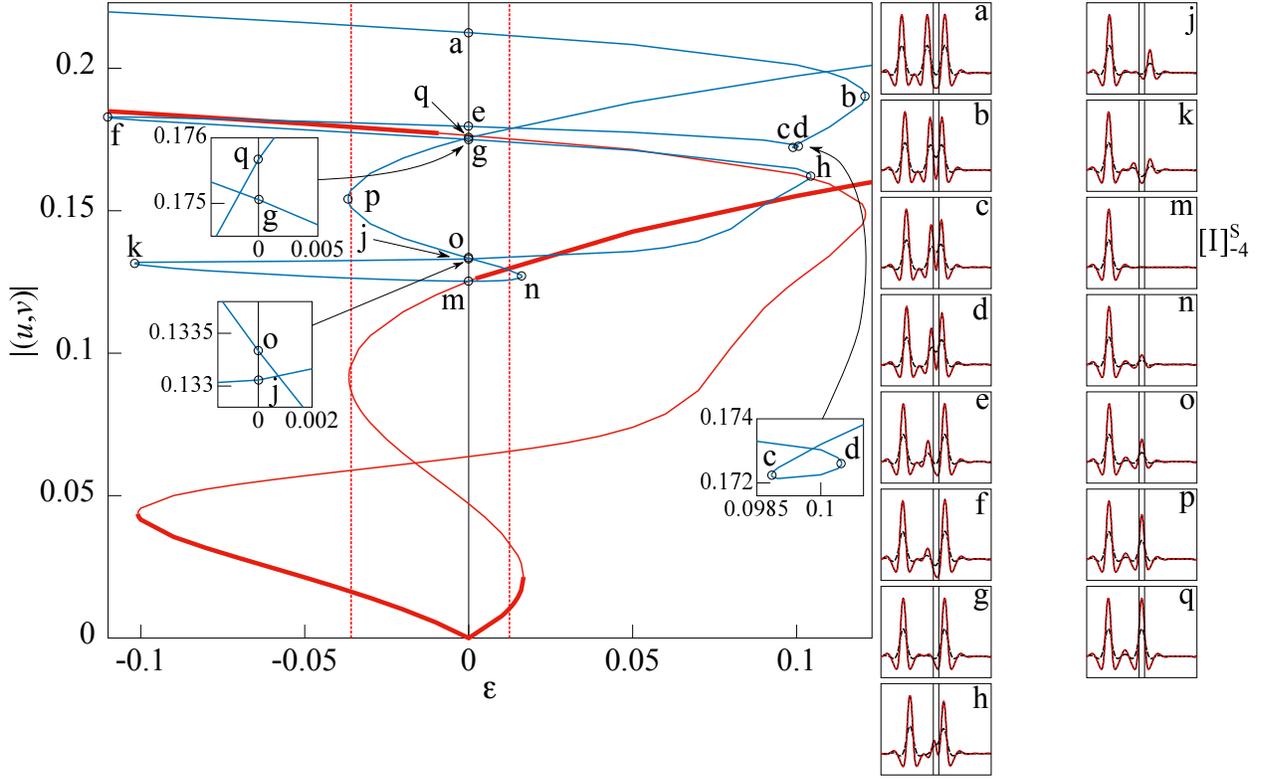}
    \caption{Bifurcation of HIOP solutions connecting to $\mathrm{[I]}^\mathrm{S}_{-4}$ branch (blue). Left-shifted one-peak solution $\mathrm{[I]}^\mathrm{S}_{-4}$ (m) can be obtained via peak-annihilation process from (a) to (m). It continued rightward and reached a two-peak solution (q) via two saddle-node points (n) and (p). Note that the three-peak solution (a) was not the same as $\mathrm{[I^3]}^\mathrm{S}_{-1}$ in Fig.\ref{S3.eps}. Similarly the two-peak solution (q) was not the same as $\mathrm{[I^2]}^\mathrm{S}_{-1}$ in Fig.\ref{S2.eps} because of the difference of width between two peaks.
    } 
    \label{S4.eps} 
    \end{center}
\end{figure}

In summary, each connected component containing $\mathrm{[I]}^\mathrm{S}_{n} (n=-1,-2,-3,-4)$ shared a similar mechanism of creating these shifted solutions. Thus, it is a peak-destruction process from a multipeak solution in normal descending order, especially the leftmost peak of the multipeak solution maintains the same location as that of the relevant $\mathrm{[I]}^\mathrm{S}_{n}$, and all other peaks disappeared through peak-destruction process, thereby resulting in $\mathrm{[I]}^\mathrm{S}_{n}$.
fter the destruction process, a peak-creation process followed at the bump region and formed a two-peak pulse solution; however, the peak-width in this case was wider than that of $\mathrm{[I^2]}^\mathrm{S}_{0}$. Therefore, this did not belong to the footprints of snakes-and-ladders at $\varepsilon = 0$. Furthermore, a complete characterization of the destruction--creation process for general $n$ remained an open question, and we believed that the observations present an essential mechanism of creating the shifted pulse solution.

\subsection{Characterization of asymptotic states in terms of HIOP}

\begin{figure}
  \begin{center}
    \includegraphics[width=.9\hsize]{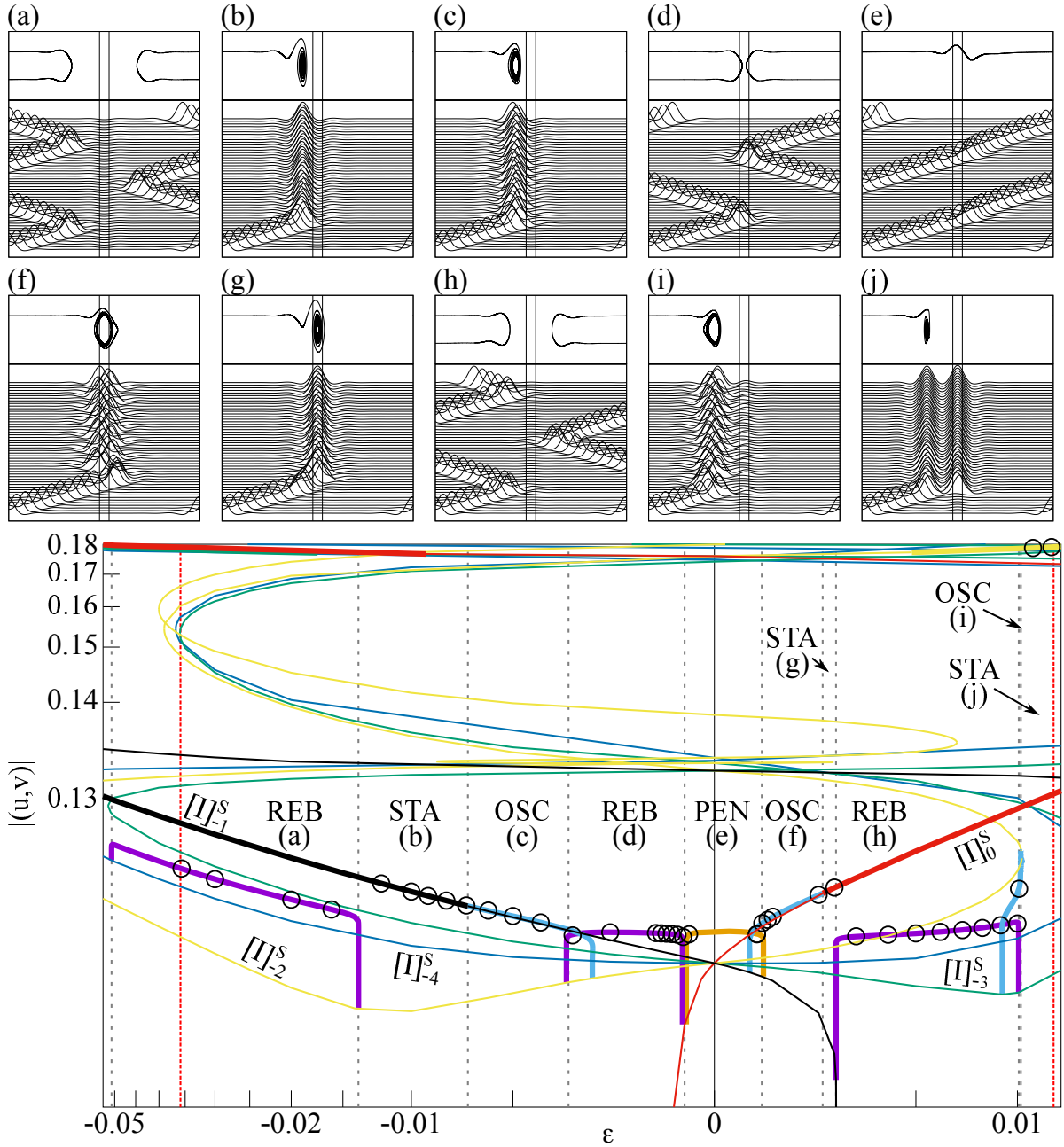}
    \caption{Global bifurcation diagram of HIOP in admissible interval.
      Solution set of all HIOP emanated from the trivial solution
      contained all asymptotic states after collision against bump. We regarded all the shifted one-peak solutions at $\varepsilon = 0$ as identical, because they coincided with each other by translation (spork structure). Time-periodic solutions PEN, REB, and OSC are indicated by ocher, purple, and cyan, respectively. All these branches exhibited similar ``staple''-like shape at both ends except the Hopf bifurcation of OSC branch at the transition from STA to OSC. Homoclinic/heteroclinic asymptote was observed at the end of staple structure, as detailed in text. Top panels depict PDE simulations, wherein the upper half displays the projection to (location, velocity)--space. Note that both vertical and horizontal axes of the diagram are not uniform.}\label{m.eps}
  \end{center}
\end{figure}

The relationship between asymptotic states of TPO after collision with
bump heterogeneity and all the relevant HIOP solutions
including time-periodic solutions is outlined in Fig.\ref{m.eps},
wherein the relevant HIOP branches are depicted using solid-colored
lines---thick lines indicate stable solutions, whereas thin lines indicate
unstable ones.
The asymptotic state for each $\varepsilon$ after collision is presented
as a circle in the bifurcation diagram.
Note that all the circles exist on one of the thick HIOP branches,
i.e., HIOP captured all the asymptotic states.

The time-periodic solutions PEN, REB, and OSC are indicated by
ocher, purple, and cyan, respectively.
Each end of PEN and REB branches contacted one of the shifted solutions
$\mathrm{[I]}^\mathrm{S}_n$, which corresponded to a homoclinic/heteroclinic bifurcation
to $\mathrm{[I]}^\mathrm{S}_n$ (see Section 5 for precise meaning).
On the other hand, an end of OSC branch is connected to one of
$\mathrm{[I]}^\mathrm{S}_n$ via Hopf bifurcation and the other end, a homoclinic bifurcation, and only for the case of PEN to OSC
transition, heteroclinic bifurcation from $\mathrm{[I]}^\mathrm{S}_{-1}$ to $\mathrm{[I]}^\mathrm{S}_{1}$.
The red vertical lines at $\varepsilon=-0.035866$ and $\varepsilon=0.012$
denote the lower and upper limits of the admissible interval,
and the response of TPO was classified as (a)--(j).
The associated spatiotemporal plots of TPO are displayed in the upper half
of Fig.\ref{m.eps}.
Although we studied in the admissible interval, the HIOP branches can be
tracked over it as already displayed in Figs.\ref{all.eps}, \ref{S0.eps},
\ref{S1.eps}, \ref{S2.eps}, \ref{S3.eps}, and \ref{S4.eps}.
The color legend of the steady solutions were the same as these figures.

Transitions for positive $\varepsilon$ is explained first,
and for negative $\varepsilon$ treated later.

\textbf{PEN$\rightarrow$OSC; (e)$\rightarrow$(f)}\\
As $\varepsilon$ was extremely small, the asymptotic state was the PEN
depicted in the thick-solid ocher line in domain (e).
The right extremity of the PEN branch bifurcated from the
$\mathrm{[I]}^\mathrm{S}_{-1}$ branch (black) via the saddle asymptote.
Slightly before this bifurcation point, the OSC solution depicted using 
cyan line bifurcated from the same branch.
The magnified figure of these bifurcations are shown in
Fig.\ref{PEN-OSC.eps} (c).
The OSC branch bifurcated via unfolding the heteroclinic orbit connecting two
saddle points $\mathrm{[I]}^\mathrm{S}_{-1}$ and $\mathrm{[I]}^\mathrm{S}_{1}$.
Although the emerged OSC solution was unstable, it recovered the stability
via the saddle-node bifurcation.
Therefore, both the OSC and PEN were stable in the range between the
saddle-node bifurcation of the OSC solution and saddle-asymptote;
however, the PEN was selected as the asymptotic state of collision.
The reason for this selection is explained based on the analysis of
reduced ODEs in Section \ref{S_Dynamics_of_reduced_ODEs_cyclic}.
The other extremity of this OSC branch denotes the Hopf bifurcation on the
$\mathrm{[I]}^\mathrm{S}_0$ branch.
Thereafter, both the PEN and OSC solution branches connected
$\mathrm{[I]}^\mathrm{S}_0$ and $\mathrm{[I]}^\mathrm{S}_{-1}$ but
in a different manner.

\textbf{OSC$\rightarrow$STA; (f)$\rightarrow$(g)}\\
The OSC branch disappears via Hopf bifurcation. After vanishing the periodic solution, the asymptotic state became STA
on $\mathrm{[I]}^\mathrm{S}_0$.

\textbf{STA$\rightarrow$REB; (g)$\rightarrow$(h)}\\
As $\varepsilon$ increased, the asymptotic state varied to REB because
an additional REB branch appeared, despite the $\mathrm{[I]}^\mathrm{S}_0$
branch was stable.
The left- and right-hand side of this REB branch terminated at
$\mathrm{[I]}^\mathrm{S}_{-1}$ and $\mathrm{[I]}^\mathrm{S}_{-3}$, respectively.
The reason for selecting the REB is explained later in terms of ODEs in
Section \ref{S_Dynamics_of_reduced_ODEs_cyclic}.

\textbf{REB$\rightarrow$OSC; (h)$\rightarrow$(i)}\\
A periodic orbit appeared via homoclinic bifurcation between
two ends of this REB solution branch.
This orbit was unstable and recovered stability after passing a
saddle-node bifurcation, which structure resembles aforementioned OSC branch.
This stable limit cycle remained after vanishing the REB branch and the
asymptotic state became OSC after REB.

\textbf{OSC$\rightarrow$STA; (i)$\rightarrow$(j)}\\
The stable OSC branch cease to exist and the asymptotic state becomes STA,
however corresponding HIOP solution is not one-peak but two-peak solution.
Further the OSC branch loses stability not via Hopf bifurcation but via
saddle-node bifurcation.
In fact the bifurcation structure of this OSC branch is very complicated
near the Hopf bifurcation point; it repeats many times of saddle-node
bifurcations before connecting to $\mathrm{[I]}^\mathrm{S}_{-2}$ branch
via Hopf bifurcation.
Although the bifurcation structure is complicated, transition process is
simple because the asymptotic state changes to STA just after the first
saddle-node bifurcation point of OSC branch where the OSC solution loses
stability.

Next, the transition process in the region of negative $\varepsilon$
is explained.

\textbf{PEN$\rightarrow$REB; (e)$\rightarrow$(d)}\\
The left end of the PEN branch contacted the $\mathrm{[I]}^\mathrm{S}_0$
branch via the saddle asymptote and the REB branch (purple) contacted
exactly at the same solution.
At this transition point, the unstable $\mathrm{[I]}^\mathrm{S}_0$ solutions
acted as the separator dividing the PEN and REB.

\textbf{REB$\rightarrow$OSC; (d)$\rightarrow$(c)}\\
The transition process is the same as that of (h) to (i) except for the
connecting branch; in this case both branch connect to
$\mathrm{[I]}^\mathrm{S}_{-2}$.
Magnified bifurcation structure is shown in Fig.\ref{OSC-REB.eps} (c).

\textbf{OSC$\rightarrow$STA; (c)$\rightarrow$(b)}\\
The transition process is the same as that of (f) to (g) except for the
connecting branch; in this case Hopf bifurcation occurs on
$\mathrm{[I]}^\mathrm{S}_{-1}$.

\textbf{STA$\rightarrow$REB; (b)$\rightarrow$(a)}\\
The transition process is the same as that of (g) to (h) except for the
connecting branch; in this case REB branch bifurcates from
$\mathrm{[I]}^\mathrm{S}_{-2}$.
Magnified bifurcation structure is shown in Fig.\ref{REB-STA.eps} (c).

In both cases, for positive and negative $\varepsilon$, the transition of
asymptotic state is almost the same except for that after PEN,
because the unstable solution located in the middle of the bump altered
the type of instability depending on the sign of $\varepsilon$;
non-saddle-type solution for positive $\varepsilon$ whereas saddle-type
for negative $\varepsilon$.
Further detailed discussion based on a dynamical system perspective is
presented in Section \ref{S_Dynamics_of_reduced_ODEs_cyclic}.

Here a difference between the analysis of the asymptotic state of collision
in terms of HIOP solution and reduced ODE dynamics in Section 5, is mentioned briefly.
The reduced dynamics is derived near the
pitchfork bifurcation point of one-peak pulse.
Therefore the analysis of reduced dynamics is valid only for the collision
of one-peak pulse to the bump.
From this point of view, asymptotic states (a)--(i) in Fig.\ref{m.eps}
can be compared with the result of the analysis in
Section \ref{S_Dynamics_of_reduced_ODEs_cyclic}, however (j) is　not
comparable because the asymptotic state has two peaks.
In this case, the TPO rebounded and invoked a transition from the background
state to another larger stationary pulse upon collision and forms a two-peak
bound state, as depicted in Fig.\ref{S2.eps} with label
$\mathrm{[I^2]}^\mathrm{S}_{-1}$ (yellow branch).
This state is exactly the same solution as Fig.\ref{S2.eps}-(n).
Although multi-pulse solutions are beyond the scope of the reduced approach, HIOP analysis is still valid for such a regime.

Lastly, we briefly remark on the behavior of TPO beyond the admissible
interval.
The behavior of TPO beyond the admissible interval is complicated.
An example is shown in Fig.\ref{0.14+0.01894.eps} for $\varepsilon=0.01894$,
which is larger then admissible interval.
The TPO collided with the bump and triggered a large-amplitude pulse,
but it was not stationary pulse and exhibited oscillations.
The behavior of TPO was almost as a REB and interacted with the breathing
motion of the pulse at the bump.
After a considerable period, the TPO and the breathing motion settled,
but they did not appear as periodic motions.
In particular, the breathing motion at the center appeared as a torus,
and we speculated that the ratio of the period of REB motion to that
of the pulse motion at the bump was irrational.
The dynamics was considerably influenced by the system size and may
vary with the length of the interval, which constitutes our future scope
of research.

\begin{figure}
  \begin{center}
    \includegraphics[width=.9\hsize]{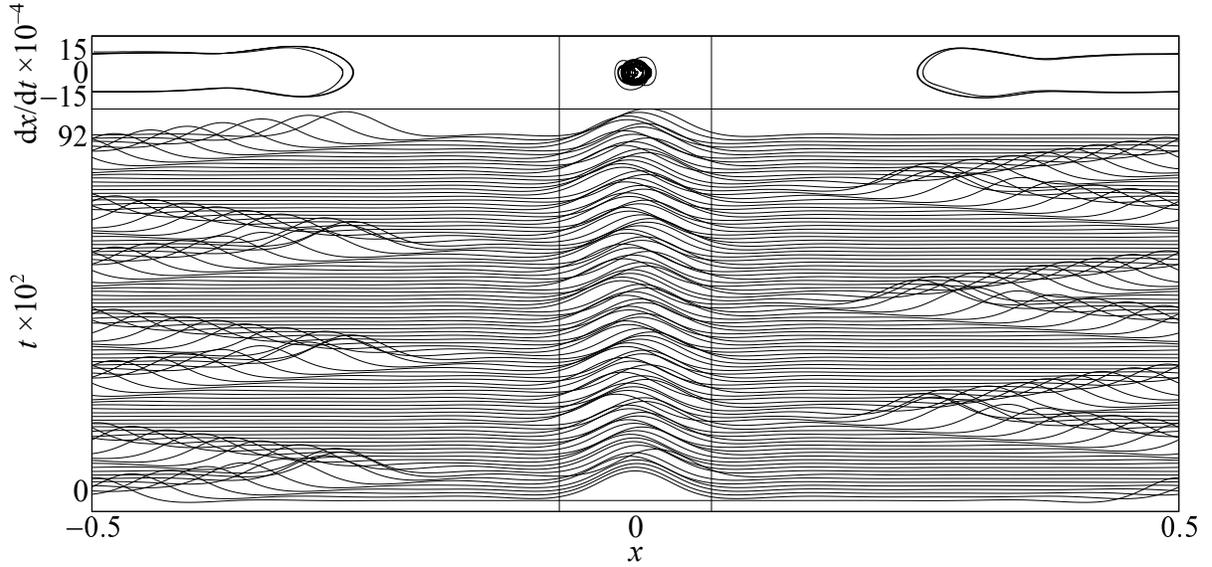}
    \caption{Example of asymptotic state beyond the admissible interval:
      $\varepsilon=0.01894$ and $d=0.14$.
      A new type of large amplitude pulse emerged at the bump center, and it oscillated around the bump for the given parameter set. The traveling pulse rebounded in the vicinity of the bump region and returned again, because it progressed around the circle. On the infinite line, the pulse never returned and standing breather recovered a periodic motion thereafter. However, in case the ratio of the period of the REB solution to the oscillation of the standing breather at the bump center is irrational for the circle, the orbit of the flow apparently formed a two dimensional torus, as depicted in top panel. }
    \label{0.14+0.01894.eps}
  \end{center}
\end{figure}

\newpage

\section{Dynamics of reduced ODEs
}\label{S_Dynamics_of_reduced_ODEs}
In previous sections, we discussed the behavior of TPO in heterogeneous media and produced the outcomes after the bump collision was completely captured by the HIOP, especially, the asymptotic behavior of TPO that converged into one of the HIOP solutions. Moreover, the transition mechanism for obtaining outputs was suggested by the global structure of the HIOP with parameter variation. In this section, the PDE dynamics was reduced to a finite dimensional system near the drift bifurcation point to conduct a more detailed analysis of the interaction between the TPO and heterogeneities. In particular, the variation in bump height caused various types of transitions for each pulse orbit, and we aim to explain mechanism and instant of such transitions based on a dynamical system perspective, which is consistent with the global structure of HIOP.

\subsection{Reduction method to finite dimensional ODEs
}\label{S_Dynamics_of_reduced_ODEs_reduction}
The reduction method to finite dimensional system near the drift bifurcation point is explained in this section, and its detailed derivation is presented in Appendix A.1. The $\varepsilon$ and $\tau$ were employed as bifurcation parameters, corresponding to the bump height and the relaxation parameter for $v$. The resulting coupled Eqs.(\ref{eq3.50}) describe the basic interactive dynamics between the pulse and heterogeneity for defining the transition mechanism based on varying strength of heterogeneity. In particular, the system of Eq.(\ref{eq3.50}) can be regarded as a normal form, because it comprises cubic nonlinearity of drift bifurcation along with oscillatory heterogeneity terms. 
Subsequently, the system was linearized around a stationary pulse solution at the drift bifurcation point $\tau=\tau_{c}= 1/\kappa_3$ (refer to Appendix A.2 for details). First, we translate the background state to zero.
\begin{equation}
  \label{eq_ode_1}
  \begin{array}{rcl}
    u(x,t)&=&\underline u+\tilde u(x,t)\\
    v(x,t)&=&\underline u+\tilde v(x,t)\\
    \kappa_1(x)&=&\underline\kappa_1+\widehat\kappa_1(x)
  \end{array}
\end{equation}
Substituting Eq.(\ref{eq_ode_1}) into Eq.(\ref{eq_bif_2}) and considering that $\underline u$ and $\underline\kappa_1$ satisfies
Eq.(\ref{eq2.03}), we obtain
\begin{equation}
  \label{eq3.06}
  \left\{ \begin{array}{lcl}
    \tilde u_t
    =\mathcal{L}\tilde u-\kappa_{3}\tilde v+N(\tilde u)
    +\widehat{\kappa}_{1}\\
    \tau\tilde v_t/\tau_c=(\tilde u-\tilde v)/\tau_c
  \end{array}\right.
\end{equation}
where
\[
\mathcal{L}=D_u\Delta+\kappa_2-\kappa_4\left(1-D_w\Delta\right)^{-1}
-3\underline u^2
\]
and
\[
N(\tilde u)=-3\underline u\tilde u^2-\tilde u^3
\]
Therefore, the background state of pulse
$\tilde U=(\tilde u,\tilde v)$ equals zero.
Let us linearize Eq.(\ref{eq3.06}) around a stationary pulse solution
$\overline{U}=(\overline{u},\overline{v})$ of (6) at $\tau_{c}$, and 
substitute
$\tilde U
=\overline U
+(\hat u,\hat v)\ope^{\lambda t}$
into the linearized equations. Thus, 
\[
\lambda\hat U=\mathcal D(\overline U;\tau_c)\hat U,
\]
where $\hat U=(\hat u,\hat v)$, and $\mathcal D(\overline U;\tau_c)$ is defined by
\begin{equation}
  \label{eq3.07}
  \mathcal{D}(\overline{U}; \tau_{c})=\mathcal{M}(\tau_{c})\left(
  \begin{array}{cc} \mathcal{L}+N'(\overline{u}) & -\kappa_{3} \\
    -\kappa_{3} & \kappa_{3} \end{array} \right),
\end{equation}
where
\begin{equation}
  \label{eq3.08}
  \mathcal{M}(\tau_{c})=\left( \begin{array}{cc} 1 & 0 \\
    0 & -1/\tau_{c}\kappa_{3} \end{array} \right).
\end{equation}
At the drift bifurcation point $\tau=\tau_{c}$, the operator $\mathcal{D}$ possesses
a double-degenerate zero eigenvalue corresponding to the
Goldstone and propagator modes. Specifically, the Goldstone mode can be expressed as
\begin{equation}
  \label{eq3.09}
  \mathcal{G}_{x}(x)=\left( \begin{array}{c}
    \frac{\partial}{\partial x} \overline{u}(x) \\
    \frac{\partial}{\partial x} \overline{u}(x) \end{array} \right),
\end{equation}
and the accompanied propagator mode of $\mathcal{D}(\overline{U}; \tau_{c})$ can be computed as
\begin{equation}
  \label{eq3.10}
  \mathcal{P}_{x}(x)=\left( \begin{array}{c} 0 \\
    -\frac{1}{\kappa_{3}}\frac{\partial}{\partial x} \overline{u}(x)
  \end{array} \right),
\end{equation}
where $\mathcal{D}\mathcal{G}_{x}=0$ and
$\mathcal{D}\mathcal{P}_{x}=\mathcal{G}_{x}$.
Based on the adjoint operator $\mathcal{D}^{*}(\overline{U}; \tau_{c})$,
the complementary Goldstone mode and the adjoint propagator mode can be explicitly evaluated using the solvability condition
$<\mathcal{G}_{x}^{*},\mathcal{G}_{x}>=0$,
$\mathcal{G}_{x}^{*}=\mathcal{M}^{-1}(\tau_{c})\mathcal{G}_{x}$:
\begin{equation}
  \label{eq3.11}
  \mathcal{G}_{x}^{*}(x)=\left( \begin{array}{c}
    \frac{\partial}{\partial x} \overline{u}(x) \\
    -\frac{\partial}{\partial x} \overline{u}(x) \end{array} \right)
\end{equation}
and
\begin{equation}
  \label{eq3.13}
  \mathcal{P}_{x}^{*}(x)=\left(\begin{array}{c}
    \frac{1}{\kappa_{3}}\frac{\partial}{\partial x}\overline{u}(x)\\
    0,
  \end{array}\right)
\end{equation}
where $\mathcal{D}^{*}\mathcal{G}_{x}^{*}=0$,
$\mathcal{D}^{*}\mathcal{P}_{x}^{*}=\mathcal{G}_{x}^{*}$.
Herein, the adjoint operator $\mathcal{D}^{*}$ is defined by
\begin{equation}
  \label{eq3.12}
  \mathcal{D}^{*}(\overline{U}; \tau_{c})=\left( \begin{array}{cc}
    \mathcal{L}+N'(\overline{U}) & 1/\tau_{c}\\
    -\kappa_{3} & -1/\tau_{c} \end{array} \right).
\end{equation}

\noindent
The dynamics of traveling pulse solution near the drift bifurcation point can be approximated using multiple time-scale $T_{1}$, $T_{2}$, and $T_{3}$ with $T_{i}=\delta^{i}t$.

\begin{equation}
  \label{eq3.19}
  \begin{array}{rcl}
    \tilde U(x,t)&=&\overline{U}(x-p(T_{1}, T_{2}, T_{3}))\\
    &&-\delta\kappa_{3}\alpha(T_{1}, T_{2})
    \mathcal{P}_{x}(x-p(T_{1}, T_{2}, T_{3}))\\
    &&+\delta^{2}\vec r(x-p(T_{1},T_{2},T_{3}))+\delta^{3}\vec R(x),
  \end{array}
\end{equation}
\noindent
Here $p(T_{1}, T_{2}, T_{3})$ represents the position of pulse, $\alpha(T_{1}, T_{2})$ is the amplitude of propagator mode, and the higher order corrections $\vec r=(r_u,r_v)$ and $\vec R=(R_u,R_v)$ belong to the orthogonal
subspace of the principal part spanned by $\{\mathcal G_x,\mathcal P_x\}$.
According to projection technique, rescaling $\alpha\delta$ by $\alpha$, and recovering original variables, 
$\widetilde{\tau}=\frac{\tau-1/\kappa_{3}}{\delta^{2}}$,
$\widehat{\kappa}_{1}(x)=\delta^{2}\widetilde{\kappa}_{1}(x)$ (refer to Appendix A.1 for details), we can obtain the following reduced two-dimensional system:

\begin{equation}
  \label{eq3.47}
  \left\{ \begin{array}{lcl}
    \dot{p}=\kappa_{3}\alpha-\frac{\left<\frac{\partial}{\partial x}
      \overline{u}, \widehat{\kappa}_{1}\right>}{\left
      <(\frac{\partial}{\partial x} \overline{u})^{2}\right>}\\
    \dot{\alpha}=\kappa_{3}^{2}(\tau-1/\kappa_{3})\alpha
    -\kappa_{3}\alpha^{3}\frac{\left<(\frac{\partial^{2}}{\partial x^{2}}
      \overline{u})^{2}\right>}{\left<(\frac{\partial}{\partial x}
      \overline{u})^{2}\right>}-\frac{\left<\frac{\partial}{\partial x}
      \overline{u}, \widehat{\kappa}_{1}\right>}{\left
      <(\frac{\partial}{\partial x} \overline{u})^{2}\right>}\\
  \end{array}\right.
\end{equation}
where $\overline{u}=\overline{u}(x-p)$ denotes the stationary pulse solution of (6) at the drift bifurcation point $\tau_{c}=1/\kappa_{3}$,
$\widehat{\kappa}_{1}=\widehat{\kappa}_{1}(x,\varepsilon,d)$ expresses the bump heterogeneity of (8), $p=p(t)$ represents the position
of pulse, and $\alpha=\alpha(t)$ indicates the velocity of pulse.
Recalling that $\widehat{\kappa}_{1}(x)$ has steep boundary, it can be approximated using a step function.
Therefore, 
$\left<\frac{\partial}{\partial x}\overline{u},\widehat{\kappa}_{1}\right>$
can be rewritten using the Dirac $\delta$ function.

\begin{equation*}
  \left<
  \frac{\partial}{\partial x} \overline{u}(x-p), \widehat{\kappa}_{1}(x)
  \right>
  =\int_{-\infty}^{\infty}\frac{\partial}{\partial x} \overline{u}(x-p)
  \widehat{\kappa}_{1}(x)dx
\end{equation*}
\begin{equation}
  \label{eq3.48}
  =\varepsilon\left[\overline{u}\left(\frac{d}{2}-p\right)
    -\overline{u}\left(-\frac{d}{2}-p\right)
    \right]
\end{equation}
\noindent
Subsequently, the heterogeneity term $f(p,d)$ can be defined by
\begin{equation}
  \label{eq3.49}
  f(p,d)=\overline{u}\left(\frac{d}{2}-p\right)
  -\overline{u}\left(-\frac{d}{2}-p\right).
\end{equation}
The function $f(p,d)$ is an odd function of $p$ and
decays exponentially in a sinusoidal way, as depicted in Fig.\ref{properties_ode.eps}.
Let $C_{1}=\left<(\frac{\partial}{\partial x} \overline{u})^{2}\right>$,
$C_{2}=\left<(\frac{\partial^{2}}{\partial x^{2}} \overline{u})^{2}\right>$, then Eq.(\ref{eq3.47}) deduces into
\begin{equation}
  \label{eq3.50}
  \left\{ \begin{array}{lcl}
    \dot{p}=\kappa_{3}\alpha-\frac{\varepsilon}{C_{1}}f(p,d)\\
    \dot{\alpha}=\kappa_{3}^{2}(\tau-1/\kappa_{3})\alpha-\kappa_{3}
    \alpha^{3}\frac{C_{2}}{C_{1}}-\frac{\varepsilon}{C_{1}}f(p,d)\\
  \end{array}\right.
\end{equation}
where $d$ denotes the bump width, $\varepsilon$ denotes its height, 
$C_{1}$ and $C_{2}$ represent positive constants, and
$\tau$ and $\kappa_{3}$ are known PDE parameters, stated in (\ref{parametervalues}).
Note that the reduced ODE system (\ref{eq3.50}) without heterogeneity can be defined as 
\begin{equation}
  \label{eq6.01}
  \left\{ \begin{array}{lcl}
    \dot{p}=\kappa_{3}\alpha\\
    \dot{\alpha}=\kappa_{3}^{2}(\tau-1/\kappa_{3})\alpha
    -\kappa_{3}\alpha^{3}C_{2}/C_{1} \\
  \end{array}\right.
\end{equation}
Thus, Eq.(\ref{eq6.01}) evidently undergoes supercritical pitchfork 
bifurcation at $\tau=\tau_{c}$ and inherits the same bifurcation of the original PDE (refer to Appendix A.2).
For $\tau<\tau_{c}$, the equilibrium $\alpha^{*}=0$ is stable and corresponds to a stationary pulse solution.
For $\tau>\tau_{c}$, the equilibria of the second equation
$\overline{\alpha}_{\pm}=\pm\sqrt{\frac{(\kappa_{3}\tau-1)C_{1}}{C_{2}}}$
are stable and correspond to a traveling pulse solution with speed $\overline{\alpha}_{\pm}$.

Ultimately, we conclude that the PDE system (\ref{eq2.03}) with bump heterogeneity can be formally reduced to the two-dimensional ODE system (\ref{eq3.50}). Notably, the system (\ref{eq3.50}) can be regarded as a normal form of TPO near the drift bifurcation, because it contains a cubic pitchfork bifurcation and an oscillatory external function.

\subsection{Reduced Dynamics}
\subsubsection{Elementary properties of critical points}
Now, we can study the reduced ODE system (\ref{eq3.50}) with bump heterogeneity. 
The profile of the bump heterogeneity term $f(p,d)$, as defined in Eq.(\ref{eq3.49}), is illustrated in Fig.\ref{properties_ode.eps}-(a).
To assure the definiteness, all numerical computations were conducted for the case $d=0.05$ (refer to Fig.\ref{properties_ode.eps}-(b)).
In addition, the following discussion can be extended to a wider case with suitable modifications.  

\begin{figure}
  \centering
  \includegraphics[width=.8\hsize]{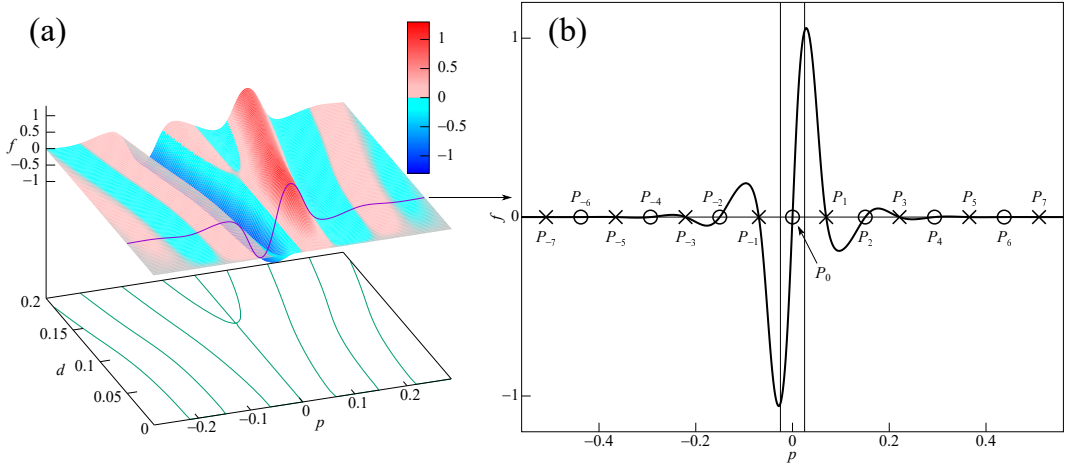}
  \caption{%
    (a) Profile of heterogeneity term $f(p,d)$ defined
    using Eq.(\ref{eq3.49}).
    (b) Cross-section of $f(p,d)$ at $d=0.05$; countably many zeros of $f(p,d)$ are present.  }\label{properties_ode.eps}
\end{figure}

The critical points of Eq.(\ref{eq3.50}) are defined using the following equations, and their properties are summarized in the following proposition.

\begin{equation}
  \label{eq6.03}
  \left\{ \begin{array}{lcl}
    \kappa_{3}\alpha-\frac{\varepsilon}{C_{1}}f(p,d)=0\\
    \kappa_{3}^{2}(\tau-1/\kappa_{3})\alpha
    -\kappa_{3}\alpha^{3}\frac{C_{2}}{C_{1}}-\frac{\varepsilon}{C_{1}}f(p,d)=0\\
  \end{array}\right.
\end{equation}

\begin{proposition}
  \label{prop6.00}
  We assumed $\tau_{c}<\tau<2\tau_{c}$, the following relationships hold.
  \begin{itemize}
  \item All the critical points of Eq.(\ref{eq3.50}), i.e., the zero points of Eq.(\ref{eq6.03}) were on the axis $\alpha=0$, coinciding with the zero points of $f(p,d)$.
  \item Countably many critical points exist for Eq.(\ref{eq3.50})
    with labels $(P_{i},0)$, $i\in \mathbb{Z}$.
  \item $(P_{i}, 0)$ is a saddle point, if $\varepsilon D(P_{i},d)<0$ where $D(P_{i},d)=\frac{\partial f(p,d)}{\partial p}|_{P_{i}}$.
    The saddle quantity (i.e., sum of real eigenvalues) is always positive at the saddle point:   
    $\sigma=\lambda_{1}+\lambda_{2}=b>0$.
    Refer to Eq.(\ref{eq6.12}) for the definition of $b$.
  \item $(P_{i}, 0)$ denotes a nonsaddle point if $\varepsilon D(P_{i},d)>0$.
    As the modulus of $\varepsilon$ increases, the nonsaddle point undergoes a  sequential variation: unstable node $\rightarrow$ unstable spiral
    $\rightarrow$ stable spiral $\rightarrow$ stable node. Refer to Fig.\ref{stability_zeros.eps}.  
  \item For each fixed $d$ and
    $\varepsilon$, only the finitely many critical points $(P_{i}, 0)$ ($i\in \mathbb{Z}$) are stable.
  \end{itemize}
\end{proposition}

\begin{proof}
  Refer to Appendix A.3.
\end{proof}

\textbf{REMARK}
The restriction of $\tau_{c}<\tau<2\tau_{c}$ is vital for our analysis. The first inequality demonstrates the post-drift bifurcation ensuring the existence of the critical point associated with TPO. The second inequality ensures that all the critical points are situated on the axis $\alpha=0$, and two types of pinning exist: stable STA and OSC. Refer to Fig.\ref{stability_zeros.eps}. The positivity of the saddle quantity $\sigma$ is essential for discussing homoclinic bifurcation, as presented in Proposition 6.\\

We denote the orbit starting from a special initial value $(p,\alpha) = (-\infty,\overline{\alpha}_{+})$ by $\mathcal{O}_\mathrm{pulse}$. The classification of outputs in the phase diagram Fig.\ref{ODE2.eps} is based on the behaviors of $\mathcal{O}_\mathrm{pulse}$. Generally, the outcome significantly depends on the initial data, and the basin boundary location of each attractor must be identified for a complete classification of the initial space. In particular, we are interested in the basin boundary relevant to the emerging transition, as illustrated in Fig.\ref{ODE2.eps}. The basin boundary can be evidently characterized by a stable manifold of a saddle point that is relevant to the transition. Moreover, as the parametric dependency on $\varepsilon$ is complicated, the prediction of outcome for a given initial data becomes subtle even in a two-dimensional ODE system. Thus, we explain the relevant details in the following subsections and introduce a couple of conventions for later convenience: The initial point $(p,\alpha) = (-\infty,\overline{\alpha}_{+})$ of $\mathcal{O}_\mathrm{pulse}$ denoted by $T^{+}_{-\infty}$ was considered as a generalized critical point located at $p = -\infty$, and similarly, for other points $T^{\pm}_{\pm \infty} \coloneqq (\pm\infty,\overline{\alpha}_{pm})$, PEN (REB) was considered as a heteroclinic orbit connecting $T^{+}_{-\infty}$ to $T^{+}_{\infty}$ ($T^{-}_{-\infty}$) based on the convention. In PDE setting and HIOP branches, the PEN and REB are time-periodic solutions on $\mathbb S^1$ and these solutions can be obtained by unfolding the corresponding heteroclinic orbits from $\mathbb R^1$ to $\mathbb S^1$.

Although the initial point $T^{+}_{-\infty}$ of $\mathcal{O}_\mathrm{pulse}$ is theoretically clear and well-defined, we approximated it on a circle to compute the associated PDE solution. Nonetheless, a small difference in the approximation process may cause a considerable difference in output as illustrated in Fig.\ref{spiral.eps}. We discuss this issue in Sections 5.2.3 and 6. 

We assign a label for each critical point as $P_{n}$,
($n=0,\pm1,\pm2,\cdots$), where $P_{0}$ is located at the center of bump
heterogeneity. For positive $\varepsilon$, the odd sequence $P_{2n-1}$ becomes saddle, and $W^{l,s}_{P_{2n-1}}$ denotes the stable manifold of $P_{2n-1}$ approaching
$P_{2n-1}$ from the left, and similarly, for other three manifolds, e.g.,
$W^{r,s}_{P_{2n-1}}$, $W^{l,u}_{P_{2n-1}}$, $W^{r,u}_{P_{2n-1}}$ as in Fig.\ref{stability_zeros.eps} (b). For negative $\varepsilon$, the even sequence $P_{2n} $ becomes saddle instead and a similar label can be assigned to each manifold. Subsequently, we assume a small bump width $d$ such that only one critical point $P_{0}$ existed inside the bump heterogeneity (such as for $d=0.05$ portrayed in Fig.\ref{properties_ode.eps}-(b)). Moreover, we could conveniently extend the following arguments to a wider case similar to \cite{Nishiuraetal2007}.

\begin{figure}
  \centering
  \includegraphics[width=.8\hsize]{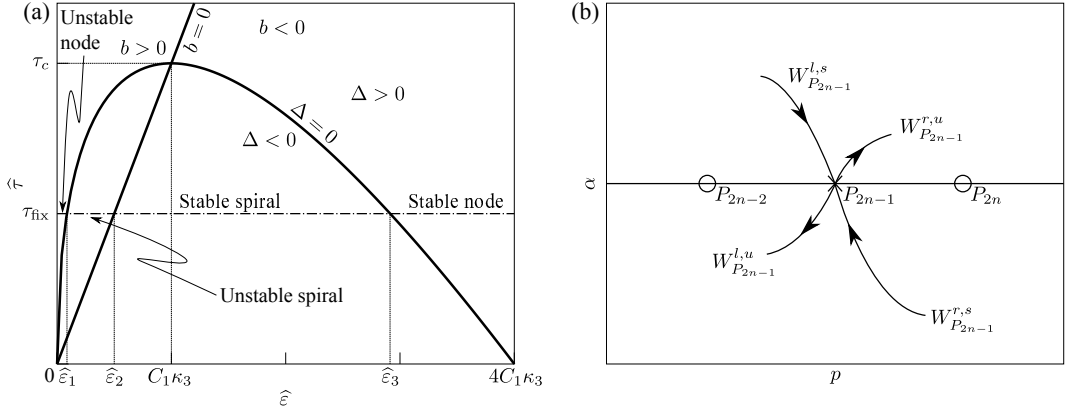}
  \caption
      { (a) Dependency of critical point $(P_{i}, 0)$ (nonsaddle type), $i\in \mathbb{Z}$ on 
        $\widehat{\varepsilon}$. The horizontal axis denotes $\widehat{\varepsilon} := \varepsilon D(P_{i},d)$ in which $D(P_{i},d)$ is defined in Proposition $\ref{prop6.00}$ and the vertical axis is $\widehat{\tau}=\tau-1/\kappa_{3}$. For a fixed $\tau=\tau_\mathrm{fix}$ ranging within $(\tau_{c}, 2\tau_{c})$, nonsaddle point underwent sequential variation as $\widehat{\varepsilon} > 0$ increased: Unstable node $\rightarrow$ unstable spiral $\rightarrow$ stable spiral $\rightarrow$ stable node. Refer to Appendix A.3 for further details. (b) Labeling stable and unstable manifolds of the saddle point. For a given saddle point $P_{2n-1}$, each stable and unstable manifold is labeled as depicted. 
      }
      \label{stability_zeros.eps}
\end{figure}



\subsubsection{Phase diagram of reduced ODEs}

The phase diagram of the ODE system (\ref{eq3.50}) in the parameter space $(d, \varepsilon)$ is presented in Fig.\ref{ODE2.eps}-(a), wherein the parameter $\tau=3.35$ is the same as that in the PDE case.
There were four distinct outputs PEN, REB, OSC and STA, and the cyclic behavior REB-OSC-STA was observed, except for the PEN regime in the vicinity of $\varepsilon =0$. The two states OSC and STA pertained to the pinning category, i.e., the orbit settled to either a limit cycle or a stable critical point. The stability of each critical point relied on $(d,\varepsilon)$ so the region of STA may contract as $d$ expands, such as that in Fig.\ref{ODE2.eps}-(a). As the transition OSC--STA is a local Hopf bifurcation, the entire phase diagram can be regarded as a repetition of pinning--depinning process, as depicted in Fig.\ref{ODE2.eps}-(b). Moreover, we initially deduce that PEN is observed for small $\varepsilon$ as expected. The proof is detailed in in Appendix A.4.

\begin{proposition}
  \label{prop6.05}
  Suppose the bump height satisfies
  $|\varepsilon|<\frac{\kappa_{3}C_{1}(\overline{\alpha}_{+}-\delta)}{2M_{0}}$,
  then the dynamics of the traveling pulse of reduced ODEs (\ref{eq3.50}) is PEN, where $\delta$ and $M_0$ are appropriately selected positive constants that are independent of small $\varepsilon$ with ($\overline{\alpha}_{+}-\delta) > 0$. 
\end{proposition}

\begin{figure}
  \centering
  \includegraphics[width=12cm]{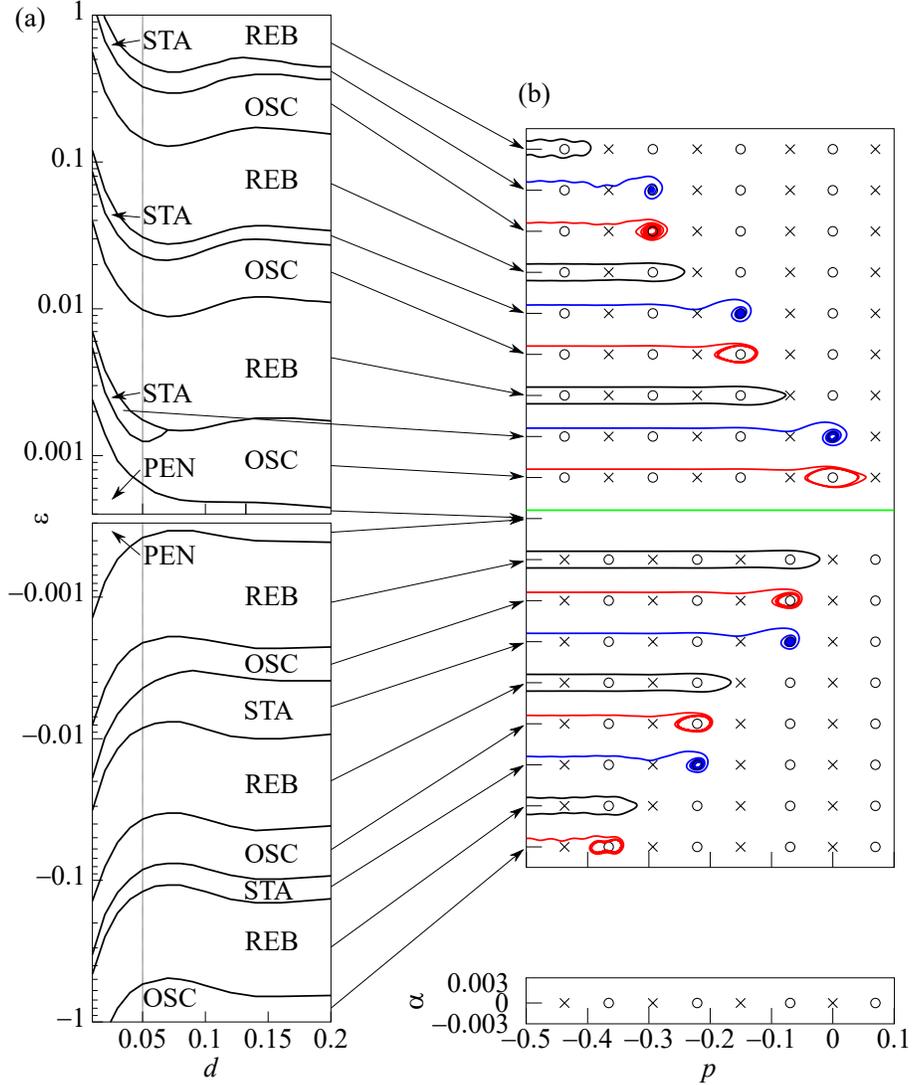} 
  \caption{(a):
    Phase diagram of reduced ODE
    system (\ref{eq3.50}) in $(d, \varepsilon)$-space wherein $\varepsilon$ axis is plotted in
    logarithmic coordinate. Parameters are adopted from the PDE case as in Eq.(\ref{parametervalues}) and $\tau=3.35$. (a) Two complete cycles of repeated dynamical response
    (REB--OSC--STA) in each figure: $\varepsilon>0$ (top) and $\varepsilon<0$ (bottom).
    Qualitative features of PDE phase diagram (Fig.22) are inherited to reduced case.
    (b) Phase diagram of ODEs with flow patterns. The symbol $\times$ denotes a saddle point and circle for nonsaddle point.}
  \label{ODE2.eps}
\end{figure}

\textbf{REMARK}
For the reduced ODE system, we observed {\it infinitely} many cyclic response REB--OSC--STA as the height varied. However, this was not the same for the
original PDE system, simply because stable traveling pulses were not
sustained when the modulus of $\varepsilon$ (or equivalently $\kappa_1$) was beyond the admissible interval (red-dotted vertical line) as depicted in
Fig.\ref{isola.eps}.
This implied that the PDE--ODE correspondence is feasible only within the admissible interval of $\varepsilon$.\\

The context presented in this section aims to clarify the pinning--depinning transition from a dynamical system perspective, assuming that the phase transition proceeds according to Fig.\ref{ODE2.eps}. Recall that we investigated the global behaviors of solution branches associated with PEN, OSC, STA, and REB in PDE sense, as depicted at the bottom of Fig.\ref{m.eps}. Subsequently, we elucidated the transition mechanism in the ODE setting and review its consistency with the global structure of HIOP, which signified that the basic mechanism of PDE was inherited to that of the ODE dynamics.\\

\subsubsection{Cyclic pinning--depinning transitions and consistency with HIOP structure}\label{S_Dynamics_of_reduced_ODEs_cyclic}

As the bump height increased (decreased), the first transition occurred from PEN to OSC (REB). Thereafter, a cyclic behavior was observed with REB--OSC--STA, which was followed by the pinning--depinning process. Thereafter, we explain that all these transitions can be understood by the interaction between the pulse orbit and stable/unstable manifolds of critical points distributed on the horizontal axis. In particular, the pinning state comprises stationary and oscillatory states depending on before/after the Hopf bifurcation, so the transition between OSC and STA stems from a local bifurcation. Subsequently, we focused on the pinning transition such as PEN to OSC and REB to OSC and the depinning process such as the STA to REB for clarifying these transition mechanisms based on dynamical system perspective. We compared between the PDE dynamics and reduced ODE dynamics, especially, regarding the behavior of the basin boundary in vicinity of the transition. Moreover, we aimed to consistently integrate the following three aspects: PDE dynamics, global structure of HIOP of Eq.(\ref{eq_bif_2}) with heterogeneity (\ref{eq4.01}), and the reduced ODE dynamics (\ref{eq3.50}). 
The following arguments were partially based on the numerical results, and the complete rigorous scenario is open for future research.




\textbf{PEN$\rightarrow$OSC}\\
First, we discuss the transition from PEN to OSC with the increasing $\varepsilon$. For extremely small and positive $\varepsilon$, all the stable/unstable manifolds of saddles $P_{2n-1}$ ($n=0,\pm1,\pm2,\cdots$) were below $\mathcal{O}_\mathrm{pulse}$ in the upper-half plane, and therefore, PEN was observed as established in Proposition.4. Recall that $\mathcal{O}_\mathrm{pulse}$ denotes the orbit with initial data $T^{+}_{-\infty}$. In addition, the pulse orbit was gradually deformed with the increasing $\varepsilon$ and started to surround the limit cycle neighboring $P_{0}$ in case $\varepsilon$ exceeded the critical height, as illustrated in Fig.\ref{PEN-OSC.eps} (black-solid line in middle panels (b) from left to right). Upon closely observing the behavior in vicinity of the transition, we noted that $\mathcal{O}_\mathrm{pulse}$ coincided with $W^{l,s}_{P_{1}}$ at the critical height and reached under it as the height is further increased ((b)-4 and (b)-5). In particular, the exchange of the order between $\mathcal{O}_\mathrm{pulse}$ and $W^{l,s}_{P_{1}}$ causes the transition from PEN to the pinning state OSC. Note that the flow pattern was odd-symmetric with respect to the origin, so the behavior of $W^{r,s}_{P_{-1}}$ is exactly the same as the mirror image of $W^{l,s}_{P_{1}}$. Therefore, it can block $\mathcal{O}_\mathrm{pulse}$ from the leftward of $P_{-1}$, and more importantly, $\mathcal{O}_\mathrm{pulse}$ belonged to the basin of the limit cycle around $P_0$.
In PDE simulation ((a)-2 in Fig.\ref{PEN-OSC.eps}), PDE orbit was situated proximate to the $\mathrm{[I]}^\mathrm{S}_{1}$ for a considerable period of time prior to leaving the bump region, which suggested that the orbit was in close vicinity to the stable manifold of $\mathrm{[I]}^\mathrm{S}_{1}$. Moreover, in case the height attained the critical value, it was located on the stable manifold of $\mathrm{[I]}^\mathrm{S}_{1}$ according to the behavior of the ODE dynamics, wherein the $\mathcal{O}_\mathrm{pulse}$ coincided with the $W^{l,s}_{P_{1}}$ at the transition point.

\begin{figure}
  \begin{center}
    \includegraphics[width=.95\hsize]{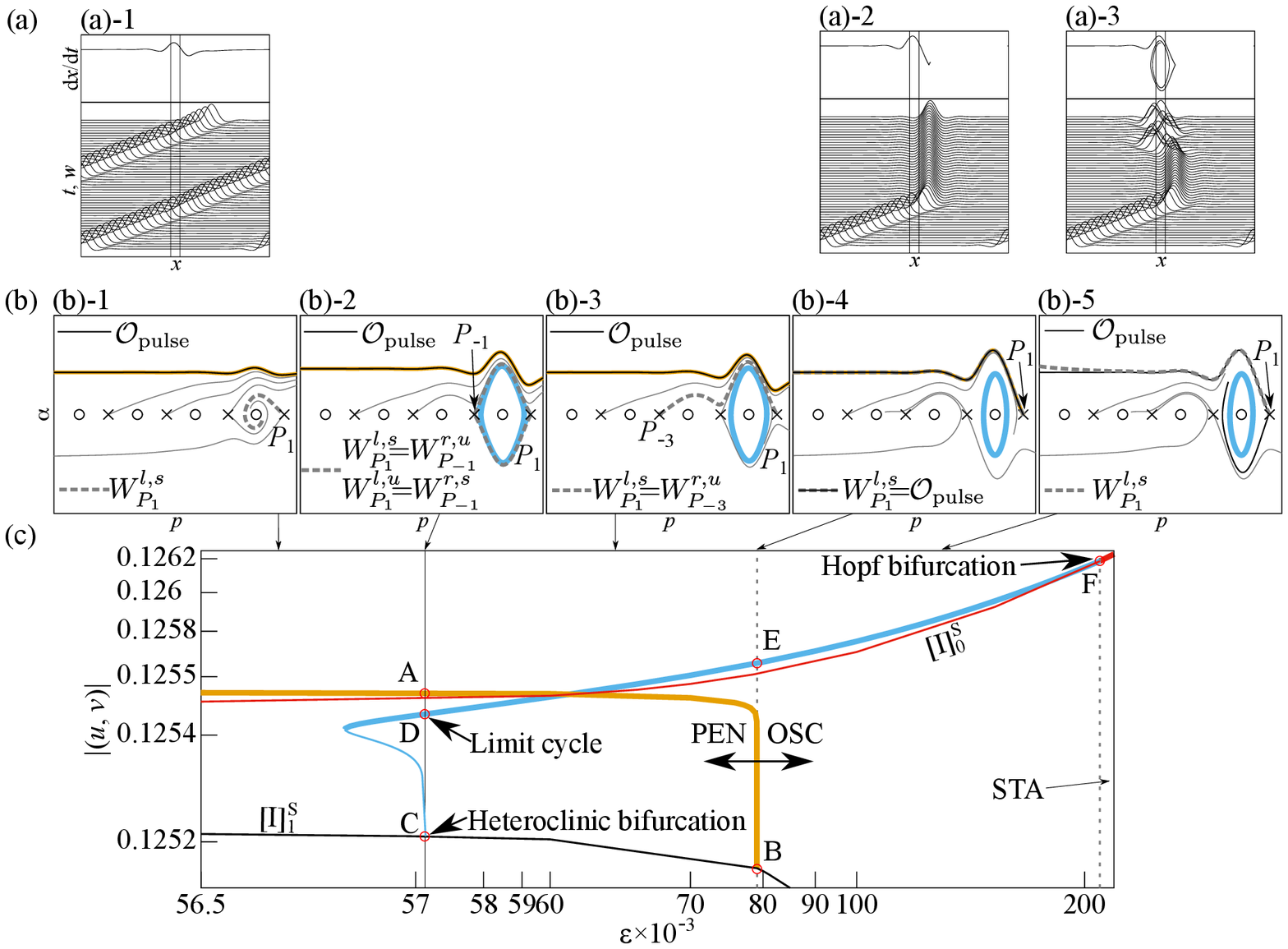}
    \caption{Dynamics near the transition point PEN--OSC.
      (a) PDE simulations, (b) ODE flows, and (c) HIOP structure, where ocher branch indicates PEN and cyan branch indicates OSC.
      In ODE setting (b) transition PEN--OSC is characterized by exchange of order between $\mathcal{O}_\mathrm{pulse}$ and $W^{l,s}_{P_{1}}$. In particular, the transition point is expressed as the height in which $\mathcal{O}_\mathrm{pulse}$ coalesces into $W^{l,s}_{P_{1}}$ ((b)-4).
      Thus, $\mathcal{O}_\mathrm{pulse}$ becomes a double-heteroclinic orbit from $T^{+}_{-\infty}$ to $P_{1}$, and from $P_{1}$ to $T^{+}_{\infty}$ at the critical height. 
      The basin boundary between PEN and OSC is expressed as $W^{l,s}_{P_{1}}$ for general initial data, and it experienced countably many reconnections among the critical points located at the left of $P_{0}$, as $\varepsilon$ approached the critical height.
      The ODE flows presented in (b)-3 display that $P_{-3}$ is connected to $P_{1}$.
           Moreover, the origin of the OSC branch could be traced to a Hopf bifurcation at F (supercritical), and the OSC region ceased to exist and switched to the STA region for larger $\varepsilon$.
As $\varepsilon$ decreased, the OSC branch experienced a saddle-node point and resulted in a heteroclinic orbit connecting $P_{1}$ to $P_{-1}$ at C.
      This heteroclinic bifurcation exchanged the order between $W^{l,s}_{P_{1}}$ and $W^{r,u}_{P_{-1}}$ to enable the interaction of $W^{l,s}_{P_{1}}$ with $\mathcal{O}_\mathrm{pulse}$ (see the text for details). 
      As for the HIOP diagram (c), the period of PEN diverged as $\varepsilon$ approached B, indicating the occurrence of a homoclinic bifurcation, which corresponds to the above double-heteroclinic bifurcation in ODE setting. Recall that HIOP is considered on a circle and ODE on $\mathbb R$. The PDE solution remains in vicinity of stationary pulse $\mathrm{[I]}^{\mathrm{S}_{1}}$ for a long period immediately before the transition and jumps to an oscillatory solution after that ((a)-2). Note that both the vertical and horizontal axes were not uniform.   
 } \label{PEN-OSC.eps}
  \end{center}
\end{figure}

\begin{figure}
  \centering
  \includegraphics[width=8cm]{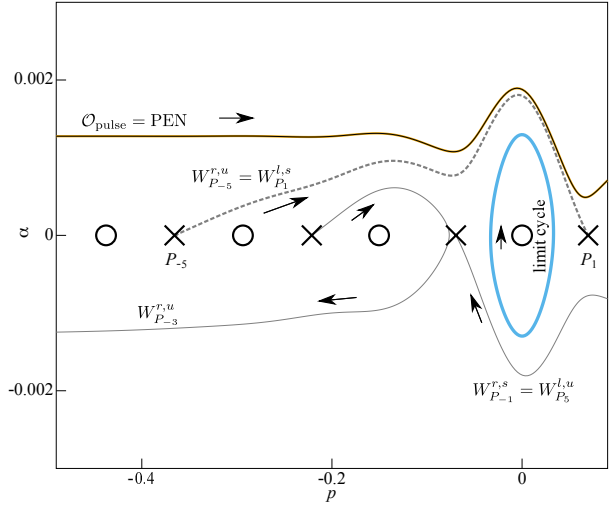}
  \caption{Reconnection process during transition from PEN to OSC at 
    $\varepsilon=0.000557127$. Basin boundary between PEN and OSC is expressed as stable manifold $W^{l,s}_{P_{1}}$ (thick-gray dotted line) that is connected to $P_{-5}$ at $\varepsilon=0.000557127$, namely $W^{l,s}_{P_{1}}$ = $W^{r,u}_{P_{-5}}$. Destination of $W^{l,s}_{P_{1}}$ in time-reversal direction moves leftward as $\varepsilon$ approaches the critical height. }
  \label{e+0.000557127_PEN.eps}
\end{figure}

\begin{figure}
  \begin{center}
    \includegraphics[width=.8\hsize]{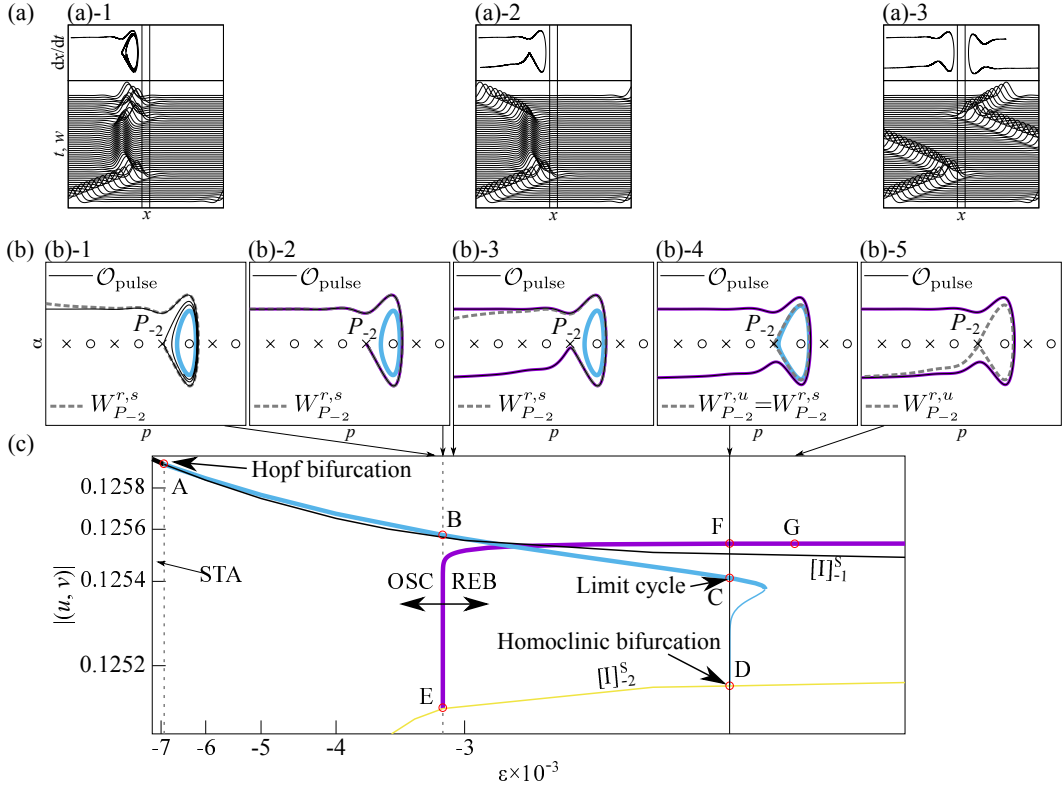}
    \caption{Oscillatory pinning point of REB-OSC.
      (a) PDE simulations, (b) ODE flows, and (c) HIOP structure.
      Purple line indicates REB, and cyan indicates OSC.
      In ODE setting (b), the transition from REB to OSC is characterized by the exchange in order from $\mathcal{O}_\mathrm{pulse}>W^{r,u}_{P_{-2}}>W^{r,s}_{P_{-2}}$ (REB) to $W^{r,s}_{P_{-2}}>\mathcal{O}_\mathrm{pulse}>W^{r,u}_{P_{-2}}$ (OSC) as $\varepsilon$ decreased.
      Two critical situations exist during the process: $W^{r,u}_{P_{-2}}=W^{r,s}_{P_{-2}}$ (homoclinic bifurcation at D, refer to Fig.\ref{e-0.00166147754_HOMO_REB.eps}), and $\mathcal{O}_\mathrm{pulse} = W^{r,s}_{P_{-2}}$ (transition state at (E), refer to Fig.\ref{e-0.0021_OSC.eps}).
      After homoclinic bifurcation, $W^{r,s}_{P_{-2}}$ underwent countably many reconnections among the critical points (Fig.\ref{e-0.00196208903100908_REB.eps}), after which the $\mathcal{O}_\mathrm{pulse}$ was approached and merged into $\mathcal{O}_\mathrm{pulse}$ in case the $\varepsilon$ attained the critical height.
      Refer to the text for further details.
      Note that both vertical and horizontal axes were not uniform.
    }
    \label{OSC-REB.eps}
  \end{center}
\end{figure}

\begin{figure}
  \centering
  \includegraphics[width=8cm]{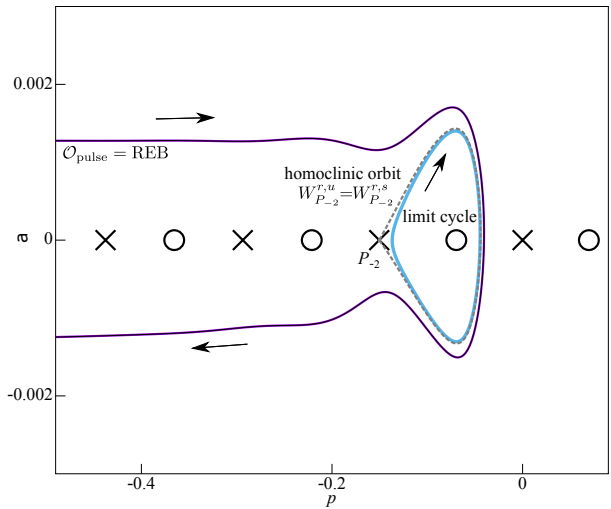}
  \caption{REB-OSC process at $\varepsilon=-0.00166147754$: Homoclinic bifurcation $W^{r,s}_{P_{-2}} = W^{r,u}_{P_{-2}}$.}
  \label{e-0.00166147754_HOMO_REB.eps}
\end{figure}

\begin{figure}
  \centering
  \includegraphics[width=8cm]{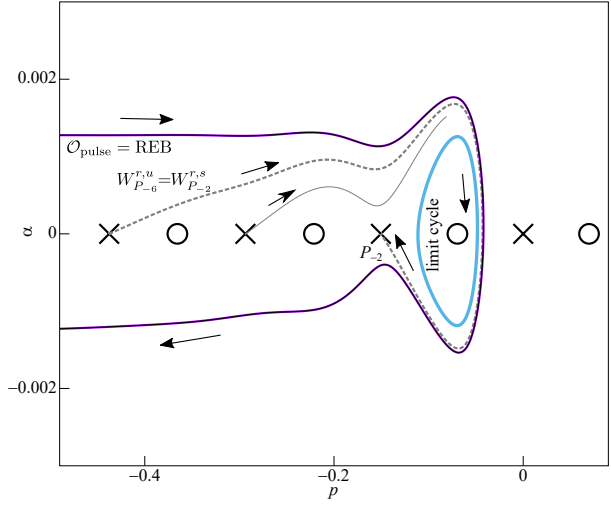}
  \caption{REB--OSC process at
    $\varepsilon=-0.001962$:
    Reconnection occurs between $P_{-2}$ and $P_{-n} (n=3, \dots)$. In this case, we have a heteroclinic connection $W^{r,s}_{P_{-2}} = W^{r,u}_{P_{-6}}$ at $\varepsilon=-0.001962$.
  }
  \label{e-0.00196208903100908_REB.eps}
\end{figure}

\begin{figure}
  \centering
  \includegraphics[width=8cm]{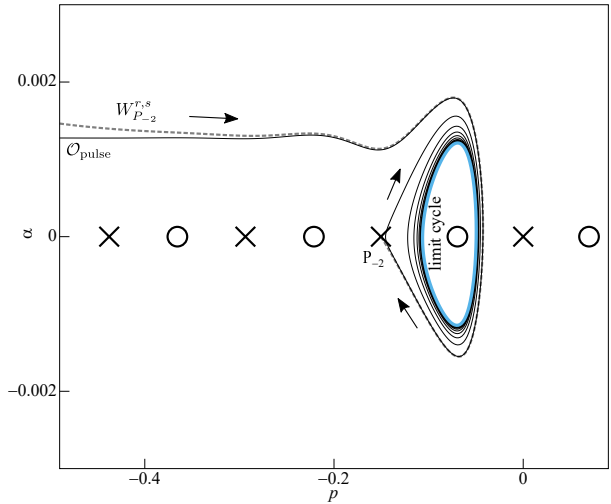}
  \caption{REB--OSC process at $\varepsilon=0.0021$: Immediately after the transition REB--OSC. $W^{r,s}_{P_{-2}}$ (gray-dotted line) was located above $\mathcal{O}_\mathrm{pulse}$ at $\varepsilon=0.0021$.  }
  \label{e-0.0021_OSC.eps}
\end{figure}

\begin{figure}
  \begin{center}
    \includegraphics[width=.95\hsize]{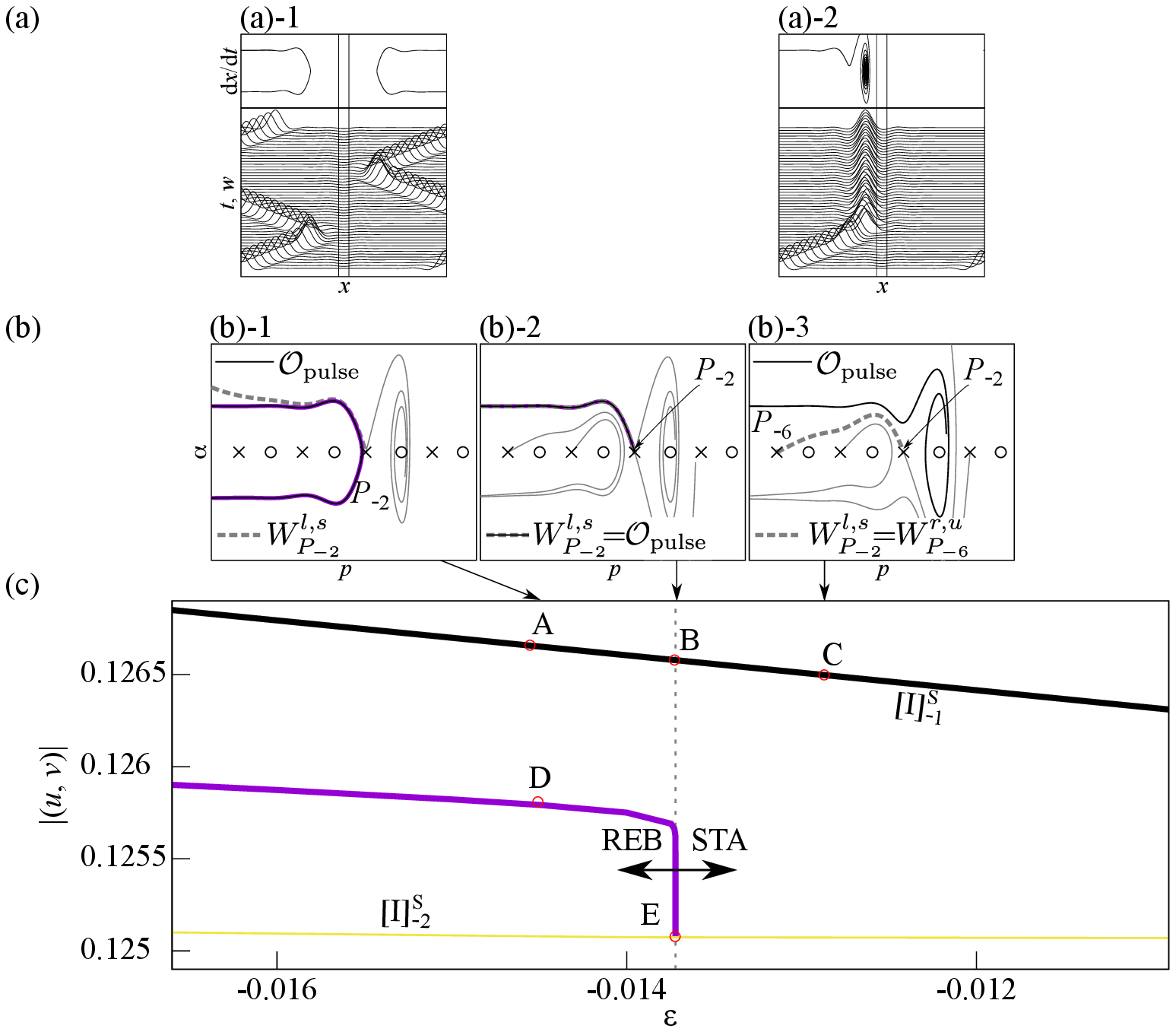}
    \caption{Depinning point of STA--REB.
      (a)	 PDE simulations, (b) ODE flows, and (c) HIOP structure.
      Purple line indicates REB, and black line indicates STA.
      The exchange of order between $\mathcal{O}_\mathrm{pulse}$ and the stable manifold $W^{l,s}_{P_{-2}}$ caused the transition.
      In fact, as $\varepsilon$ is decreased, $\mathcal{O}_\mathrm{pulse} > W^{l,s}_{P_{-2}}$ ((b)-3), $\mathcal{O}_\mathrm{pulse} = W^{l,s}_{P_{-2}}$ of $P_{-2}$ ((b)-2), and $\mathcal{O}_\mathrm{pulse} < W^{l,s}_{P_{-2}}$ of $P_{-2}$ ((b)-1).
      The infinitely many reconnections between the left stable manifold of $P_{-2}$ and other unstable manifolds of critical points occurred immediately prior to the transition, in particular, the ODE flow (b)-3 displayed a heteroclinic connection from the $P_{-6}$ to $P_{-2}$. (c) The HIOP structure supported the above mechanism, for instance, the REB solution (purple) converged on the $\mathrm{[I]}^\mathrm{S}_{-2}$ branch (yellow), which indicated that its period diverged spending most of the time near the unstable stationary pulse $\mathrm{[I]}^\mathrm{S}_{-2}$ as $\varepsilon$ approached the critical height.   
    }
    \label{REB-STA.eps}
  \end{center}
\end{figure}

Interestingly, the behavior of $W^{l,s}_{P_{1}}$ (gray-dotted line in Fig.\ref{PEN-OSC.eps} (b)) was observed in the time-reversal direction. For small values of $\varepsilon$, it surrounded the critical point $P_0$ ((b)-1). For slightly larger values of $\varepsilon$, the destination became $P_{-3}$ ((b)-3), and it coincided with the $\mathcal{O}_\mathrm{pulse}$ at the critical height ((b)-4). Eventually, it surpassed beyond $\mathcal{O}_\mathrm{pulse}$ as $\varepsilon$ exceeded it ((b)-5). The \textit{infinitely many reconnections of heteroclinic orbits} occurred during this process, which connected the $P_1$ and all the remaining critical points $P_n$ located on the left side of $P_1$. In other words, $W^{l,s}_{P_{1}}$ coincided with the unstable manifolds $W^{r,u}_{P_{n}}$ (n=-1, -2, ...) as $\varepsilon$ is increased, and it attained $\mathcal{O}_\mathrm{pulse}$ at the critical height. In context, a magnified snapshot of such a reconnection process is displayed in Fig.\ref{e+0.000557127_PEN.eps}, wherein a heteroclinic connection between the $P_{1}$ and $P_{-5}$ was realized at $\varepsilon = 0.000557127$. The transition PEN--OSC occurred after the reconnections were complete, and it was characterized by the exchange of order between $W^{l,s}_{P_{1}}$ and $\mathcal{O}_\mathrm{pulse}$. The heteroclinic connections to saddle points occurred for special discrete values of $\varepsilon$, whereas the connections to nonsaddle critical points such as the unstable spiral were present for the remaining $\varepsilon$.


Each heteroclinic orbit connecting $P_n$ to $P_1$ (or equivalently $W^{l,s}_{P_{1}}$) is a \textit{basin boundary} (separator) between the PEN and OSC in the phase space. Specifically, if the initial data is above or below it, then the orbit is asymptotically defined as PEN or OSC. Such a basin boundary could be defined in the PDE setting, especially, the boundary denotes a hypermanifold including the connecting orbit between two unstable stationary pulses associated with the $P_{1}$ and $P_{n}$. This slightly influences the initial condition for the yielding PDE dynamics in a practical setting. In particular, the basin boundary rapidly shifted leftward as $\varepsilon$ increased, and thus, the selection of the initial condition was not proven under a priori condition, i.e., any stationary pulse-like profile located away from the bump belonged to which side of the basin boundary. Moreover, in case the $P_n$ was an unstable spiral, the two basins were interwoven in a rotational manner to predict the outcome for initial profiles similar to the stationary pulse associated with $P_n$. We believe that the subtle behaviors presented in Fig.\ref{spiral.eps} is rooted in this fact. Thus, this aspect is further detailed in Section 6. In contrast, such a sensitivity does not occur for the \textit{monotone} tail case, as remarked in and after Proposition 5.


Notably, the transition from PEN to OSC can be characterized based on the global behavior of the branches of HIOP in the PDE setting, as depicted in Fig.\ref{PEN-OSC.eps}(c). There are four solution branches: PEN (ocher) and OSC (cyan) are time-periodic ones, and $\mathrm{[I]}^\mathrm{S}_0$ (red) and $\mathrm{[I]}^\mathrm{S}_{1}$ (black) are stationary ones. Recall that $\mathrm{[I]}^\mathrm{S}_n$ stands for the n-shifted solution as in Fig.\ref{all.eps}. Both PEN and OSC branches contacted $\mathrm{[I]}^\mathrm{S}_{1}$ at B and C, respectively, and the OSC branch was initiated at the Hopf bifurcation point F from $\mathrm{[I]}^\mathrm{S}_0$ and persisted as a stable branch up to a saddle-node point around $\varepsilon =0.000567$. Global bifurcation is responsible for contacting each stationary branch, for instance, the period of PEN branch diverged as $\varepsilon$ approached the critical height and became homoclinic to $\mathrm{[I]}^\mathrm{S}_{1}$ at B. Recall that we earlier solved the PDE on a circle.

Similarly, the period of oscillating motion around $P_0$ diverged at C as well, and it emerged as a heteroclinic orbit connecting $P_1$ to $P_{-1}$. Note that this is one of the reconnections among the heteroclinic orbits that connected $P_1$ to the remaining critical points. The HIOP diagram clarified the transition occurred through the global bifurcation of PEN branch, where it ceases to exist, i.e., the point B in this case. Evidently, the PDE simulations, ODE dynamics, and the HIOP diagram were consistent each other. In summary, we have

\begin{proposition}
  \label{prop PEN2OSC}
  As $\varepsilon>0$ increased, the first transition occurred from PEN to OSC via the exchange of order between $\mathcal{O}_\mathrm{pulse}$ and $W^{l,s}_{P_{1}}$, i.e., the left stable manifold of $P_{1}$. The basin boundary between the two outputs for the general initial condition was given as $W^{l,s}_{P_{1}}$, which experienced infinitely many reconnections before coalescing into $\mathcal{O}_\mathrm{pulse}$. In particular, the destination of $W^{l,s}_{P_{1}}$ in time-reversal direction exchanged from $P_{-n}$ to $P_{-n-1}$ ($n$=1, 2, ...) with increasing $\varepsilon$. As $\varepsilon$ exceeded the critical height, $W^{l,s}_{P_{1}}$ was located above $\mathcal{O}_\mathrm{pulse}$. Refer to Fig.\ref{PEN-OSC.eps} for further details.  
\end{proposition}

\textbf{REMARK -monotone tail case-}
For the monotone tail case, no critical points existed on the $\alpha=0$, except that in the vicinity of the bump region, because the inhomogeneous term $f(p,d)$ has a constant sign away from the bump owing to the monotonicity. The basin boundary for the transition PEN--OSC can be characterized by the stable manifold of $P_1$ (if existed) in a similar manner. However, infinitely many reconnections did not occur for the monotone case, simply because no critical points existed far from bump region. In particular, all the orbits below the $\mathcal{O}_\mathrm{pulse}$ monotonically converged to $(p,\alpha) = (-\infty,0)$ in time-reversal direction ($t \rightarrow -\infty$) before the critical height, including the stable manifold of the saddle point $P_1$. This implied that any orbit starting from the horizontal line $\alpha=0$ was below the basin boundary at all instances, so it was unable to penetrate the bump region. Therefore, no ambiguity existed regarding the destination for the orbits starting from the horizontal line in the neighborhood of the transition point. They were all trapped by the bump heterogeneity and there was no sensitive dependence on the initial conditions.\\

\textbf{PEN$\rightarrow$REB}\\
For negative $\varepsilon$, $P_{0}$ became a saddle point and the first transition occurred from PEN to REB. The basic mechanism of this transition originated in the exchange of the order between $\mathcal{O}_\mathrm{pulse}$ and the stable manifold $W^{l,s}_{P_{0}}$ and a similar discussion was possible as in the PEN--OSC case of Proposition \ref{prop PEN2OSC}. The details were excluded for brevity to the reader. We will further discuss this case in Section 6, as it is relevant to the sensitive dependence observed in Fig.\ref{spiral.eps}. \\

Progressively, we study the cyclic transition REB--OSC--STA, considering the case in the negative $\varepsilon$, as the same argument is valid for positive $\varepsilon$. The cycle comprises the pinning process REB--OSC and a depinning process of STA-REB. Recall that the OSC--STA process was a local Hopf bifurcation, and thus, it was excluded from the current scope of discussion. We primarily focused on the pinning transition from REB to OSC that accounted for the most subtle case of the cycle.\\

\textbf{REB$\rightarrow$OSC}\\
The transition from REB to OSC occurred around $\varepsilon \approx -0.00313$ (see the point E in Fig.\ref{OSC-REB.eps}).
The behaviors of the traveling pulses of PDE are presented in (a) of Fig.\ref{OSC-REB.eps}, ODE flows (b), and the global structure of HIOP (c). The panel (a)-2 reflects that TPO remained in the vicinity of the unstable stationary pulse $\mathrm{[I]}^\mathrm{S}_{-2}$ for a long period, before travelling into the regime of OSC ((a)-1) as $\varepsilon$ decreased. This suggested that the relationship between the TPO and the stable manifold of $\mathrm{[I]}^\mathrm{S}_{-2}$ is vital for understanding the transition. Consequently, the transition occurred exactly at the value where $\mathcal{O}_\mathrm{pulse}$ coincided with $W^{r,s}_{P_{-2}}$ in the reduced ODE system. 
Prior to explaining it, we closely observed the HIOP diagram presented in Fig.\ref{OSC-REB.eps} (c), wherein two time-periodic branches and two stationary ones appeared as REB (purple), OSC (cyan) branches, and $\mathrm{[I]}^\mathrm{S}_{-1}$ (black), $\mathrm{[I]}^\mathrm{S}_{-2}$ (yellow) branches. The REB branch terminated at E: the stationary branch $\mathrm{[I]}^\mathrm{S}_{-2}$. As the branch approached E, the associated PDE solution spends most of the time near $\mathrm{[I]}^\mathrm{S}_{-2}$ as in (a)-2, and then rebounded and remained near $\mathrm{[I]}^\mathrm{S}_{2}$ (mirror image of $\mathrm{[I]}^\mathrm{S}_{-2}$) for a long period. Furthermore, this process was repeated, because TPO was progressing on a circle. Therefore, the REB branch eventually approached a heteroclinic orbit connecting $\mathrm{[I]}^\mathrm{S}_{-2}$
to $\mathrm{[I]}^\mathrm{S}_{2}$ as $\varepsilon$ tended to E. In contrast, the OSC branch emerged at A of $\mathrm{[I]}^\mathrm{S}_{-1}$ via Hopf bifurcation and underwent a saddle-node bifurcation before terminating at D of $\mathrm{[I]}^\mathrm{S}_{2}$ via homoclinic bifurcation. The associated homoclinic behavior cannot be observed in the PDE simulations, because it occurred for the unstable portion of the oscillatory motion arising at the saddle-node bifurcation.
As the Hopf bifurcation became supercritical, a bistable regime from B to the saddle-node point in which both the REB and OSC were prominently observable. The detailed mechanism of transition from REB to OSC will be clarified in the framework of the reduced ODE system as below.


The transition from REB to OSC can be characterized based on the alteration of the order from $\mathcal{O}_\mathrm{pulse}>W^{r,u}_{P_{-2}}>W^{r,s}_{P_{-2}}$ (REB) to $W^{r,s}_{P_{-2}}>\mathcal{O}_\mathrm{pulse}>W^{r,u}_{P_{-2}}$ (OSC) in which the inequalities were measured using their $\alpha$-coordinates at the location of the critical point $P_{-1}$. More importantly, in the REB regime ($\varepsilon \approx -0.001962$), the $\mathcal{O}_\mathrm{pulse}$ was greater than the stable and unstable manifolds of $P_{-2}$ located at the critical point $P_{-1}$ (refer to the magnified view presented in Fig.\ref{e-0.00196208903100908_REB.eps}). In contrast, the order varied in OSC regime immediately after the transition point $\varepsilon \approx -0.00313$ ( (b)-1 in Fig.\ref{OSC-REB.eps} and its magnification Fig.\ref{e-0.0021_OSC.eps}). 
According to the two inequalities, two steps were required for this alteration: $W^{r,s}_{P_{-2}}$ coincides with $W^{r,u}_{P_{-2}}$, and thereafter, $W^{r,s}_{P_{-2}}$ increased above $\mathcal{O}_\mathrm{pulse}$ as $\varepsilon$ decreased.
The first case is a homoclinic bifurcation at $P_{-2}$ (refer to Fig.\ref{e-0.00166147754_HOMO_REB.eps}) and the second case is the exchange of the order between $\mathcal{O}_\mathrm{pulse}$ and $W^{r,s}_{P_{-2}}$ (refer to Fig.\ref{e-0.0021_OSC.eps}). 
Homoclinic bifurcation consisting of $W^{r,s}_{P_{-2}}$ and $W^{r,u}_{P_{-2}}$ is already detected in HIOP diagram (refer to Fig.\ref{OSC-REB.eps} (c) at D). Note that the intersecting manner of these two manifolds can be explicitly computed using the saddle quantity $\sigma$ of $P_{-2}$ (refer to Proposition 3), which implied that $W^{r,s}_{P_{-2}}$ crossed the $W^{r,u}_{P_{-2}}$ transversally upward as $\varepsilon$ decreased (refer to, for instance, Section 6 of \cite{Kuznetsov} and \cite{Guckenheimer1983}). The homoclinic bifurcation is essential for opening the gate for $\mathcal{O}_\mathrm{pulse}$ to enter the basin of limit cycle. In particular, $W^{r,u}_{P_{-2}}$ blocked the $\mathcal{O}_\mathrm{pulse}$ from entering it before the homoclinic bifurcation. Immediately after the homoclinic bifurcation, the order among the three objects altered to $\mathcal{O}_\mathrm{pulse}>W^{r,s}_{P_{-2}}>W^{r,u}_{P_{-2n}}$. Thus, $\mathcal{O}_\mathrm{pulse}$ should be under $W^{r,s}_{P_{-2}}$ (green orbit) for $\mathcal{O}_\mathrm{pulse}$ to surround the limit cycle, as depicted in Fig.\ref{e-0.0021_OSC.eps}, because $W^{r,s}_{P_{-2}}$ denotes the basin boundary separating REB and OSC. Therefore, the transition point can be characterized by the $\varepsilon$-value, where $\mathcal{O}_\mathrm{pulse}$ coincided with the $W^{r,s}_{P_{-2}}$, wherein $\mathcal{O}_\mathrm{pulse}$ is a double-heteroclinic orbit connecting three critical points: $T^{+}_{-\infty}$, $P_{-2}$, and $T^{-}_{-\infty}$. Remarkably, the basin boundary $W^{r,s}_{P_{-2}}$ experienced infinitely many reconnections of heteroclinic orbits by exchanging the destinations among the critical points located on the left-hand side of $P_{-2}$. As such, an image of heteroclinic orbit connecting $P_{-2}$ and $P_{-6}$ is illustrate din Fig.\ref{e-0.00196208903100908_REB.eps}. 
The $\mathcal{O}_\mathrm{pulse}$ could enter the basin of the limit cycle around $P_{-1}$ after this occurrence. Therefore, we can conclude the following.

\begin{proposition}
  \label{prop REB2PIN}
  The transition occurred from REB to OSC in the vicinity of ${P_{-2n+1}} (n=1, ...)$, as $\varepsilon$ negatively increased, and it is characterized by the exchange of the order between $\mathcal{O}_\mathrm{pulse}$ and $W^{r,s}_{P_{-2n}}$. Namely, $\mathcal{O}_\mathrm{pulse}$ was above the $W^{r,s}_{P_{-2n}}$ in the upper $\alpha$-plane of the REB regime, but the order was reversed in the OSC regime. However, they coincided with each other at the critical height, and the corresponding basin boundary can be expressed as  $W^{r,s}_{P_{-2n}}$. To fulfill this exchange, a homoclinic bifurcation at $P_{-2n}$ must occur prior to the transition. More precisely, the exchange of order among the three objects, namely, $\mathcal{O}_\mathrm{pulse}>W^{r,u}_{P_{-2n}}>W^{r,s}_{P_{-2n}}$ (REB) altered into $W^{r,s}_{P_{-2n}}>\mathcal{O}_\mathrm{pulse}>W^{r,u}_{P_{-2n}}$ (OSC). There are two steps for this alteration: first $W^{r,s}_{P_{-2n}}$ coinciding with $W^{r,u}_{P_{-2n}}$ (homoclinic bifurcation as in Fig.\ref{e-0.00166147754_HOMO_REB.eps}), second, it passed transversely across $\mathcal{O}_\mathrm{pulse}$ (transition to OSC as in Fig.\ref{e-0.0021_OSC.eps}) as $\varepsilon$ negatively increased. An unstable time-periodic motion emerging at the saddle-node bifurcation, blocked the remaining orbits from entering the basin of the stable limit cycle, and it disappeared at the homoclinic bifurcation. Thereafter, $\mathcal{O}_\mathrm{pulse}$ could approach $W^{r,s}_{P_{-2n}}$ and transversely crossed it at the transition height. For the PDE setting, exactly the same mechanism controlled the transition by replacing the saddle $P_{-2n}$ using the unstable stationary pulse $\mathrm{[I]}^\mathrm{S}_{-2n}$, as depicted in Fig.\ref{OSC-REB.eps} (c). In particular, note that the associated time-periodic PDE branches, i.e., the REB (purple) and OSC (cyan) branches, contacted the $\mathrm{[I]}^\mathrm{S}_{-2}$ branch (yellow) in a double-heteroclinic and homoclinic manner, respectively, which corresponded to the above-mentioned two steps of order alteration in the ODE setting. 
\end{proposition}

\textbf{STA$\rightarrow$REB}\\
In a similar manner, the depinning process STA--REB can be understood by studying the interaction between $\mathcal{O}_\mathrm{pulse}$ and the stable manifold $W^{l,s}_{P_{-2n}}$ of $P_{-2n}$ upon the occurrence of REB around $P_{-2n}$. As observed in Fig.\ref{REB-STA.eps} (n=1), we could refer that the exchange between the $\mathcal{O}_\mathrm{pulse}$ and the stable manifold $W^{l,s}_{P_{-2}}$ (gray-dotted line) caused the transition. Furthermore, we can observe the infinitely many reconnections occurring between the saddle $P_{-2}$ and other critical points immediately before the transition, for instance, the ODE flow (b)-3 of Fig.\ref{REB-STA.eps} displayed a heteroclinic connection from $P_{-6}$ to $P_{-2}$. The HIOP structure of Fig.\ref{REB-STA.eps} (c) supported the above-mentioned mechanism, for instance, the REB branch (purple) converged on the $\mathrm{[I]}^\mathrm{S}_{-2}$ branch (yellow) at E, signifying that its period diverged and spent most of the time in the neighborhood of unstable stationary pulse $\mathrm{[I]}^\mathrm{S}_{-2}$ as $\varepsilon$ approached the critical height. Thus, the following proposition was entailed in this discussion.






\begin{proposition}
  \label{prop Pin2REB}
  For the transition from STA around ${P_{-2n+1}} (n=1, ...)$ to REB as $\varepsilon$ negatively increased, the exchange of the order occurred between $\mathcal{O}_\mathrm{pulse}$ and $W^{l,s}_{P_{-2n}}$. In particular, $\mathcal{O}_\mathrm{pulse}$ was above the $W^{l,s}_{P_{-2n}}$ in the upper $\alpha$-plane for the STA regime, but the order switched in the REB regime. They coincided each other at the critical height and $W^{l,s}_{P_{-2n}}$ acted as the basin boundary between the STA and REB. For the PDE setting, the associated HIOP diagram presented in Fig.\ref{REB-STA.eps} (bottom) depicted a consistent scenario, as detailed earlier. For further details, refer to the caption of Fig.\ref{REB-STA.eps}. 
\end{proposition}

We summarize the discussion for the cyclic transition REB--OSC--STA as follows. \\
\begin{proposition}
  \label{prop_summary}
  (1) As the bump height $\varepsilon>0$
  ($\varepsilon<0$) increased (decreased), the traveling pulse orbit $\mathcal{O}_\mathrm{pulse}$ of Eq.(\ref{eq3.50}) experienced infinitely many repeated cycles of outputs (REB--OSC--STA), except at the vicinity of the PEN regime (small height). \\
  (2) The basin boundary separating the two regimes can be expressed as the stable manifold of the relevant saddle point located proximate to the event.
  The exchange of order between $\mathcal{O}_\mathrm{pulse}$ and the basin boundary controlled the transition among the cycle.\\ 
  (3) The basin boundary relevant to the transition underwent infinitely many reconnections via switching heteroclinic orbits among the critical points as $\varepsilon$ approached the critical height. \\
  (4) At the transition from REB to OSC, the exchange of order among the three objects occurred, i.e, $\mathcal{O}_\mathrm{pulse}>W^{r,u}_{P_{-2n}}>W^{r,s}_{P_{-2n}}$ (REB) altered into $W^{r,s}_{P_{-2n}}>\mathcal{O}_\mathrm{pulse}>W^{r,u}_{P_{-2n}}$ (OSC). This alteration comprised two steps: homoclinic bifurcation to the saddle point $P_{-2n}$, which was vital for opening a gate for $\mathcal{O}_\mathrm{pulse}$ to enter the basin of the limit cycle (OSC), and the exchange of order between the $\mathcal{O}_\mathrm{pulse}$ and $W^{r,s}_{P_{-2n}}$.\\
  (5) The transition from OSC to STA emerged from a local Hopf bifurcation of the relevant nonsaddle critical point.\\
  (6) The global structure of HIOP for the PDE system and the response of traveling pulse against the bump were consistent with all the above results. \\
\end{proposition}

\section{Sensitive dependence on initial conditions and structure of basin boundary}

\begin{figure}
  \centering
  \includegraphics[width=.8\hsize]{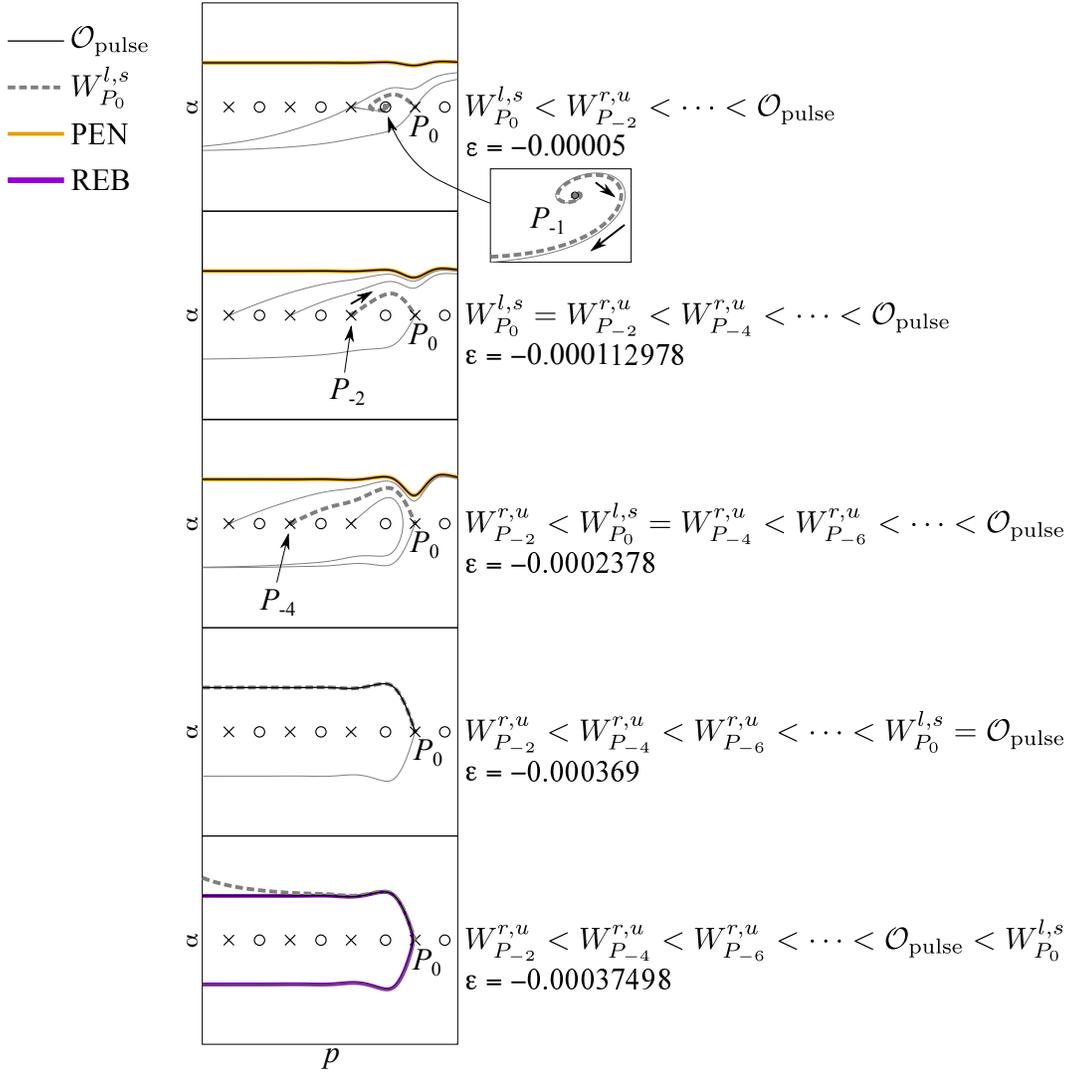}
  \caption{
    Transition process from PEN to REB: sensitive dependence on
    initial condition and structure of basin boundary $W^{l,s}_{P_{0}}$.
    The left-hand side stable manifold of $P_{0}$ is vital as a separator between PEN and REB.  
    As $\varepsilon$ decreased, the $W^{l,s}_{P_{0}}$ experienced infinitely many reconnections among the critical points. Additionally, the destination of $W^{l,s}_{P_{0}}$ in time-reversal direction at the top panel corresponded to the unstable spiral point adjacent to $P_{0}$, and thereafter, it altered to the following saddle point at $\varepsilon = -0.000112978$, and so on. As the destination was a critical point of the spiral type, the basin between the PEN and REB was categorized into two parts in a spiral manner, so we observed a subtle sensitivity to the initial condition, as depicted in Fig.\ref{spiral.eps}. Note that the connection to such an unstable critical point exhibited an open property, namely, it persisted even for slightly smaller $\varepsilon$ up to $\varepsilon=-0.000112978$.   }
  \label{PEN-REB_scenario.eps}
\end{figure}

In this section, we return to the issue of sensitive dependence on the initial conditions, as demonstrated in Fig.\ref{spiral.eps}. As the bump height belonged inside the PEN regime in the phase diagram of Fig.\ref{phase_PDE.eps}, the TPO started from $T^{+}_{-\infty}$ and travelled across the bump even if a small perturbation was added to it. Nevertheless, the two orbits depicted in Fig.\ref{spiral.eps} were extremely close initially, but they behaved distinctly: one (blue) travelled across the bump and the other (red) eventually rebounded. Moreover, both the initial profiles were extremely proximate to an unstable stationary pulse pinned at the bump heterogeneity, namely, $\mathrm{[I]}^\mathrm{S}_{-1}$ of Fig.\ref{S1.eps}. This strongly suggested that $\mathrm{[I]}^\mathrm{S}_{-1}$ belongs to the basin boundary separating PEN and REB, and its unstable manifold behaved in a spiral manner around it. Thus, we clarified the details in the framework of reduced ODE system as follows.

The basin boundary for general initial data is vital for understanding the present situation. Herein, we focus on the transition from PEN to REB for negative $\varepsilon$, which is relevant to Fig.\ref{spiral.eps}. The meaning of sensitive dependence around a specific initial point $P_{-1}$ in the reduced ODE setting can be described as follows (see Fig.\ref{PEN-REB_scenario.eps}): for any small neighborhood of $P_{-1}$, at least two initial values can be determined on the half-ray in any direction originating at $P_{-1}$, which may yield various outputs. Notably, we added the extra ``half-ray'' condition to the usual ``any small'' neighborhood condition. This is partially because we aim to avoid the case in which the basin boundary crosses an arbitrarily selected half-ray originating at $P_{-1}$ only finitely many times, such as in the case connecting it to a real hyperbolic saddle point with a well-defined tangent space along the unstable manifold. As discussed earlier, the height represented in Fig.\ref{spiral.eps} belonged inside of the PEN regime ($\varepsilon = -0.00005$) to ensure that the TPO and nearby orbits do not exhibit any sensitivity, regardless of the variation in its height. Moreover, A detailed reconnection process of the basin boundary between PEN and REB for negative $\varepsilon$ is presented in Fig.\ref{PEN-REB_scenario.eps}. The stable manifold $W^{l,s}_{P_{0}}$ of $P_0$ influenced the basin boundary, and it travelled leftward as $\varepsilon$ becomes negatively larger and experienced infinitely many reconnections among the critical points. Eventually, it coincided with $\mathcal{O}_\mathrm{pulse}$ at the critical transition height ($\varepsilon = -0.000369$). Furthermore, the sensitive dependence associated with the results presented in Fig.\ref{spiral.eps} occurred when the destination of $W^{l,s}_{P_{0}}$ in the time-reversal direction is a critical point of unstable spiral type, such as that in the top panel of Fig.\ref{PEN-REB_scenario.eps} (refer to the inlet of its magnification). Although the connection to a real saddle point occurred infinitely many times during the reconnection process, the half-ray condition was not satisfied therein, because the basin was classified into two parts along the right unstable direction in a small neighborhood of the saddle point. Note that the connection to unstable spiral exhibtied an open property, i.e., it persisted in an open set of $\varepsilon$, and there existed an open interval of $\varepsilon$ between the top and second panels in Fig.\ref{PEN-REB_scenario.eps}. In contrast, the connection to the real saddle point generically occurred only for discrete values of $\varepsilon$.

The subtleness observed in Fig.\ref{transition.eps} is essentially different from the above-mentioned case, wherein the height is extremely proximate to a critical transition, e.g., from PEN to REB (leftmost). More importantly, the TPO coincided with the associated basin boundary at the transition height, so we can easily predict the fate of the orbit, whether it was above or below the $\mathcal{O}_\mathrm{pulse}$ at least in the ODE setting provided the exact initial point is provided in advance. In a real simulation of PDE on a circle, as depicted in Fig.\ref{transition.eps}, the outcome after collision could not be conveniently predicted, because it relied on the length of the circle and an approximation of the initial profile. Moreover, infinitely many times reconnection of basin boundary occurred as the height approached the transition in a subtle manner to locate the approximated initial condition on a particular side of the basin boundary. Recall that $\mathrm{[I]}^\mathrm{T}$ resembles $\mathrm{[I]}^\mathrm{S}$  as in Fig.\ref{snakes_and_ladders.eps}; therefore, $\mathrm{[I]}^\mathrm{S}_{-n}$ for large $n$ was proximate to $\mathrm{[I]}^\mathrm{T}$ in case the $\tau$ is in the vicinity of the drift bifurcation point $\tau_c$.

The reduction to ODE system immensely aided the understanding of the subtle aspects related to the behaviors of TPO of the original three-component reaction diffusion system, which originated from the interaction between the oscillating tails and heterogeneity.
As such, the insightful observations obtained for the 1D system can enlighten more complex behaviors of the traveling spots in 2D space.\\

\section{Conclusion and outlook}

In this study, we characterized the interplay between the TPO and the heterogeneity of bump type for the generalized three-component FitzHugh--Nagumo Eqs.(\ref{eq2.01}). First, we numerically investigated the existence and stability properties of the SPO and TPO in a homogeneous space. Consequently, we observed that the SPOs formed a ``snaky'' structure, as depicted in Fig.\ref{snakes_and_ladders.eps}. In contrast, the TPOs constituted a ``figure-eight-like stack of isolas'' that was located adjacent to the snaky structure of the SPO, as portrayed in Fig.\ref{isola.eps}. Thus, we adopted the parameter $\kappa_1$ (voltage-difference in gas-discharge model) as a bifurcation parameter, and observed a drift bifurcation from the SPO to TPO by introducing another parameter $\tau$, and these two sheets of solutions formed a gutter-like-structure in $(\kappa_1, \tau)$-parameter space  (refer to Fig.\ref{bif3.eps}). The edge of the isola structure determined the admissible interval on which the stable TPO is observed. Outside this interval, the destruction or replication of the pulses was typically observed as reported in \cite{NU_PhysicaD_1999,Burke_thesis,doi:10.1146/annurev-conmatphys-031214-014514}.

\begin{figure}
  \centering
  \includegraphics[width=.8\hsize]{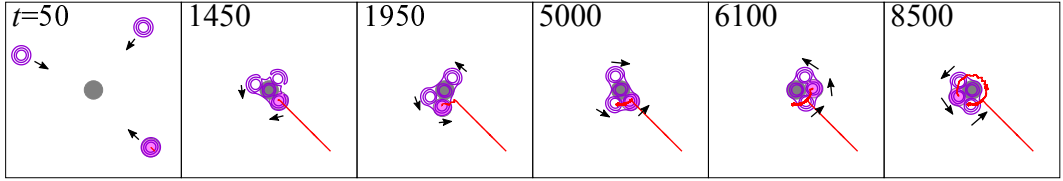}
  \caption{
    Interaction of three traveling spots with disk heterogeneity. They collided with the disk and adhered to the surface. Initially, they were separated, but subsequently, they adhered together on the outer surface of the disk.
  }
  \label{spot.eps}
\end{figure}

For the heterogeneous problem, a nonlocal interaction appeared between the tails of TPO and the heterogeneity of bump type. As the translation invariance was lost, new nonconstant background states emerged, and their continuation with respect to the height $\varepsilon$ of the bump formed a variety of HIOP solution branches, including countably many stationary pulse solutions attached to the bump at appropriate distance. Most of these structure were unstable, but a certain number recovered their stabilities with increasing bump height and resulted in a stationary pinning at that location. In contrast to the monotone tail case, the TPO exhibited a variety of dynamic response, including pinning and depinning processes in addition to PEN and REB. Recall that TPO can be regarded as a solution located at $-\infty$ and traveling toward the heterogeneity. Moreover, stable/unstable manifolds of HIOP interacted with TPO in a complex manner, which caused a subtle dependence on the initial condition and created challenges in predicting the behavior after collision, even in one-dimensional space. Additionally, an initial condition could be set in vicinity to one of the (unstable) stationary pulse, and generally, it was extremely subtle to judge the resulting orbit, as illustrated in Fig.\ref{spiral.eps}. Nevertheless, for the 1D case, a systematic global exploration of the solution branches was induced by heterogeneities, and the reduction method to finite-dimensional ODEs (\ref{eq3.50}) allowed us to clarify such a subtle dependence on the initial conditions along with the detailed mechanism of the transitions from PEN to pinning and pinning to REB based on a dynamical system perspective. In terms of the reduced ODE system, the basin boundary between the two distinct outputs is characterized by a special heteroclinic orbit connecting two critical points, which underwent infinitely many successive reconnection of the heteroclinic orbits among the critical points, as the strength of heterogeneity varied. The subtle dependence of the initial condition was rooted in this fact, and the basin boundary between PEN and REB can be expressed using the heteroclinic orbit (gray-dotted line) surrounding the critical point of the spiral type (in time-reversal direction), as depicted in Fig.\ref{PEN-REB_scenario.eps}. Although the two initial profiles around the unstable stationary pulse associated with the critical point in Fig.\ref{spiral.eps} were almost indistinguishable, they separated into two distinct orbits after a long period.

Remarkably, all the asymptotic states of the TPO after collisions were contained in the HIOP to enable the prediction of the outcome of TPO after colliding with the bump in case we carefully analyzed the HIOP for a given parameter, unless the initial condition was proximate to the basin boundary of the outputs. Note that HIOP structure can be extended outside the admissible interval, which indicated several useful facts for understanding the complex dynamics produced by the interaction between the TPO and the heterogeneity. A simple example entails the emergence of the stable two-peak HIOP, 
such as that presented in Fig.\ref{m.eps}-(j), which already existed within the admissible interval and could be extended further to larger $\varepsilon$ (refer to Fig.\ref{S2.eps}). Such a multi-peak pulse is beyond the scope of ODE approach and reflects the strong interaction between the TPO and heterogeneity. Ultimately, we listed the several open questions for the localized moving patterns with oscillatory tails.\\

\begin{itemize}
\item{Interaction between TPOs in homogeneous space: We primarily focused on the interaction between TPO and heterogeneity; however, the interaction between TPOs in homogeneous space is an interesting concern. In context, various types of localized multipeak stationary and traveling pulses may appear through the collision process of TPO with single peak. A subtle aspect regarding this is that several types of multipeak solutions existed even if the number of peaks was fixed owing to the nature of the oscillatory tail. For instance, suppose we assume a head-on collision of the two traveling pulses and formed a bound-state with two peaks, then the distance between the two peaks cannot be uniquely determined and may depend on the velocity of the pulse. Note that we have already referred several distinct types of two-peak solutions in Fig.\ref{S2.eps}--Fig.\ref{S4.eps}. Possibly, there may appear countably-many two-peak solutions with vavrying peak widths, but the stable ones are limited subject to an unknown selection rule.}

\item{Multipeak pulses in heterogeneous space: As we referred to the snaky-structure and figure-eight-like stack of isolas of traveling pulses, we referred to a family of multipeak SPOs and TPOs as well. However, the exploration of a systematic analysis for the dynamics of multipeak traveling pulses remains a topic of inquisitivity, including their outputs after collision with bump. In context, certain aspects of the global behavior of stationary multipeak solutions were already uncovered in Section 4. We believe that the global structure of the associated HIOP for multipeak pulses is vital for understanding the behavior of multipeak traveling pulses encountering heterogeneities.
}

\item{Higher-dimensional case: We employ the three-component system (\ref{eq2.01}), because it supported the coexistence of stable traveling spots or lumps in 2D or 3D space, which was not favorable for the class of the two-component system. Moreover, the study of interaction between these objects and their dynamics in heterogeneous media is still a research hotspot. Numerically, we observed a variety of interesting patterns such as rotating-bound state and crystal-like structure \cite{Liehr2013} including traveling patterns of various molecular shapes. Additionally, the interaction between the traveling spots and heterogeneity displayed enriched and complex dynamics that became more complicated than the 1D case, e.g., an interaction between the three traveling spots and disk-shaped heterogeneity is presented in Fig.\ref{spot.eps}. They collided with the disk, and upon adjusting the distance between them, they eventually settled to a rotating motion surrounding the disk. Thus, the reduction to ODEs near the drift bifurcation point similar to the 1D case can aid in understanding the dynamics. }

\item{ODE reduction for higher-dimensional space: An ODE system similar to (\ref{eq3.50}) can be derived for traveling spots or lumps in higher dimensional heterogeneous space, as already reported in \cite{Liehr2013}. However, the challenge lies in the complexity of behaviors of stable/unstable manifolds for the infinitely many equilibria distributed across the entire phase space as the parameters contained in the heterogeneity varied. A simple geometry representing a rod or disk is an existing challenge for completing the task for heterogeneity.}

\item{Wave-particle duality: As discussed in Section 1, the duality of the wave and particle properties were reflected in the macroscopic objects in case they possessed oscillatory tails. More importantly, we were strongly motivated by the seminal papers of Couder et al. \cite{PhysRevLett.97.154101} studying bouncing droplets on a liquid surface, which elaborately demonstrated this duality through both experiments and theoretical considerations. Therefore, we believe that the system developed in this study inherits such a property to a certain extent, and defining the wave--particle duality still remains as a challenge in the framework of reaction diffusion systems.}

\end{itemize}

\section*{Acknowlwdgements}
Yasumasa Nishiura gratefully acknowledges the support by JSPS KAKENHI Grant
number JP20K20341.

\appendix
\section{Appendix}
\label{appendixA}
\subsection{Derivation of reduced ODE}
We derive formally the ODE system (25) from the original system (6).
Recalling the multiple time-scale expansion Eqs.(\ref{eq3.19}) and substituting the propagator mode $\mathcal{P}_{x}$ defined by Eq.(\ref{eq3.10}) into Eq.(\ref{eq3.19}), we have
\begin{equation}
  \label{eq3.21}
  \left\{\begin{array}{rcl}
  \tilde u(x,t)&=&\overline{u}(x-p(T_1,T_2,T_3))\\
  &&+\delta^2r_u(x-p(T_1,T_2,T_3))\\
  &&+\delta^3R_u(x)\\
  \tilde v(x,t)&=&\overline{u}(x-p(T_1,T_2,T_3))\\
  &&+\delta\alpha(T_1,T_2)\frac{\partial\overline{u}}{\partial x}
  (x-p(T_1,T_2,T_3))\\
  &&+\delta^2r_v(x-p(T_1,T_2,T_3))\\
  &&+\delta^3R_v(x)
  \end{array}\right.
\end{equation}
Note that the stationary solution satisfies $\overline{u}(x) = \overline{v}(x)$ and the corrections terms $\vec r$ and $\vec R$ belong to the orthogonal
subspace of the principal part spanned by $\{\mathcal G_x,\mathcal P_x\}$.
Thus $%
\left<\mathcal{G}_x^*,\vec r\right>=\left<\mathcal{P}_x^*,\vec r\right>=%
\left<\mathcal{G}_x^*,\vec R\right>=\left<\mathcal{P}_x^*,\vec R\right>=0$.

Introducing a critical parameter $\widetilde{\tau}$ by
\begin{equation}
  \label{eq3.22}
  \delta^{2}\widetilde{\tau}=\tau-\tau_{c}
\end{equation}
and Taylor expansion to the nonlinearity $N(u)$ up to the first order
is conducted as
\begin{equation}
  \label{eq3.23}
  N(\overline{u}+\delta^{2}r_{u}+\delta^{3}R_{u})=N(\overline{u})
  +N'(\overline{u})(\delta^{2}r_{u}+\delta^{3}R_{u})
\end{equation}
and rescale $\widehat{\kappa}_{1}(x)$ as
\begin{equation}
  \label{eq3.24}
  \widehat{\kappa}_{1}(x)=\delta^{2}\widetilde{\kappa}_{1}(x)
\end{equation}
Substituting $T_{i}=\delta^{i}t$ and Eqs.(\ref{eq3.21})--(\ref{eq3.24}) into
Eq.(\ref{eq3.06}),
we have
\begin{equation}
  \label{eq3.25}
  \begin{array}{l}
    -\frac{\partial}{\partial x}\overline{u}\left\{
    \delta\frac{\partial}{\partial T_{1}} p
    +\delta^{2}\frac{\partial}{\partial T_{2}} p
    +\delta^{3}\frac{\partial}{\partial T_{3}} p\right\}
    -\delta^{3}\frac{\partial r_u}{\partial x}
    \frac{\partial p}{\partial T_1}\\
    =\mathcal{L}\overline{u}+N(\overline{u})-\kappa_{3}\overline{u}
    -\delta\kappa_{3}\alpha\frac{\partial}{\partial x}\overline{u}\\
    +\delta^{2}\mathcal{L}r_{u}+\delta^{2}N'(\overline{u})r_{u}
    -\delta^{2}\kappa_{3}r_{v}+\delta^{2}\widetilde{\kappa}_{1}\\
    +\delta^{3}\mathcal{L}R_{u}+\delta^{3}N'(\overline{u})R_{u}
    -\delta^{3}\kappa_{3}R_{v}
  \end{array}
\end{equation}
\begin{equation}
  \label{eq3.26}
  \begin{array}{l}
    (\delta^{2}\widetilde{\tau}/\tau_c+1)[
      -\frac{\partial}{\partial x}\overline{u}\{
      \delta\frac{\partial}{\partial T_{1}}p
      +\delta^{2}\frac{\partial}{\partial T_{2}}p
      +\delta^{3}\frac{\partial}{\partial T_{3}}p\\
      -\delta^{2}\frac{\partial}{\partial T_{1}}\alpha
      -\delta^{3}\frac{\partial}{\partial T_{2}}\alpha
      \}\\
      -\frac{\partial^{2}}{\partial x^{2}}\overline{u}\{
      \delta^{2}\alpha\frac{\partial}{\partial T_{1}}p
      +\delta^{3}\alpha\frac{\partial}{\partial T_{2}}p
      +\delta^{4}\alpha\frac{\partial}{\partial T_{3}}p\}\\
      +\delta^{3}\frac{\partial r_v}{\partial x}
      \frac{\partial p}{\partial T_1}]\\
    =\delta^{2}\frac{r_{u}}{\tau_c}+\delta^{3}\frac{R_{u}}{\tau_c}
    -\delta\frac{\alpha}{\tau_c}\frac{\partial}{\partial x}\overline{u}\\
    -\delta^{2}\frac{r_{v}}{\tau_c}-\delta^{3}\frac{R_{v}}{\tau_c}
  \end{array}
\end{equation}
Because these equations have to be fulfilled for all time-scales
simultaneously, we can sort them by orders of $\delta$:

\noindent$O(1)$
\begin{equation}
  \label{eq3.27}
  \vec 0=\left(\begin{array}{c}
    \mathcal{L}\overline{u}+N(\overline{u})-\kappa_{3}\overline{u}\\
    0
  \end{array}\right)
\end{equation}
\noindent$O(\delta)$
\begin{equation}
  \label{eq3.28}
  \left(\begin{array}{c}
    -\frac{\partial}{\partial T_{1}}p\\
    -\frac{\partial}{\partial T_{1}}p
  \end{array}\right)\frac{\partial}{\partial x} \overline{u}
  =\left(\begin{array}{c}
    -\kappa_{3}\alpha\\
    -\frac{\alpha}{\tau_c}
  \end{array} \right)\frac{\partial}{\partial x}\overline{u}
\end{equation}
\noindent$O(\delta^2)$
\begin{equation}
  \label{eq3.29}
  \begin{array}{l}
    -\left(\begin{array}{c}
      \frac{\partial}{\partial T_{2}} p\\
      \frac{\partial}{\partial T_{2}} p-\frac{\partial}{\partial T_{1}}\alpha
    \end{array}\right)\frac{\partial}{\partial x}\overline{u}
    -\left(\begin{array}{c}
      0\\
      \alpha\frac{\partial}{\partial T_{1}}p
    \end{array}\right)\frac{\partial^{2}}{\partial x^{2}}\overline{u}\\
    =\mathcal{D}\left(\begin{array}{c}
      r_{u}\\
      r_{v}
    \end{array}\right)+\left(\begin{array}{c}
      \widetilde{\kappa}_{1}\\
      0
    \end{array}\right)
  \end{array}
\end{equation}
\noindent$O(\delta^3)$
\begin{equation}
  \label{eq3.30}
  \begin{array}{l}
    \left(\begin{array}{c}
      -\frac{\partial}{\partial T_{3}}p\\
      -\frac{\widetilde{\tau}}{\tau_{c}}\frac{\partial}{\partial T_{1}}p
      -\frac{\partial}{\partial T_{3}}p
      +\frac{\partial}{\partial T_{2}}\alpha
    \end{array}\right)\frac{\partial}{\partial x}\overline{u}\\
    -\left(\begin{array}{c}
      0\\
      \alpha\frac{\partial}{\partial T_{2}}p
    \end{array}\right)\frac{\partial^{2}}{\partial x^{2}}\overline{u}
    +\left(\begin{array}{c}
      \frac{\partial r_u}{\partial x}\frac{\partial p}{\partial T_1}\\
      \frac{\partial r_v}{\partial x}\frac{\partial p}{\partial T_1}
    \end{array}\right)
    =\mathcal{D}\left(\begin{array}{c}
      R_{u}\\
      R_{v}
    \end{array}\right)
  \end{array}
\end{equation}
with $\mathcal{D}$ denoting the linear operator
(\ref{eq3.07}) at the drift bifurcation point.

Equation of $O(1)$ is automatically fulfilled.
Projection of $\mathcal{O}(\delta)$ onto $\kappa_3\mathcal{P}_{x}^{*}$ yields
\[
-\frac{\partial p}{\partial T_1}
\left<\left(\frac{\partial\overline{u}}{\partial x}\right)^2\right>
=-\kappa_3\alpha
\left<\left(\frac{\partial\overline{u}}{\partial x}\right)^2\right>
\]
and thus
\begin{equation}
  \label{eq3.31}
  \frac{\partial p}{\partial T_1}
  =\kappa_3\alpha
\end{equation}
Projection of $\mathcal{O}(\delta)$ onto $\mathcal{G}_{x}^{*}$ yields
\[
0=\left(-\kappa_3\alpha+\frac{\alpha}{\tau_c}\right)
\left<\left(\frac{\partial\overline{u}}{\partial x}\right)^2\right>
\]
and hence
\begin{equation}
  \label{eq3.33}
  \tau_{c}=\frac{1}{\kappa_{3}}
\end{equation}
Projection of $\mathcal{O}(\delta^{2})$ onto $\kappa_3\mathcal{P}_{x}^{*}$
yields
\[
\displaystyle
-\frac{\partial p}{\partial T_2}
\left<\left(\frac{\partial\overline{u}}{\partial x}\right)^2\right>
=\left<\frac{\partial\overline{u}}{\partial x},\tilde\kappa_1\right>
+\kappa_3\left<\mathcal P_x^*,\mathcal D\vec r\right>
\]
and thus
\begin{equation}
  \label{eq3.34}
  \frac{\partial p}{\partial T_2}
  =-\frac{
    \left<\frac{\partial}{
      \partial x}\overline{u},\widetilde{\kappa}_{1}\right>}{
    \left<
    (\frac{\partial}{\partial x} \overline{u})^{2}\right>}
\end{equation}
because
$%
\left<\mathcal P_x^*,\mathcal D\vec r\right>=%
\left<\mathcal D^*\mathcal P_x^*,\vec r\right>=%
\left<\mathcal G_x^*,\vec r\right>=0$.
Projection of $\mathcal{O}(\delta^{2})$ onto $\mathcal{G}_{x}^{*}$
yields
\[
-\frac{\partial\alpha}{\partial T_1}
\left<\left(\frac{\partial\overline{u}}{\partial x}\right)^2\right>
+\left<
\frac{\partial  \overline{u}}{\partial x  },
\frac{\partial^2\overline{u}}{\partial x^2}
\right>
=\left<\mathcal G_x^*,\mathcal D\vec r\right>
+\left<\frac{\partial\overline{u}}{\partial x},\tilde\kappa_1\right>
\]
and thus
\begin{equation}
  \label{eq3.35}
  \frac{\partial}{\partial T_{1}}\alpha
  =-\frac{
    \left<
    \frac{\partial}{\partial x}\overline{u},\widetilde{\kappa}_{1}\right>}{
    \left<
    (\frac{\partial}{\partial x} \overline{u})^{2}\right>}
\end{equation}
because
\begin{equation}
  \label{eq3.51}
  \left<
  \frac{\partial  \overline{u}}{\partial x  },
  \frac{\partial^2\overline{u}}{\partial x^2}
  \right>
  =0
\end{equation}
and $%
\left<\mathcal G_x^*,\mathcal D\vec r\right>=%
\left<\mathcal D^*\mathcal G_x^*,\vec r\right>=%
\left<\vec 0,\vec r\right>=0%
$.
Projecting $\mathcal{O}(\delta^3)$ onto $\kappa_3\mathcal{P}_{x}^{*}$
and substituting Eq.(\ref{eq3.31}), we have
\[
-\frac{\partial p}{\partial T_3}
\left<\left(\frac{\partial\overline{u}}{\partial x}\right)^2\right>
+\kappa_3\alpha\left<
\frac{\partial\overline{u}}{\partial x},
\frac{\partial r_u        }{\partial x}
\right>
=\kappa_3\left<\mathcal P_x^*,\mathcal D\vec R\right>
\]
and thus
\begin{equation}
  \label{eq3.32}
  \frac{\partial p}{\partial T_3}
  \left<\left(\frac{\partial\overline{u}}{\partial x}\right)^2\right>
  =C_3\alpha
\end{equation}
where
$%
C_3=\kappa_3\left<%
\frac{\partial\overline{u}}{\partial x},%
\frac{\partial r_u        }{\partial x}%
\right>%
$.
Note that
$%
\left<\mathcal P_x^*,\mathcal D\vec R\right>=%
\left<\mathcal D^*\mathcal P_x^*,\vec R\right>=%
\left<\mathcal G_x^*,\vec R\right>=0$.
Projecting $\mathcal{O}(\delta^{3})$ onto $\mathcal{G}_{x}^{*}$ and
substituting Eqs.(\ref{eq3.31}), (\ref{eq3.33}), and (\ref{eq3.51}), we have
\begin{equation}
  \label{eq3.36}
  \begin{array}{l}
    \displaystyle
    \left(-\tilde\tau\kappa_3^2\alpha+\frac{\partial\alpha}{\partial T_2}\right)
    \left<\left(\frac{\partial\overline{u}}{\partial x}\right)^2\right>
    +\kappa_3\alpha\left<
    \mathcal G_x^*,\frac{\partial\vec r}{\partial x}
    \right>\\
    \displaystyle
    =\left<\mathcal G_x^*,\mathcal D\vec R\right>
  \end{array}
\end{equation}
Note that the right-hand side is zero because $%
\left<\mathcal G_x^*,\mathcal D\vec R\right>=%
\left<\mathcal D^*\mathcal G_x^*,\vec R\right>=%
\left<\vec 0,\vec R\right>=0%
$.
To estimate the second term of the left-hand side,
$T_1$-derivative of $O(\delta^2)$ equations are considered as
\begin{equation}
  \label{eq3.37}
  \begin{array}{l}
    -\left(\begin{array}{c}
      \kappa_3\frac{\partial\alpha}{\partial T_{2}}\\
      0
    \end{array}\right)\frac{\partial}{\partial x}\overline{u}
    +\left(\begin{array}{c}
      0\\
      \kappa_3^2\alpha^3
    \end{array}\right)\frac{\partial^{3}}{\partial x^{3}}\overline{u}\\
    =\kappa_3\alpha\left(\begin{array}{c}
      N''\frac{\partial\overline{u}}{\partial x}r_u\\0
    \end{array}\right)
    -\kappa_3\alpha\mathcal{D}\frac{\partial\vec r}{\partial x}
  \end{array}
\end{equation}
Note that terms which contain $\frac{\partial^2\overline{u}}{\partial x^2}$
is ignored because of Eq.(\ref{eq3.51}).
Projection of Eq.(\ref{eq3.37}) onto $\kappa_3\mathcal P_x^*$ yields
\[
\begin{array}{l}
  \displaystyle
  -\kappa_3\frac{\partial\alpha}{\partial T_2}
  \left<\left(\frac{\partial\overline{u}}{\partial x}\right)^2\right>\\
  \displaystyle
  =\kappa_3\alpha\left<
  \frac{\partial\overline{u}}{\partial x},
  N''\frac{\partial\overline{u}}{\partial x}r_u
  \right>
  -\kappa_3^2\alpha\left<
  \mathcal G_x^*,
  \frac{\partial\vec r}{\partial x}
  \right>
\end{array}
\]
and projection of Eq.(\ref{eq3.37}) onto $\mathcal G_x^*$ yields
\[
\begin{array}{l}
  \displaystyle
  -\kappa_3\frac{\partial\alpha}{\partial T_2}
  \left<\left(\frac{\partial\overline{u}}{\partial x}\right)^2\right>
  -\kappa_3^2\alpha^3\left<
  \frac{\partial  \overline{u}}{\partial x  },
  \frac{\partial^3\overline{u}}{\partial x^3}
  \right>\\
  \displaystyle
  =\kappa_3\alpha\left<
  \frac{\partial\overline{u}}{\partial x},
  N''\frac{\partial\overline{u}}{\partial x}r_u
  \right>
\end{array}
\]
and therefore
\begin{equation}
  \label{eq3.38}
  \kappa_3\alpha\left<
  \mathcal G_x^*,
  \frac{\partial\vec r}{\partial x}
  \right>
  =-\kappa_3\alpha^3\left<
  \frac{\partial  \overline{u}}{\partial x  },
  \frac{\partial^3\overline{u}}{\partial x^3}
  \right>
\end{equation}
Substituting Eq.(\ref{eq3.38}) into Eq.(\ref{eq3.36}), we obtain
\begin{equation}
  \label{eq3.41}
  \frac{\partial}{\partial T_{2}}\alpha=\widetilde{\tau}\kappa_{3}^{2}\alpha
  +\frac{
    \left<\frac{\partial}{\partial x}\overline{u},
    \frac{\partial^{3}}{\partial x^{3}}\overline{u}\right>}{
    \left<(\frac{\partial}{\partial x}\overline{u})^{2}\right>
  }\kappa_{3}\alpha^{3}
\end{equation}

Now we are ready to compute the time derivatives of $p$ and $\alpha$ by using the relations
(\ref{eq3.31}), (\ref{eq3.34}), (\ref{eq3.35}), (\ref{eq3.32}), and
(\ref{eq3.41}).
\begin{equation}
  \label{eq3.42}
  \left\{\begin{array}{lcl}
  \frac{\partial p}{\partial T_{1}}=\kappa_3\alpha\\
  \frac{\partial p}{\partial T_{2}}=\frac{\partial a}{\partial T_{1}}=
  -\frac{\left<
    \frac{\partial}{\partial x}\overline{u},\widetilde{\kappa}_{1}\right>}{
    \left<(\frac{\partial}{\partial x} \overline{u})^{2}\right>}\\
  \frac{\partial\alpha}{\partial T_{2}}=
  \widetilde{\tau}\kappa_{3}^{2}\alpha
  +\frac{\left<\frac{\partial}{\partial x}\overline{u}
    \frac{\partial^{3}}{\partial x^{3}}\overline{u}\right>}{
    \left<\left(\frac{\partial\overline{u}}{\partial x}\right)^{2}\right>}
  \kappa_{3}\alpha^{3}\\
  \frac{\partial p}{\partial T_3}=C_3\alpha
  \end{array}\right.
\end{equation}
Note that, using integration by parts, we can prove that
\begin{equation}
  \label{eq3.43}
  \begin{array}{lcl}
    \left<\frac{\partial}{\partial x}\overline{u},
    \frac{\partial^{3}}{\partial x^{3}}\overline{u}\right>
    =-\left<\left(
    \frac{\partial^{2}}{\partial x^{2}} \overline{u}\right)^{2}\right>
  \end{array}
\end{equation}


\noindent
Using the relationship for the time derivatives,
\begin{equation}
  \label{eq3.45}
  \left\{ \begin{array}{lcl}
    \frac{dp}{dt}=\dot{p}
    =\delta\frac{\partial}{\partial T_{1}} p
    +\delta^{2}\frac{\partial}{\partial T_{2}} p
    +\delta^{3}\frac{\partial}{\partial T_{3}} p\\
    \frac{d\alpha}{dt}=\dot{\alpha}
    =\delta\frac{\partial}{\partial T_{1}} \alpha
    +\delta^{2}\frac{\partial}{\partial T_{2}} \alpha
  \end{array}\right.
\end{equation}
we have
\begin{equation}
  \label{eq3.46}
  \left\{ \begin{array}{lcl}
    \dot{p}=\delta\kappa_{3}\alpha+\delta^3C_3\alpha
    -\delta^{2}\frac{\left<\frac{\partial}{\partial x} \overline{u},
      \widetilde{\kappa}_{1}\right>}{
      \left<(\frac{\partial}{\partial x} \overline{u})^{2}\right>}\\
    \dot{\alpha}=\delta^{2}\widetilde{\tau}\kappa_{3}^{2}\alpha
    -\delta^{2}\kappa_{3}\alpha^{3}\frac{\left<(\frac{\partial^{2}}{
        \partial x^{2}} \overline{u})^{2}\right>}{\left<(\frac{\partial}{
        \partial x} \overline{u})^{2}\right>}-\delta\frac{\left
      <\frac{\partial}{
        \partial x} \overline{u}, \widetilde{\kappa}_{1}\right>}{
      \left<(\frac{\partial}{\partial x} \overline{u})^{2}\right>}
  \end{array}\right.
\end{equation}

The final step is replace $\alpha\delta$ by $\alpha$ in Eq.
(\ref{eq3.46}), and recovering original variables, 
$\widetilde{\tau}=\frac{\tau-1/\kappa_{3}}{\delta^{2}}$,
$\widehat{\kappa}_{1}(x)=\delta^{2}\widetilde{\kappa}_{1}(x)$,
and ignoring the term $\delta^2C_3$ because it is small,
then we can get the reduced 2-dim ODEs.

\subsection{Drift bifurcation point $\tau_c$}
We prove the existence of drift bifurcation point for the original three-component system (1), namely for the case $D_{v}\neq0$ and $\theta\neq0$, which leads to the conclusion after taking the limit of $D_{v}\rightarrow0$ and $\theta\rightarrow0$. Linearize the three component reaction diffusion system around
stationary pulse solution
$\overline{U}=(\overline{u}(x), \overline{v}(x), \overline{w}(x))$,
then we have the following linearized eigenvalue problem for $\hat U=(\hat u,\hat v, \hat w)$.

\begin{equation}
  \label{linearized3comp}
  \lambda\hat U=\mathcal D(\overline U;\tau, \theta)\hat U
\end{equation}
where $D(\overline U;\tau, \theta)$ can be decomposed into
\begin{equation}
  \mathcal{D}(\overline{U}; \tau, \theta)=\mathcal{M}(\tau, \theta)\mathcal{L}(\overline{U})
\end{equation}
where 

\begin{equation}
  \label{eq1.01appen}
  \mathcal{M}(\tau, \theta)=\left( \begin{array}{ccc} 1 & 0 & 0 \\
    0 & -\frac{1}{\kappa_{3}\tau} & 0 \\
    0 & 0 & -\frac{1}{\kappa_{4}\theta} \end{array} \right)
\end{equation}
and

\begin{equation}
  \label{3comp.linear.operator}
  \mathcal{L}(\overline{U})=\left(\begin{array}{ccc} D_u\Delta+\kappa_2-3\overline u^2 & -\kappa_3 & -\kappa_4 \\
    -\kappa_3 & -\kappa_{3}D_v\Delta+\kappa_3 & 0 \\
    -\kappa_4 & 0 & -\kappa_{4}D_w\Delta+\kappa_4 \end{array} \right)
\end{equation}
Translation invariance implies that there always exists a zero eigenvalue and the associated Goldstone mode $\mathcal{G}_{x}$ is given by
\begin{equation}
  \label{eq1.02appen}
  \mathcal{G}_{x}=\left( \begin{array}{c} \frac{\partial}{\partial x}
    \overline{u}(x) \\ \frac{\partial}{\partial x} \overline{v}(x) \\
    \frac{\partial}{\partial x} \overline{w}(x) \end{array} \right)
\end{equation}
Then the adjoined operator $\mathcal{D}^{*}(\overline{U}; \tau)$ admits the following Goldstone mode $\mathcal{G}^{*}_{x}$ in explicit form.
\begin{equation}
  \label{eq1.03appen}
  \mathcal{G}^{*}_{x}=\mathcal{M}^{-1}\mathcal{G}_{x}=\left( \begin{array}{c}
    \frac{\partial}{\partial x} \overline{u}(x) \\
    -(\kappa_{3}\tau)\frac{\partial}{\partial x} \overline{v}(x) \\
    -(\kappa_{4}\theta)\frac{\partial}{\partial x} \overline{w}(x)
  \end{array} \right)
\end{equation}
At the drift-bifurcation point $\tau=\tau_{c}$, the operator $\mathcal{D}$
has degenerate double zero eigenvalue similar to the discussion of
Section \ref{S_Dynamics_of_reduced_ODEs_reduction}:
$\mathcal{D}\mathcal{G}_{x}=0$, and
$\mathcal{D}\mathcal{P}_{x}=\mathcal{G}_{x}$.
The solvability condition at $\tau=\tau_{c}$ is given by
\begin{equation}
  \label{eq1.04appen}
  <\mathcal{G}^{*}_{x}, \mathcal{G}_{x}>=0
\end{equation}
Using Eqs.(\ref{eq1.02appen}) and (\ref{eq1.03appen}),
Eq.(\ref{eq1.04appen}) is equivalent to
\begin{equation*}
  \left<\left( \begin{array}{c} \frac{\partial}{\partial x}
    \overline{u}(x) \\ -(\kappa_{3}\tau)\frac{\partial}{\partial x}
    \overline{v}(x) \\ -(\kappa_{4}\theta)\frac{\partial}{\partial x}
    \overline{w}(x) \end{array} \right), \left( \begin{array}{c}
    \frac{\partial}{\partial x} \overline{u}(x) \\
    \frac{\partial}{\partial x} \overline{v}(x) \\
    \frac{\partial}{\partial x} \overline{w}(x) \end{array} \right)\right>=0.
  \quad\quad\quad\quad\quad\quad\quad\quad\quad
\end{equation*}
This yields
\begin{equation}
  \label{eq1.05appen}
  \tau_{c}=\frac{\left<
    (\frac{\partial}{\partial x}\overline{u}(x))^{2}\right>
    -\kappa_{4}\theta\left<
    (\frac{\partial}{\partial x}\overline{w}(x))^{2}\right>}{
    \kappa_{3}\left<
    (\frac{\partial}{\partial x}\overline{v}(x))^{2}\right>}.
\end{equation}
Now, we consider $D_{v}\rightarrow0$ and $\theta\rightarrow0$.
Noting that $\overline{v}(x)\equiv\overline{u}(x)$, Eq.(\ref{eq1.05appen})
becomes
\begin{equation}
  \label{eq1.06appen}
  \tau_{c}=\frac{1}{\kappa_{3}}
\end{equation}

\subsection{Proof of Proposition 3}
The proof is a direct consequence of the following three lemmas.
\begin{lemma}
  \label{prop6.01}
  Let $\tau_{c}<\tau<2\tau_{c}$, then the solution of Eq.(\ref{eq6.03})
  must fit $\alpha=0$, $f(p,d)=0$
\end{lemma}
\begin{proof}
  If $\alpha\neq0$, then
  \begin{equation}
    \label{eq6.04}
    \alpha=\frac{\varepsilon}{C_{1}\kappa_{3}}f(p,d)
  \end{equation}

  Substituting Eq.(\ref{eq6.04}) into the second equation of Eq.(\ref{eq6.03}),
  we have
  \begin{equation}
    \label{eq6.05}
    f(p,d)^{2}=(\kappa_{3}\tau-2)\frac{C_{1}^{3}
      \kappa_{3}^{2}}{C_{2}\varepsilon^{2}}
  \end{equation}

  By definition, we know that $C_{1}>0$, $C_{2}>0$, $\kappa_{3}^{2}>0$,
  $\varepsilon^{2}>0$.
  Furthermore, we assume that $\tau_{c}<\tau<2\tau_{c}$,
  then $\kappa_{3}\tau-2<0$.
  That means Eq.(\ref{eq6.05}) doesn't have solution.

  Then, the solution of Eq.(\ref{eq6.03}) must fit $\alpha=0$, $f(p,d)=0$

\end{proof}

\begin{lemma}
  \label{prop6.02}
  There exist countably many critical points for the reduced ODEs
  (\ref{eq3.50}).
\end{lemma}
\begin{proof}
  In view of Prop. (\ref{prop6.01}), it suffices to study the zero points of $f(p,d)=0$. Recall that $f(p,d)=\overline{u}(d/2-p)-\overline{u}(-d/2-p)$ and $\overline{u}$ is a stationary pulse solution with oscillatory tails with the asymptotic behavior:
  \begin{equation}
    \label{eq6.06}
    \overline{u}(x)\rightarrow \ope^{a|x-x_{0}|}cos(b|x-x_{0}|)+\underline{u},
    as \ \ x \rightarrow \pm \infty
  \end{equation}
  where $\lambda=a+bi, a < 0$ is the eigenvalue determined by the background state,
  $x_{0}$ is the appropriate phase for the stationary pulse and $\underline{u}$ is the background state. Then $f(p,d)$ should behave asymptotically for large $p$

  \[
  f(p,d) \approx
  \ope^{a|p-d/2|}\cos(b|p-d/2|)
  -\ope^{a|p+d/2|}\cos(b|p+d/2|),     \]
  and is transformed into
  \[
  f(p,d)=
  -\ope^{a|p|}\sqrt{A^2+B^2}\cos(b|p|-\phi),
  \]
  where
  \[
  A=2\cosh\frac{ad}{2}\sin{bd}{2},
  B=2\sinh\frac{ad}{2}\cos{bd}{2},
  \]
  and
  \[
  \sin\phi=\frac{A}{\sqrt{A^2+B^2}},\;\cos\phi=\frac{B}{\sqrt{A^2+B^2}},
  \]
  which implies that there exist countably many critical points for the reduced
  ODEs (\ref{eq3.50}).
\end{proof}

\begin{lemma}
  \label{prop6.03}
  For a fixed $\tau$ ($\tau_{c}<\tau<2\tau_{c}$) and a width $d$,  the critical point $(P_{i}, 0)$, $i\in \mathbb{Z}$ satisfies the following properties:
  
  \begin{itemize}
  \item $(P_{i}, 0)$ is saddle iff $\varepsilon D(P_{i},d)<0$.
    The saddle quantity (i.e., sum of the real eigenvalues) is always positive there:   
    $\sigma=\lambda_{1}+\lambda_{2}=b>0$.
    See Eq.(\ref{eq6.12}) for the definition of $b$.
  \item $(P_{i}, 0)$ is non-saddle if $\varepsilon D(P_{i},d)>0$.
    As $|\varepsilon|>0$ is increased, $(P_{i}, 0)$ undergoes the following sequential
    change: unstable node $\rightarrow$ unstable spiral $\rightarrow$
    stable spiral $\rightarrow$ stable node.
  \item For a fixed $\varepsilon$, there are only finitely many critical points that are stable.   
  \end{itemize}
\end{lemma}

\begin{proof}
  For simplicity, we only consider the case $\varepsilon >0$. 
  The Jacobian of Eq.(\ref{eq3.50}) at $(P_{i}, 0)$, $i\in \mathbb{Z}$ is 
  \begin{equation}
    \label{eq6.10}
    A=\left( \begin{array}{cc} -\frac{\varepsilon}{C_{1}}D(P_{i},d)
      & \kappa_{3} \\
      -\frac{\varepsilon}{C_{1}}D(P_{i},d) & \kappa_{3}^{2}\widehat{\tau}
    \end{array} \right).
  \end{equation}
  where $\widehat{\tau}=\tau-1/\kappa_{3}$ and
  \begin{equation}
    \label{eq6.09}
    \left.D(P_{i},d)=\frac{\partial f(p,d)}{\partial p}\right|_{P_{i}}.
  \end{equation}

  Then, the characteristic equation is,
  \begin{equation}
    \label{eq6.11}
    \lambda^{2}+(\frac{\varepsilon}{C_{1}}D(P_{i},d)
    -\kappa_{3}^{2}\widehat{\tau})\lambda+\frac{\varepsilon\kappa_{3}
      D(P_{i},d)}{C_{1}}(1-\kappa_{3}\widehat{\tau})=0.
  \end{equation}

  Let,
  \begin{equation}
    \label{eq6.12}
    \left\{ \begin{array}{lcl}
      b=\kappa_{3}^{2}\widehat{\tau}-\frac{1}{C_{1}}\widehat{\varepsilon}\\
      c=\frac{\kappa_{3}\widehat{\varepsilon}}{C_{1}}(1-\kappa_{3}\widehat{\tau})\\
    \end{array}\right.
  \end{equation}
  where $\widehat{\varepsilon} := \varepsilon D(P_{i},d)$.

  Then the determinant $\Delta=b^{2}-4c$ becomes
  \begin{equation}
    \label{eq6.13}
    \Delta=\frac{\widehat{\varepsilon}^{2}}{C_{1}^{2}}
    +\kappa_{3}^{4}\widehat{\tau}^{2}-\frac{2\kappa_{3}
      \widehat{\varepsilon}}{C_{1}}(2-\kappa_{3}\widehat{\tau})
  \end{equation}

  Note that $b$ and $\Delta$ depend implicitly on $i$ through $\widehat{\varepsilon}$. The eigenvalues of Eq.(\ref{eq6.11}) are completely characterized by the signs of $b$ and $\Delta$. Since $(1-\kappa_{3}\widehat{\tau}) > 0$, the sign of c is determined by that of $\widehat{\varepsilon}$ so that $(P_i, 0)$ is saddle iff $\widehat{\varepsilon} < 0$ for positive $\varepsilon$. For non-saddle case $\widehat{\varepsilon} > 0$, the straight line $b=0$ determines the stability of the critical point and the parabolic-like curve $\Delta=0$ separates the complex and real eigenvalues as depicted in Fig.\ref{stability_zeros.eps}. 
  It is easy to see that the curve $\Delta=0$ 
  takes a unique maximum point at $\widehat{\varepsilon}=C_{1}\kappa_{3}$ with value $\widehat{\tau}= 1/\kappa_{3}$ and crosses the zero level at $\widehat{\varepsilon}=4C_{1}\kappa_{3}$. Note that $b=0$ also passes the maximum point of $\Delta=0$ as in Fig.\ref{stability_zeros.eps}. It is easy to see from Fig.\ref{stability_zeros.eps} that any non-saddle point $(P_{i}, 0)$ experiences a sequential change for a fixed $\tau$: unstable node $\rightarrow$ unstable spiral $\rightarrow$ stable spiral $\rightarrow$ stable node when $\widehat{\varepsilon}$ crosses $\widehat{\varepsilon}=\widehat{\varepsilon}_{i} (i=1,2,3)$. Finally the finiteness of stable critical points comes from the exponential decaying property of $D(P_{i},d)$, namely $\widehat{\varepsilon} (= \varepsilon D(P_{i},d))$ crosses the line $b=0$ to the leftward for some finite $i$ when $i$ becomes large.
\end{proof}

\subsection{Proof of Proposition \ref{prop6.05}}
\begin{proof}
  We assume $\varepsilon > 0$ and $\alpha>0$ without loss of generality.
  When $\varepsilon=0$, it is clear that $\overline{\alpha}_{\pm}=\pm\sqrt{\frac{\kappa_{3}
      \widehat{\tau}C_{1}}{C_{2}}}$ are two attractive invariant manifolds as in Fig.\ref{illu010b.eps}. Recall that $\mathcal{O}_\mathrm{pulse}$ is $\overline{\alpha}_{+}$. What we have to show is that there is an invariant region sandwiched $\mathcal{O}_\mathrm{pulse}$ satisfying $\dot{p}>0$ in the closure of this region and the vector field transversally goes inside of the region at the boundary as shown in Fig.\ref{illu009.eps}. Let us estimate the effect of the heterogeneous term to the vector field. Stationary pulse solution $\overline{u}(x)$ is bounded so that $|\overline{u}(x)|<M_{0}$ holds for some positive constant $M_{0}$.  Then, by definition, we have
  \begin{equation}
    \label{eq6.150}
    |f(p,d)|<2M_{0}
  \end{equation}
  Therefore the signs of $\dot{\alpha}$ and $\dot{p}>0$ can be estimated as 

  \begin{equation}
    \label{eq6.170}
    \kappa_{3}^{2}\widehat{\tau}\alpha-\kappa_{3}\alpha^{3}\frac{C_{2}}{C_{1}}
    -\frac{\varepsilon}{C_{1}}2M_{0}<\dot{\alpha}<\kappa_{3}^{2}
    \widehat{\tau}\alpha-\kappa_{3}\alpha^{3}\frac{C_{2}}{C_{1}}
    +\frac{\varepsilon}{C_{1}}2M_{0}
  \end{equation}
  
  \begin{equation}
    \dot{p}=\kappa_{3}\alpha-\frac{\varepsilon}{C_{1}}f(p,d)>
    \kappa_{3}\alpha-\frac{\varepsilon}{C_{1}}2M_{0}
  \end{equation}
  
  \begin{figure}
    \centering
    \includegraphics[width=8cm]{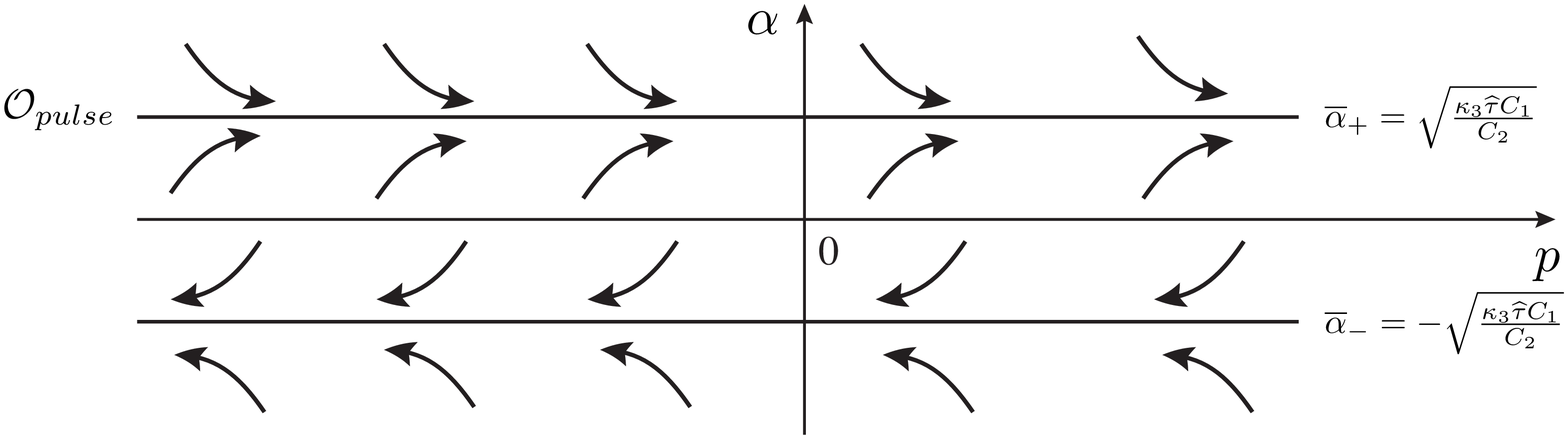}
    \caption{Flows in reduced ODEs without heterogeneity, Eq.(\ref{eq6.01}).
      $\overline{\alpha}_{\pm}=\pm\sqrt{\frac{\kappa_{3}
          \widehat{\tau}C_{1}}{C_{2}}}$ are two attractive invariant manifold.
      Pulse orbit $\mathcal{O}_\mathrm{pulse}$ is $\overline{\alpha}_{+}$.}
    \label{illu010b.eps}
  \end{figure}

  \begin{figure}
    \centering
    \includegraphics[width=8cm]{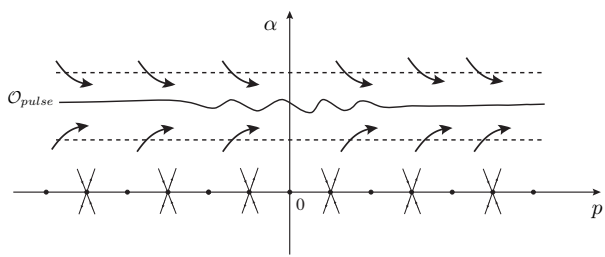}
    \caption{Flows in reduced ODEs with heterogeneity, Eq.(\ref{eq3.50}).
      Pulse orbit $\mathcal{O}_\mathrm{pulse}$ is distorted by heterogeneity term
      $\frac{\varepsilon}{C_{1}}f(p,d)$.
      Critical points are on $\alpha=0$.
      Saddle points and non-saddle points exist in pair on  $\alpha=0$ line.
      Only flows in $\alpha>0$ area is plotted.}
    \label{illu009.eps}
  \end{figure}

  We easily see that there is a positive constant $z_{0}$ independent of small $\varepsilon$ such that the band region
  $[\overline{\alpha}_{+}-z_{0}, \overline{\alpha}_{+}+z_{0}]$ 
  satisfies the transversality at the boundary and $\dot{p}>0$ holds, since the cubic nullcline of $\alpha$ is perturbed by $\frac{\varepsilon}{C_{1}}2M_{0}$. It follows from $\dot{p}>0$ that $\varepsilon$ should satisfy the inequality    
  \begin{equation}
    \label{eq6.162}
    0<\varepsilon<\frac{\kappa_{3}C_{1}(\overline{\alpha}_{+}-z_{0})}{2M_{0}}
  \end{equation}
  
  This concludes the proof.
\end{proof}

\bibliography{osc}{}
\end{document}